%% file: super.tex
\documentclass[a4paper,11pt,reqno]{amsart}

\usepackage[margin=1in,includehead,includefoot]{geometry}
\usepackage[utf8]{inputenc}
\usepackage[T1]{fontenc}
\usepackage{lmodern,microtype}
\usepackage[foot]{amsaddr}
\usepackage{upref}
\usepackage{mathtools,amssymb,amsthm}
\usepackage{enumitem}
\usepackage{graphicx}
\usepackage{booktabs}
\usepackage{hyperref}
\usepackage[b]{esvect} % for arrow symbols
\usepackage{arydshln}
\usepackage{multirow}
\usepackage{mdframed}
\usepackage{xpatch}
\usepackage[margin=1cm]{caption}
\usepackage{bookmark}
\usepackage[normalem]{ulem}
\usepackage{accents}
\usepackage{tikz}

\graphicspath{{./graphics/}}

\hypersetup{colorlinks=true, citecolor=darkblue, linkcolor=darkblue}

\newtheorem{theorem}{Theorem}
\newtheorem{corollary}[theorem]{Corollary}
\newtheorem{proposition}[theorem]{Proposition}
\newtheorem{lemma}[theorem]{Lemma}

\theoremstyle{definition}

\newtheorem{remark}[theorem]{Remark}

\newcommand{\R}{\mathbb{R}} % reals
\newcommand{\rvec}[1]{\ensuremath{\vv{\ensuremath{#1}}}}
\newcommand{\lvec}[1]{\reflectbox{\ensuremath{\vv{\reflectbox{\ensuremath{#1}}}}}}
\renewcommand{\c}[1]{\mathcal{#1}} % calligraphic letters
\newcommand{\h}{\widehat} % hat letters
\newcommand{\ol}{\overline}
\newcommand{\ul}{\underline}
\renewcommand{\b}{\bar}
\newcommand\ub[1]{\underaccent{\bar}{#1}}

\newcommand{\cH}{\c{H}}
\newcommand{\cL}{\c{L}}
\newcommand{\cO}{\c{O}}
\newcommand{\cR}{\c{R}}
\newcommand{\cT}{\c{T}}

\newcommand{\uL}{\mathrm{L}}
\newcommand{\uM}{\mathrm{M}}
\newcommand{\uR}{\mathrm{R}}
\newcommand{\pmn}{[\ol{n}]\cup [n]}
\newcommand{\seq}{\subseteq}
\newcommand{\hj}{{\widehat{\jmath}}}

\DeclareMathOperator{\rank}{rk} % lattice rank
\DeclareMathOperator{\codim}{codim} % codimension
\DeclareMathOperator{\id}{id}
\DeclareMathOperator{\vspan}{span}

\newcommand{\meet}{\wedge} % meet
\newcommand{\join}{\vee} % join

\definecolor{darkblue}{rgb}{0,0,0.7} % darkblue color
\definecolor{darkred}{rgb}{0.7,0,0} % darkred color
\newcommand{\darkblue}{\color{darkblue}} % darkblue command
\newcommand{\red}{\color{darkred}}

\newcommand{\defn}[1]{\textsl{\darkblue #1}} % colored definition
\newcommand{\OEIS}[1]{{\rm \href{http://oeis.org/#1}{#1}}}

\usepackage{todonotes}

\newenvironment{algo}[2]%
{{\bfseries #1} {\itshape (#2)}.}{}

\makeatletter
\xpatchcmd{\endmdframed}
  {\aftergroup\endmdf@trivlist\color@endgroup}
  {\endmdf@trivlist\color@endgroup\@doendpe}
  {}{}
\makeatother

\surroundwithmdframed[linecolor=black,linewidth=1pt,skipabove=2mm,skipbelow=2mm,leftmargin=0mm,rightmargin=0mm,innerleftmargin=2mm,innerrightmargin=2mm,innertopmargin=2mm,innerbottommargin=2mm]{algo}


\hypersetup{
  pdftitle={Combinatorial generation via permutation languages. VII. Supersolvable hyperplane arrangements},
  pdfauthor={Sofia Brenner, Jean Cardinal, Thomas McConville, Arturo Merino, Torsten M\"utze}
}

\title[Combinatorial generation via permutation languages]{Combinatorial generation via permutation languages. \\ VII. Supersolvable hyperplane arrangements}

\author{Sofia Brenner}
\address[Sofia Brenner]{Institut f\"ur Mathematik, Universit\"at Kassel, Germany}
\email{sbrenner@mathematik.uni-kassel.de}

\author{Jean Cardinal}
\address[Jean Cardinal]{Computer Science Department, Universit\'e Libre de Bruxelles (ULB), Belgium}
\email{jean.cardinal@ulb.be}

\author{Thomas McConville}
\address[Thomas McConville]{Department of Mathematics, Kennesaw State University, United States}
\email{tmcconvi@kennesaw.edu}

\author{\\Arturo Merino}
\address[Arturo Merino]{Universidad de O'Higgins, Rancagua, Chile}
\email{arturo.merino@uoh.cl}

\author{Torsten M\"utze}
\address[Torsten M\"utze]{Institut f\"ur Mathematik, Universit\"at Kassel, Germany}
\email{tmuetze@mathematik.uni-kassel.de}

\thanks{An extended abstract of this paper has been accepted for presentation at SODA~2026.}
\thanks{Sofia Brenner and Torsten M\"utze were supported by German Science Foundation grant~522790373.
Jean Cardinal was supported by the Fonds de la Recherche Scientifique-FNRS under grant n°~T003325F.
Arturo Merino was supported by ANID FONDECYT Iniciación No.~11251528.}

\begin{document}

\begin{abstract}
For an arrangement~$\cH$ of hyperplanes in~$\R^n$ through the origin, a \emph{region} is a connected subset of~$\R^n\setminus\cH$.
The \emph{graph of regions}~$G(\cH)$ has a vertex for every region, and an edge between any two vertices whose corresponding regions are separated by a single hyperplane from~$\cH$.
We aim to compute a Hamiltonian path or cycle in the graph~$G(\cH)$, i.e., a path or cycle that visits every vertex (=region) exactly once.
Our first main result is that if~$\cH$ is a supersolvable arrangement, then the graph of regions~$G(\cH)$ has a Hamiltonian cycle.
More generally, we consider quotients of lattice congruences of the poset of regions~$P(\cH,R_0)$, obtained by orienting the graph~$G(\cH)$ away from a particular base region~$R_0$.
Our second main result is that if~$\cH$ is supersolvable and~$R_0$ is a canonical base region, then for any lattice congruence~$\equiv$ on~$P(\cH,R_0)=:L$, the cover graph of the quotient lattice~$L/{\equiv}$ has a Hamiltonian path.

These paths and cycles are constructed by a generalization of the well-known Steinhaus-Johnson-Trotter algorithm for listing permutations.
This algorithm is a classical instance of a \emph{combinatorial Gray code}, i.e., an algorithm for generating a set of combinatorial objects by applying a small change in each step.
When applying our two main results to well-known supersolvable arrangements, such as the coordinate arrangement, and the braid arrangement and its subarrangements, we recover a number of known Gray code algorithms for listing various combinatorial objects, such as binary strings, binary trees, triangulations, rectangulations, acyclic orientations of graphs, congruence classes of quotients of the weak order on permutations, and of acyclic orientation lattices.
These were obtained earlier from the framework of zigzag languages of permutations proposed by Hartung, Hoang, M\"{u}tze, and Williams (\textit{Trans.\ Amer.\ Math.\ Soc.}, 2022).
When applying our main results to the type~$B$ Coxeter arrangement and its subarrangements, we obtain a number of new Gray code algorithms for listing (pattern-avoiding) signed permutations, symmetric triangulations, acyclic orientations of certain signed graphs, and in general for combinatorial families of Coxeter type~$B$, which generalizes the theory of zigzag languages to signed permutations.
Our approach also yields new Hamiltonicity results for large classes of polytopes, in particular signed graphic zonotopes and the type~$B$ quotientopes of Padrol, Pilaud, and Ritter (\textit{Int.\ Math.\ Res.\ Not.}, 2023).
\end{abstract}

\maketitle

\section{Introduction}\label{sec:introduction}

\subsection{Combinatorial generation}

Given a family of combinatorial objects, such as permutations, binary strings, binary trees, partitions etc., the \defn{combinatorial generation} problem asks to efficiently list all objects from the family, one after the other and each object exactly once.
This algorithmic problem is closely related to the problems of counting, random sampling, and optimization.
It is covered in depth in Knuth's seminal book `The Art of Computer Programming Vol.~4A'~\cite{MR3444818}.
In order for the generation algorithm to be efficient, it is often helpful to list the objects in a way such that any two consecutive objects differ only by a small, elementary change, sometimes called a \defn{flip}.
An example for this is the classical binary reflected Gray code for binary strings, in which all binary strings of a fixed length~$n$ are listed in such a way that any two consecutive strings differ only in a single bit~\cite{gray_1953}.
The term \defn{combinatorial Gray code} is used more broadly for any listing of combinatorial objects subject to an appropriately defined flip operation between successive objects~\cite{MR1491049,MR4649606}.
The flip operation also gives rise to a \defn{flip graph} on the set of objects, in which two objects are adjacent whenever they differ by a single flip.
The problem of finding a combinatorial Gray code amounts to finding a \defn{Hamiltonian path} in the corresponding flip graph, i.e., a path that visits every vertex exactly once.
A Gray code is \defn{cyclic} if the first and last object also differ in a flip, and this corresponds to a \defn{Hamiltonian cycle} in the flip graph.

\subsection{The Steinhaus-Johnson-Trotter algorithm for permutations}
\label{sec:SJT}

\begin{figure}[b!]
\begin{center}
\begin{tabular}{cc}
\raisebox{-\height}{\begin{tabular}{cc}
\toprule
$n$ & \\ \midrule
1 & 1 \\
2 & 1{\red 2}, {\red 2}1 \\
3 & 12{\red 3}, 1{\red 3}2, {\red 3}12, {\red 3}21, 2{\red 3}1, 21{\red 3} \\
4 & 123{\red 4}, 12{\red 4}3, 1{\red 4}23, {\red 4}123, \\
& {\red 4}132, 1{\red 4}32, 13{\red 4}2, 132{\red 4}, \\
& 312{\red 4}, 31{\red 4}2, 3{\red 4}12, {\red 4}312, \\
& {\red 4}321, 3{\red 4}21, 32{\red 4}1, 321{\red 4}, \\
& 231{\red 4}, 23{\red 4}1, 2{\red 4}31, {\red 4}231, \\
& {\red 4}213, 2{\red 4}13, 21{\red 4}3, 213{\red 4}\hspace{3pt} \\
\bottomrule
\end{tabular}} & \hspace{5mm}
\raisebox{-\height}{\input{graphics/permAcycle}} \\ & \\
(a) & (b)
\end{tabular}
\end{center}
\caption{(a)~Steinhaus-Johnson-Trotter listings of permutations for $n=1,\ldots,4$, with the largest element~$n$ highlighted; (b)~visualization of the listing for $n=4$ as a Hamiltonian cycle in the permutahedron.}
\label{fig:permA4cycle}
\end{figure}

An early, and now classical example of a cyclic Gray code is the so-called Steinhaus-Johnson-Trotter algorithm~\cite{MR157881,MR159764,DBLP:journals/cacm/Trotter62} for listing all permutations of $[n]:=\{1,\ldots,n\}$ by \defn{adjacent transpositions}, i.e., any two consecutive permutations differ by an exchange of two entries at adjacent positions.
The method is also known as `plain changes' and has been used since a long time by bell ringers (see~\cite{MR1416624}).
An easy way to describe this listing of permutations is by the following greedy algorithm due to Williams~\cite{MR3126386}:
Start with the identity permutation, and then repeatedly apply an adjacent transposition to the last permutation in the list that involves the largest possible value so as to create a new permutation, which then gets added to the list.
Equivalently, the construction can be described by induction on~$n$.
Given the listing~$L_{n-1}$ for permutations of length~$n-1$, the listing~$L_n$ for permutations of length~$n$ is obtained by replacing every permutation~$\pi$ in~$L_{n-1}$ by $n$ permutations of length~$n$ obtained by inserting the largest value~$n$ in all possible positions in~$\pi$.
Specifically, the value~$n$ moves alternatingly from right to left and from left to right in the successive permutations of~$L_{n-1}$, in a \defn{zigzag} pattern; see Figure~\ref{fig:permA4cycle}\,(a).
This algorithm can be implemented \defn{looplessly}, i.e., so that each permutation is generated in constant time, while using only linear memory to store the current permutation and a few auxiliary arrays.

The corresponding flip graph for this problem, namely the Cayley graph of the symmetric group generated by adjacent transpositions, is the well-known \defn{permutahedron}, and the Steinhaus-Johnson-Trotter algorithm yields a Hamiltonian cycle in this graph, i.e., the first and last permutation also differ in an adjacent transposition; see Figure~\ref{fig:permA4cycle}\,(b).

\subsection{Zigzag languages of permutations}
\label{sec:zigzag}

In a recent series of papers, M\"utze and coauthors showed that the simple greedy algorithm to generate the Steinhaus-Johnson-Trotter listing can be generalized to efficiently generate many other families of combinatorial objects that are in bijection to certain subsets of permutations.
Specifically, a \defn{zigzag language} is a subset of permutations of~$[n]$ that satisfies a certain closure property (see~\cite{MR4391718} for details), and via those bijections the flip operations on the permutations from the zigzag language translate to natural flip operations on the combinatorial objects.
The algorithm to generate a zigzag language presented in~\cite{MR4391718} in each step greedily moves the largest possible value in the permutation over a number of neighboring smaller entries, so as to create a new permutation in the zigzag language, which then gets added to the list.
The zigzagging framework applies to generate pattern-avoiding permutations~\cite{MR4391718}, lattice quotients of the weak order on permutations~\cite{MR4344032}, pattern-avoiding rectangulations~\cite{MR4598046}, elimination trees of chordal graphs~\cite{DBLP:journals/talg/CardinalMM25}, acyclic orientations of chordal graphs and hypergraphs and quotients of acyclic reorientation lattices~\cite{MR4614413}, and pattern-avoiding binary trees~\cite{MR4775168}.

The main contribution of this paper is to treat many of the aforementioned results under a common umbrella, namely that of hyperplane arrangements.
This yields an abstract unifying view, yet provides simpler and shorter proofs, and at the same time yields more general results, including new Hamiltonicity results for interesting flip graphs and polytopes, and concrete new Gray code algorithms for a variety of combinatorial objects.
The following sections provide three complementary views on the families of objects we aim to generate.
Much of the terminology and definitions are spelled out in more detail in Section~\ref{sec:background}.

\subsection{Hyperplane arrangements and graph of regions}
\label{sec:arrangements}

The geometric perspective that we will use is that of \defn{hyperplane arrangements}, defined as a nonempty finite set of real hyperplanes through the origin.
Given such an arrangement $\cH$ in $\R^n$, the set $\cR=\cR(\cH):=\R^n\setminus\cH$ is a collection of open convex sets called \defn{regions}.
The \defn{graph of regions}~$G(\cH)$ has $\cR(\cH)$ as its vertex set, and an edge between any two vertices corresponding to regions that are separated by exactly one hyperplane from~$\cH$; see Figure~\ref{fig:cube3}.

\begin{figure}[b!]
\makebox[0cm]{ % artificial box to center the picture
\begin{tabular}{ccc}
\includegraphics[page=9]{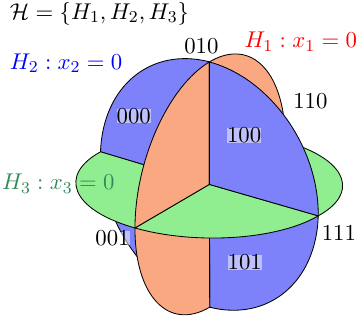} &
\includegraphics[page=10]{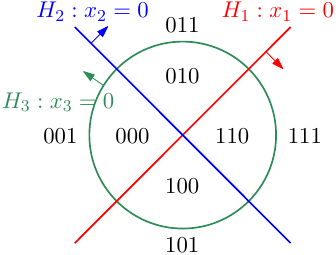} &
\input{graphics/cube} \\
(a) & (b) & (c) \\
\end{tabular}
}
\caption{(a) The coordinate arrangement for $n=3$ with regions labeled by binary strings; (b)~its stereographic projection from the south pole with the normal vectors for each hyperplane; (c)~its graph of regions, realized as a polytope, namely the 3-dimensional hypercube.}
\label{fig:cube3}
\end{figure}

Our goal is to construct a Hamiltonian path or cycle in~$G(\cH)$, i.e., we aim to visit all regions, one after the other and each region exactly once.
This problem generalizes a number of combinatorial generation problems and their corresponding flip graphs:
\begin{itemize}[leftmargin=5mm]
\item The \defn{coordinate arrangement} is the set of hyperplanes~$\cH$ in~$\R^n$ whose normal vectors are the canonical base vectors $\{\rvec{e_i} \mid i\in [n]\}$.
The regions of this arrangement~$\cH$ are in bijection with binary strings of length~$n$, and in the graph of regions~$G(\cH)$, any two regions differing in a single bit are joined by an edge, i.e., $G(\cH)$ is the graph of the $n$-dimensional hypercube; see Figure~\ref{fig:cube3}.
\item The \defn{braid arrangement}, also known as the \defn{type~$A$ Coxeter arrangement}, is the set of hyperplanes~$\cH$ in~$\R^n$ defined by the normal vectors $\{\rvec{e_i}-\rvec{e_j} \mid 1\leq i<j\leq n\}$.
The regions of this arrangement~$\cH$ are in bijection with permutations of~$[n]$, and in~$G(\cH)$, any two regions differing in an adjacent transposition are joined by an edge, i.e., $G(\cH)$ is the graph of the $(n-1)$-dimensional permutahedron; see Figure~\ref{fig:permA4}.
\item The \defn{type~$B$ Coxeter arrangement} is the set of hyperplanes~$\cH$ in~$\R^n$ defined by the normal vectors $\{\rvec{e_i}\pm \rvec{e_j} \mid 1\leq i<j\leq n\}\cup\{\rvec{e}_i\mid i\in[n]\}$.
The regions of this arrangement~$\cH$ are in bijection with signed permutations of~$[n]$.
A \defn{signed permutation} is a permutation of~$[n]$ in which every entry has a positive or negative sign.
In~$G(\cH)$, any two regions differing either in an adjacent transposition or a sign change of the first entry are joined by an edge, i.e., $G(\cH)$ is the graph of the $n$-dimensional $B$-permutahedron; see Figure~\ref{fig:permB3}.
\item Given a simple graph~$F=([n],E)$, the \defn{graphic arrangement} of~$F$ is the subarrangement of the braid arrangement consisting only of the hyperplanes with normal vectors~$\{\rvec{e_i} - \rvec{e_j}\mid \{i,j\}\in E\}$.
The regions of this arrangement~$\cH$ are in bijection with acyclic orientations of~$F$, and in in~$G(\cH)$, any two regions differing in reversing a single arc in the acyclic orientation of~$F$ are joined by an edge.
\end{itemize}

\begin{figure}[b!]
\makebox[0cm]{ % artificial box to center the picture
\begin{tabular}{ccc}
\includegraphics[page=1,scale=0.8]{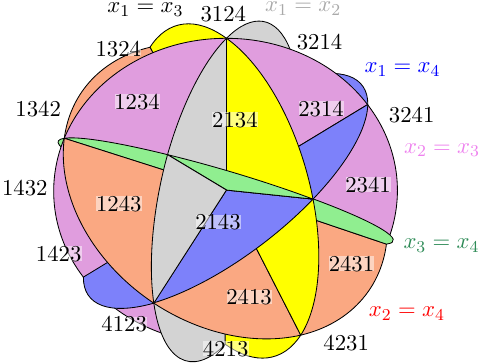} &
\includegraphics[page=3,scale=0.8]{hyper} &
\input{graphics/permA} \\
(a) & (b) & (c) \\
\end{tabular}
}
\caption{(a) The braid arrangement (=type~$A$ Coxeter arrangement) for $n=4$ projected to three dimensions with the regions labeled by permutations; (b)~its stereographic projection; (c)~its graph of regions, realized as a polytope, namely the 3-dimensional permutahedron of type~$A$.}
\label{fig:permA4}
\end{figure}

\begin{figure}[t!]
\makebox[0cm]{ % artificial box to center the picture
\begin{tabular}{ccc}
\includegraphics[page=5,scale=0.8]{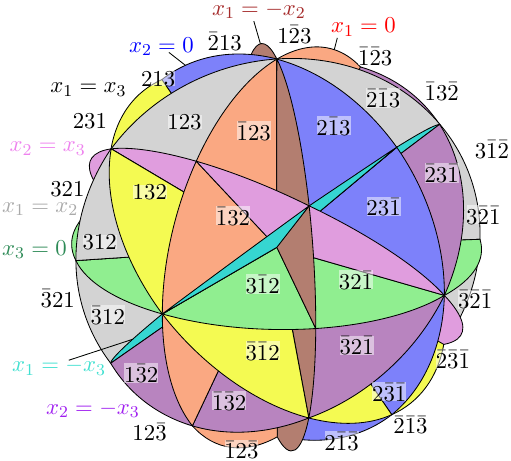} &
\includegraphics[page=7,scale=0.8]{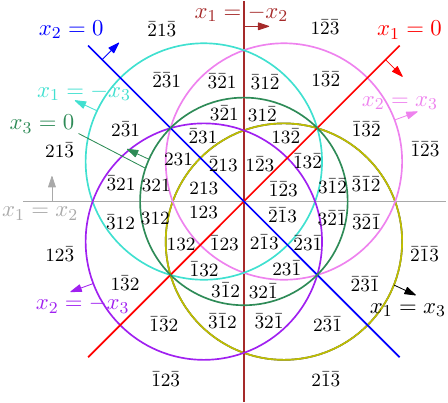} &
\input{graphics/permB} \\
(a) & (b) & (c) \\
\end{tabular}
}
\caption{(a) The type~$B$ Coxeter arrangement for $n=3$ with the regions labeled by signed permutations (barred entries have a negative sign); (b)~its stereographic projection; (c)~its graph of regions, realized as a polytope, namely the 3-dimensional $B$-permutahedron.}
\label{fig:permB3}
\end{figure}

An interesting special case of this problem is when the arrangement is that of a reflection group.
In that case, it was shown by Conway, Sloane, and Wilks~\cite{MR1032382} that the graph of regions is always Hamiltonian, hence that a Gray code always exists.

We conclude this section with some easy observations about the graph~$G(\cH)$.
First of all, the number of regions~$|\cR(\cH)|$ is always even, i.e., $G(\cH)$ has an even number of vertices.
This follows from the observation that the opposition map $x\mapsto -x$ on~$\R^n$ induces an involution without fixed points on the set of regions~$\cR(\cH)$.
Furthermore, the graph~$G(\cH)$ is bipartite, which can be shown by induction on the number of hyperplanes in~$\cH$.
If the number of hyperplanes in the arrangement~$\cH$ is odd, then the two partition classes of~$G(\cH)$ have the same size (as the opposition map switches classes in this case).
On the other hand, if the number of hyperplanes is even then the two partition classes of~$G(\cH)$ can have different sizes, which rules out the existence of a Hamiltonian cycle, and if the size difference is more than~1 then it also rules out the existence of a Hamiltonian path; see Figure~\ref{fig:nonham}~(a).
Even if the partition classes are balanced, a Hamiltonian path is not guaranteed; see Figure~\ref{fig:nonham}~(b).

\begin{figure}
\includegraphics[page=11]{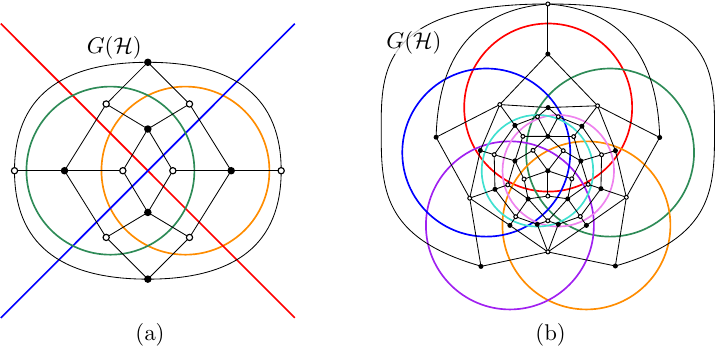}
\caption{Stereographic projections of two hyperplane arrangements~$\cH$ in~$\R^3$ for which~$G(\cH)$ does not admit a Hamiltonian path (nor cycle): (a) $\cH$ consists of 4 hyperplanes and the two partition classes of~$G(\cH)$ have sizes~6 (black vertices) and 8 (white vertices); (b) $\cH$ consists of 7 hyperplanes and both partition classes of~$G(\cH)$ have size~22.}
\label{fig:nonham}
\end{figure}

\subsection{Posets of regions and lattice congruences}

By choosing one region~$R_0\in\cR(\cH)$ as \defn{base region}, one can orient the edges of the graph of regions~$G(\cH)$ away from~$R_0$.
Because $G(\cH)$ is bipartite, it becomes the cover graph of a graded poset~$P(\cH,R_0)$, called the \defn{poset of regions}.
For example, if $\cH$ is the aforementioned type~$A$ Coxeter arrangement, and $R_0$ is the identity permutation, then $P(\cH,R_0)$ is the \defn{weak order of type~$A$}, i.e., permutations ordered by their inversion sets; see Figure~\ref{fig:congA}\,(a).
Similarly, for the type~$B$ Coxeter arrangement and $R_0$ the identity permutation, $P(\cH,R_0)$ is the \defn{weak order of type~$B$}, i.e., signed permutations ordered by their inversion sets; see Figure~\ref{fig:congB}\,(a).

In these two cases, $P(\cH,R_0)$ has the additional structure of a \defn{lattice}, i.e., for any two elements~$X,Y\in\cR(\cH)$, there is a unique largest element below~$X$ and~$Y$, called the \defn{meet~$X\meet Y$}, and a unique smallest element above~$X$ and~$Y$, called the \defn{join~$X\join Y$}.
If we have a lattice on a ground set~$P$, a \defn{lattice congruence} is an equivalence relation~$\equiv$ on~$P$ that is compatible with the meet and join operations, i.e., we require that if $X\equiv X'$ and $Y\equiv Y'$ then $X\meet Y\equiv X'\meet Y'$ and $X\join Y\equiv X'\join Y'$.
The \defn{quotient lattice} is the lattice on~$P/{\equiv}$, i.e., on the equivalence classes on~$P$ formed by~$\equiv$, with the order relation inherited from~$P$.
Consequently, the cover graph of~$P/{\equiv}$ is obtained from the cover graph of~$P$ by contracting the vertices in each equivalence class to a single vertex.

\begin{figure}
\begin{tabular}{cc}
\includegraphics[page=2]{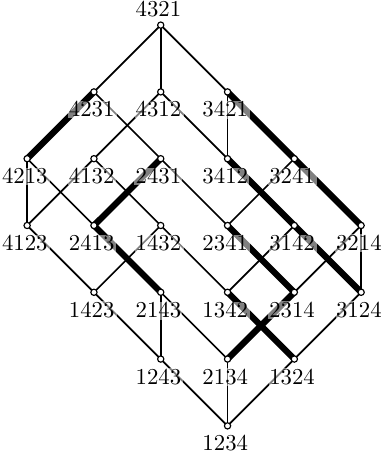} &
\includegraphics[page=3]{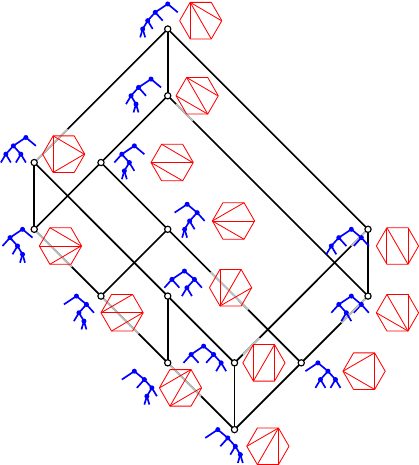} \\
(a) & (b)
\end{tabular}
\caption{
(a) The weak order of type~$A$ for $n=4$, obtained by orienting the permutahedron shown in Figure~\ref{fig:permA4}\,(c) away from the base region~$R_0=1234$.
The bold lines connect pairs of permutations that are in the same equivalence class of the sylvester congruence.
It is defined by the rewriting rule $\_b\_ca\_\equiv \_b\_ac\_$ where $a<b<c$, which means that whenever we see a subsequence of three entries $b$, $c$, $a$ with $a<b<c$ in a permutation, where $c$ and $a$ are at consecutive positions, then the permutation obtained by transposing $c$ and~$a$ is in the same equivalence class.
This corresponds to destroying one occurrence of the pattern 231.
Hence the 231-avoiding permutations are representatives of the equivalence classes and sit at the bottom of each class.
(b) Quotient lattice for the congruence on the left, namely the Tamari lattice with $n$-vertex binary trees (blue) or triangulations of a convex $(n+2)$-gon (red) corresponding to the 231-avoiding permutations.}
\label{fig:congA}
\end{figure}

\begin{figure}
\begin{tabular}{cc}
\includegraphics[page=2,scale=0.8]{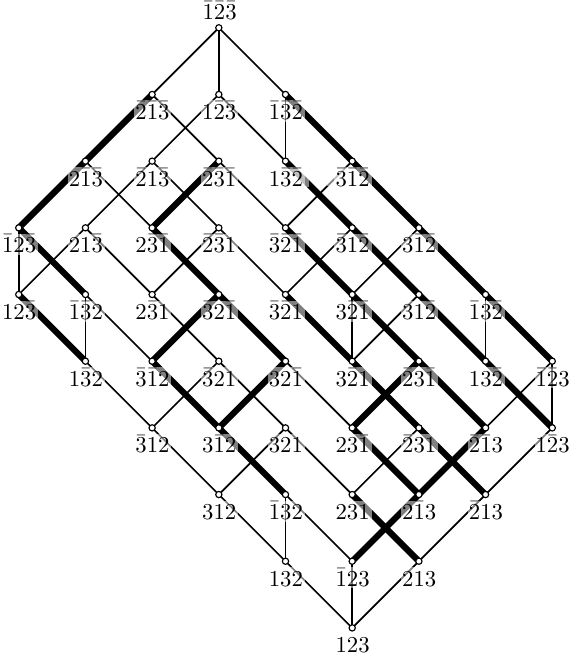} &
\includegraphics[page=3,scale=0.8]{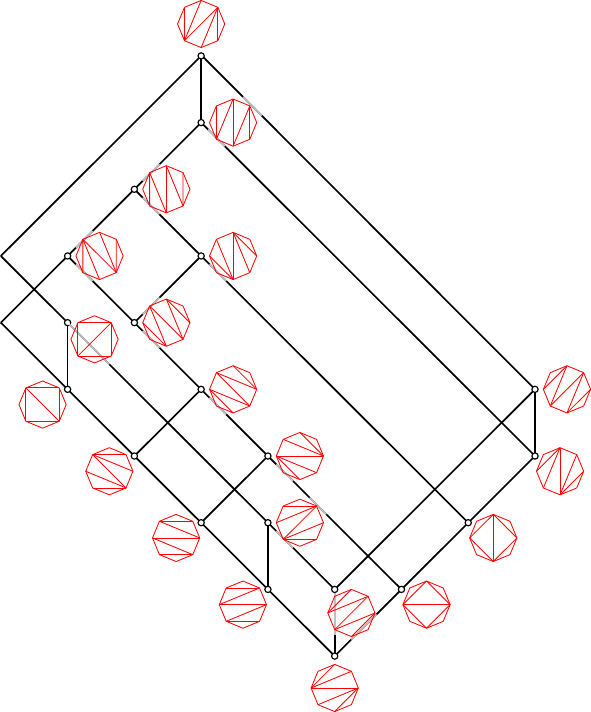} \\
(a) & (b)
\end{tabular}
\caption{(a)~The weak order of type~$B$ for $n=3$, obtained by orienting the $B$-permutahedron shown in Figure~\ref{fig:permB3}\,(c) away from the base region~$R_0=123$.
The bold lines connect pairs of signed permutations that are in the same equivalence class of the congruence defined by the rewriting rule $\_b\_ca\_\equiv \_b\_ac\_$ where $a<b<c$ and $b>0$, which has to be applied to the full notation of each signed permutation $x=(x_1,\ldots,x_n)$, i.e., to the string $(\ol{x_n},\ol{x_{n-1}},\ldots,\ol{x_1},x_1,x_2,\ldots,x_n)$.
The pattern-avoiding signed permutations sit at the bottom of each equivalence class.
(b)~Quotient lattice, namely the type~$B$ Tamari lattice for the congruence on the left, with the point-symmetric triangulations of a convex $(2n+2)$-gon corresponding to the pattern-avoiding signed permutations (the correspondence is given by Theorem~\ref{thm:h-bij}).}
\label{fig:congB}
\end{figure}

For the weak order of type~$A$, there is a large number of distinct lattice congruences, and for several of them, the corresponding congruence classes encode interesting combinatorial objects (under suitable bijections), and the cover graphs of the resulting quotient lattices are interesting flip graphs.
For example, the well-known \defn{Tamari lattice} is the quotient of the weak order of type~$A$ under the so-called sylvester congruence~\cite{MR2142078,MR3235205}.
Its elements can be identified with 231-avoiding permutations, and hence they are in bijection with binary trees and triangulations of a convex polygon, all of which are Catalan families; see Figure~\ref{fig:congA}.
In the cover graph of the quotient lattice, edges are precisely between binary trees that differ in a tree rotation, i.e., we obtain the rotation graph of binary trees~\cite{MR928904,MR3197650}.
In the corresponding triangulations, a tree rotation translates to a flip that removes one edge between two triangles and replaces it by the other diagonal in the resulting empty quadrilateral.

Different families of rectangulations can also be obtained as quotients of the weak order of type~$A$~\cite{MR2914637,MR2871762,MR2864445,MR3878132}, and the flip operations are tree rotations on the dual twin binary trees.
The \defn{Boolean lattice}, with the hypercube as its cover graph, is also a lattice quotient of the weak order of type~$A$.
In fact, all lattice congruences of the weak order form a lattice themselves, with the order relation being refinement of the equivalence classes.

If $\cH$ is the graphic arrangement of a graph~$F$ and $R_0$ a fixed acyclic orientation of~$F$, then $P(\cH,R_0)$ is the reorientation poset of all acyclic orientations of~$F$ with respect to the reference orientation~$R_0$.
Pilaud~\cite{MR4827881} gave a necessary and sufficient condition on~$R_0$ for this poset to be a lattice, and he also described the corresponding lattice congruences and quotients.

Reading~\cite{MR2258260} defined a type~$B$ Tamari lattice via a lattice congruence of the weak order of type~$B$.
In Figure~\ref{fig:congB}, we show the quotient lattice with the corresponding combinatorial objects, namely point-symmetric triangulations of a convex polygon.
Edges in the cover graph are between triangulations that differ in a flip of one or two symmetric edges.

\subsection{Polytopes}
\label{sec:poly}

Finally, the aforementioned flip graphs all arise as the \defn{skeletons} of well-known combinatorial \defn{polytopes}.
The \defn{zonotope} of an arrangement of hyperplanes~$\cH$ is the polytope obtained as the Minkowski sum of line segments in the directions of the normal vectors of the hyperplanes in~$\cH$.
Conversely, the \defn{normal fan} of the zonotope is the fan defined by the arrangement~$\cH$.

\begin{figure}[b!]
\makebox[0cm]{ % artificial box to center the picture
\begin{tabular}{ccc}
\includegraphics[page=2,scale=0.8]{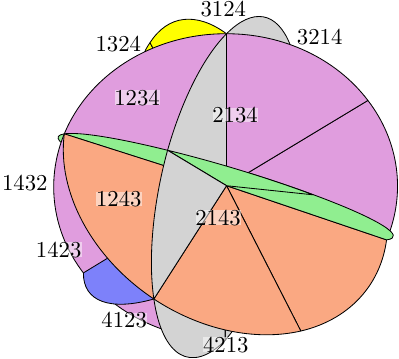} &
\includegraphics[page=4,scale=0.8]{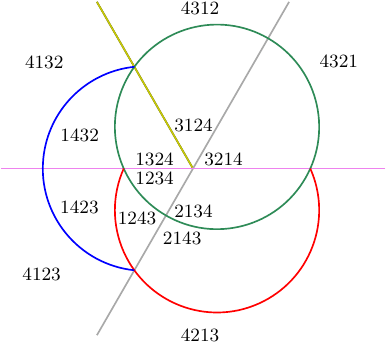} &
\input{graphics/type_a_asso_perm} \\
(a) & (b) & (c) \\
\end{tabular}
}
\caption{(a) Quotient fan obtained from the type~$A$ Coxeter fan in Figure~\ref{fig:permA4}\,(a) under the sylvester congruence of Figure~\ref{fig:congA}, and (b) its stereographic projection.
It is the normal fan of the associahedron shown in~(c), obtained by removing certain hyperplanes that bound the permutahedron inside.}
\label{fig:quotientopeA}
\end{figure}

\begin{figure}[b!]
\makebox[0cm]{ % artificial box to center the picture
\begin{tabular}{ccc}
\includegraphics[page=6,scale=0.8]{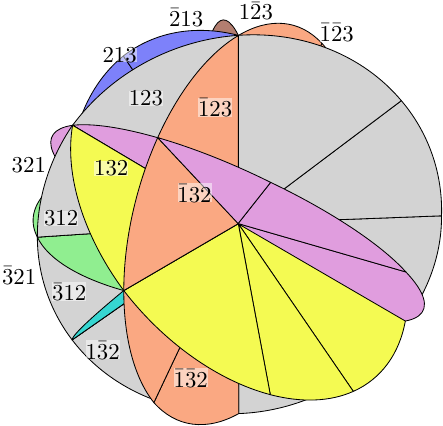} &
\includegraphics[page=8,scale=0.8]{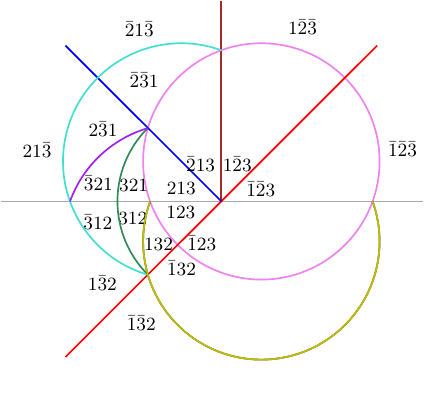} &
\input{graphics/type_b_asso_perm} \\
(a) & (b) & (c) \\
\end{tabular}
}
\caption{(a) Quotient fan obtained from the type~$B$ Coxeter fan in Figure~\ref{fig:permB3}\,(a) under the congruence of Figure~\ref{fig:congB}, and (b) its stereographic projection.
It is the normal fan of the $B$-associahedron shown in~(c), obtained by removing certain hyperplanes that bound the $B$-permutahedron inside.}
\label{fig:quotientopeB}
\end{figure}

The vertices of the zonotope are in bijection with the regions~$\cR(\cH)$ of the arrangement, and its skeleton is the graph of regions~$G(\cH)$.
For example, the hypercube is the zonotope of the coordinate arrangement; see Figure~\ref{fig:cube3}.
Similarly, the \defn{permutahedron} is the zonotope of the braid arrangement (see Figure~\ref{fig:permA4}), and the \defn{$B$-permutahedron} is the zonotope of the type~$B$ Coxeter arrangement (see Figure~\ref{fig:permB3}).
The \defn{graphic zonotope} is the zonotope of the graphic arrangement of a graph.

A lattice congruence of type~$A$ acts on the normal fan of the permutahedron by gluing together regions that belong to the same equivalence class of the congruence.
Pilaud and Santos~\cite{MR3964495} showed that the resulting \defn{quotient fan} is the normal fan of a polytope they called~\defn{quotientope}; see Figure~\ref{fig:quotientopeA}.
Put differently, their result shows that the cover graph of any lattice quotient of the weak order on permutations can be realized geometrically as the skeleton of a polytope.
Quotientopes thus generalize permutahedra, associahedra~\cite{MR2108555}, rectangulotopes~\cite{MR4833064}, hypercubes etc.
Padrol, Pilaud, and Ritter~\cite{MR4584712} later gave a construction of quotientopes via Minkowski sums.
Notably, their approach also works for all lattice congruences on the type~$B$ Coxeter arrangement, i.e., we obtain \defn{type~$B$ quotientopes}, including the \defn{$B$-associahedron}~\cite{MR1979780}; see Figure~\ref{fig:quotientopeB}.
We refer to the recent expository paper of Pilaud, Santos, and Ziegler~\cite{MR4675114} for context and references about associahedra and their generalizations.

Table~\ref{tab:families} gives an overview of the various structures discussed here and their connections.

\begin{table}[h!]
\caption{Correspondence between combinatorial objects encoded by different hyperplane arrangements and their lattice congruences, posets and polytopes.}	
\makebox[0cm]{ % artificial box to center the picture
\footnotesize
\begin{tabular}{llll}
\toprule
Arrangement & Combinatorial objects & Poset & Polytope \\
\midrule
coordinate & binary strings & Boolean lattice & hypercube \\
\midrule
type~$A$ & permutations & weak order of type $A$ & permutahedron of type~$A$ \\
         & binary trees, triangulations & Tamari lattice & associahedron \\
         & (diagonal/generic) rectangulations & lattice of rectangulations & rectangulotopes \cite{MR4833064} \\
         & permutrees & rotation lattice & permutreehedra \cite{MR3856522} \\
graphic  & acyclic orientations of graphs & acyclic reorientation order & graphic zonotope \\
\midrule
type~$B$ & signed permutations & weak order of type~$B$ & permutahedron of type~$B$ \\
         & symmetric triangulations & Tamari lattice of type~$B$ & associahedron of type~$B$ \\
         & acyclic orientations of signed graphs & acyclic reorientation order of type~$B$ & signed graphic zonotope \\
\bottomrule
\end{tabular}
}
\label{tab:families}
\end{table}

\subsection{Our results}

The zigzagging procedure of the Steinhaus-Johnson-Trotter algorithm and its generalizations~\cite{MR4391718} exploit a well-studied property of hyperplane arrangements called \defn{supersolvability}.
This term was first used in the context of group theory, as a property of the lattice of subgroups of a group.
Stanley~\cite{MR0309815} later gave a lattice-theoretic definition, which was then used to qualify hyperplane arrangements.
A convenient definition, in our context, is the following: A hyperplane arrangement~$\cH$ is \defn{supersolvable} if the set~$\cH$ can be partitioned into
two nonempty subsets~$\cH_0$ and~$\cH_1$, where $\cH_0$ is a supersolvable arrangement of lower rank, and for any pair of distinct hyperplanes in~$\cH_1$, their intersection is contained in a hyperplane of $\cH_0$~\cite{MR1036875}.
(We refer to Section~\ref{sec:supersolvable} for complete definitions.)
When a supersolvable arrangement~$\cH$ is split into the two subarrangements~$\cH_0$ and~$\cH_1$, then the regions of the arrangement~$\cH_0$ can be seen as equivalence classes of the regions of~$\cH$, and the regions in each such equivalence class are linearly ordered, i.e., they form a path in the graph of regions~$G(\cH)$.
The zigzagging method alternatingly moves back and forth along those paths, following a Hamiltonian cycle in the graph of regions~$G(\cH_0)$ of the lower-rank arrangement~$\cH_0$ that is built inductively; see Figure~\ref{fig:zigzag}.
As the number of regions~$|\cR(\cH_0)|$ is always even, the last visited region is adjacent to the starting region, i.e., we obtain a Hamiltonian cycle in~$G(\cH)$.

\begin{figure}[h!]
\includegraphics[page=12]{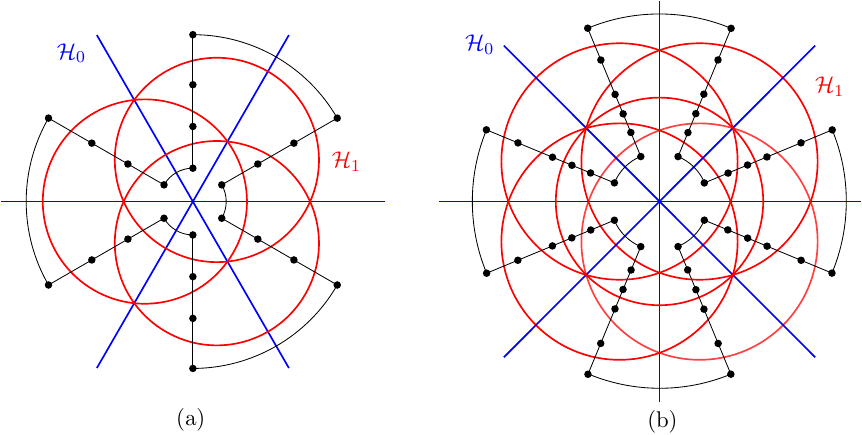}
\caption{Illustration of the supersolvable partition $\cH=\cH_0\cup\cH_1$ of the Coxeter arrangements of (a) type~$A$ and (b)~type~$B$ from Figures~\ref{fig:permA4} and~\ref{fig:permB3}, respectively (as stereographic projections).
The resulting Hamiltonian cycles in the graph of regions obtained from the zigzagging procedure are shown in black, and they correspond to the listings given in Figures~\ref{fig:permA4cycle} and~\ref{fig:permB3cycle}, respectively.}
\label{fig:zigzag}
\end{figure}

This leads to our first result.\footnote{After submitting this paper, we learned that K{\"o}rber, Schnieders, Stricker, and Walizadeh~\cite{koerber_et_al_2025} found an independent proof of Theorem~\ref{thm:super-ham}.}

\begin{theorem}
\label{thm:super-ham}
Let~$\cH$ be a supersolvable hyperplane arrangement of rank~$n\geq 2$.
Then the graph of regions~$G(\cH)$ has a Hamiltonian cycle of even length.
\end{theorem}

Bj\"{o}rner, Edelman and Ziegler~\cite{MR1036875} showed that the regions of a supersolvable arrangement~$\cH$ can always be ordered into a lattice, for a suitable choice of base region~$R_0\in\cR(\cH)$.
Our second main result is that the cover graph of the quotient of any such lattice always has a Hamiltonian path, which can be found by the zigzagging procedure.

\begin{theorem}
\label{thm:super-quotient}
Let~$\cH$ be a supersolvable hyperplane arrangement, and let $L:=P(\cH,R_0)$ be its lattice of regions for a canonical base region~$R_0$.
Then for any lattice congruence $\equiv$ on~$L$, the cover graph of~$L/{\equiv}$ has a Hamiltonian path.
\end{theorem}

(We refer to Section~\ref{sec:regions} for the definition of canonical base region.)
The Hamiltonian paths and cycles in those results can be computed by the following simple greedy algorithm based on the supersolvable partition~$\cH=\cH_0\cup\cH_1$:
Start at a canonical base region, and then repeatedly apply the following rule where to move from the currently visited region:
If possible, traverse a hyperplane from~$\cH_1$ that leads to a previously unvisited region, otherwise recursively select a hyperplane to traverse from~$\cH_0$ according to this rule.

\subsection{Applications of our results}

With the two results above, we recover a number of previously known Hamiltonicity results and Gray code algorithms, and we provide several new ones.

\subsubsection{The coordinate arrangement and binary strings}

The coordinate arrangement~$\cH$ with normal vectors~$\{\rvec{e_i}\mid i\in[n]\}$ is clearly supersolvable.
Indeed, for the required partition~$\cH=\cH_0\cup\cH_1$ we can take $\cH_1$ as the singleton hyperplane with normal vector~$\rvec{e_n}$.
The resulting Hamiltonian cycle in the hypercube is the well-known \defn{binary reflected Gray code}; see Figure~\ref{fig:brgc}.

\begin{table}[h!]
\caption{The classical binary reflected Gray code for $n=1,\ldots,4$.}	
\begin{center}
\begin{tabular}{cc}
\toprule
$n$ & \\
\midrule
1 & 0, 1 \\
2 & 0{\red 0}, 0{\red 1}, 1{\red 1}, 1{\red 0} \\
3 & 00{\red 0}, 00{\red 1}, 01{\red 1}, 01{\red 0}, 11{\red 0}, 11{\red 1}, 10{\red 1}, 10{\red 0} \\
4 & 000{\red 0}, 000{\red 1}, 001{\red 1}, 001{\red 0}, 011{\red 0}, 011{\red 1}, 010{\red 1}, 010{\red 0}, \\
& 110{\red 0}, 110{\red 1}, 111{\red 1}, 111{\red 0}, 101{\red 0}, 101{\red 1}, 100{\red 1}, 100{\red 0}\;{} \\
\bottomrule
\end{tabular}
\end{center}
\label{fig:brgc}
\end{table}

It has the following greedy description (\cite{MR3126386}): Start with the all-0 string of length~$n$, and then repeatedly flip the rightmost bit in the last string so as to create a new binary string.

\subsubsection{The type~$A$ Coxeter arrangement and permutations}

Recall that the braid arrangement, or type~$A$ Coxeter arrangement, is defined by the normal vectors $\{\rvec{e_i}-\rvec{e_j}\mid 1\leq i<j\leq n\}$.
The supersolvability of this arrangement~$\cH$ is witnessed by letting $\cH_1$ be the set of hyperplanes with normal vectors whose $n$th component is nonzero.
Indeed, consider two of them, with normal vectors $\rvec{e_i}-\rvec{e_n}$ and $\rvec{e_j}-\rvec{e_n}$, respectively, with $i<j<n$.
Then their intersection is a flat of codimension~2 lying in the hyperplane with normal vector $\rvec{e_i}-\rvec{e_j}$, which belongs to $\cH_0$.
Theorem~\ref{thm:super-ham} yields the aforementioned Steinhaus-Johnson-Trotter listing of permutations (recall Section~\ref{sec:SJT}).

\subsubsection{The type~$B$ Coxeter arrangement and signed permutations}
\label{sec:sperm}

Recall that the type~$B$ Coxeter arrangement is defined by the hyperplanes with normal vectors $\{\rvec{e_i}\pm\rvec{e_j}\mid 1\leq i<j\leq n\}\cup \{\rvec{e_i}\mid i\in [n]\}$.
Similarly to before, the supersolvability of this arrangement~$\cH$ is witnessed by letting $\cH_1$ be the hyperplanes with normal vectors whose $n$th component is nonzero.

As mentioned before, the regions of~$\cH$ are in bijection with signed permutations, i.e., permutations of~$[n]$ in which every entry has a positive or negative sign.
Adjacencies in the graph~$G(\cH)$ are adjacent transpositions or a sign change of the first entry.
The Gray code for signed permutations obtained from Theorem~\ref{thm:super-ham} is shown in Figure~\ref{fig:permB3}, and it is cyclic for all~$n$.
Note that this Gray code is different from the Gray code given by Conway, Sloane, and Wilks~\cite{MR1032382} and also from the one given by Korsh, LaFollette, and Lipschutz~\cite{MR2841338}.

\begin{figure}[b!]
\begin{center}
\begin{tabular}{cc}
\raisebox{-\height}{\begin{tabular}{cc}
\toprule
$n$ & \\ \midrule
1 & $1, \b{1}$ \\
2 & $1{\red 2}, {\red 2}1, \b{{\red 2}}1, 1\b{{\red 2}}, \b{1}\b{{\red 2}}, \b{{\red 2}}\b{1}, {\red 2}\b{1}, \b{1}{\red 2}$ \\
3 & $12{\red 3}, 1{\red 3}2, {\red 3}12, \b{{\red 3}}12, 1\b{{\red 3}}2, 12\b{{\red 3}},$ \\
  & $21\b{{\red 3}}, 2\b{{\red 3}}1, \b{{\red 3}}21, {\red 3}21, 2{\red 3}1, 21{\red 3},$ \\
  & $\b{2}1{\red 3}, \b{2}{\red 3}1, {\red 3}\b{2}1, \b{{\red 3}}\b{2}1, \b{2}\b{{\red 3}}1, \b{2}1\b{{\red 3}},$ \\
  & $1\b{2}\b{{\red 3}}, 1\b{{\red 3}}\b{2}, \b{{\red 3}}1\b{2}, {\red 3}1\b{2}, 1{\red 3}\b{2}, 1\b{2}{\red 3},$ \\
  & $\b{1}\b{2}{\red 3}, \b{1}{\red 3}\b{2}, {\red 3}\b{1}\b{2}, \b{{\red 3}}\b{1}\b{2}, \b{1}\b{{\red 3}}\b{2}, \b{1}\b{2}\b{{\red 3}},$ \\
  & $\b{2}\b{1}\b{{\red 3}}, \b{2}\b{{\red 3}}\b{1}, \b{{\red 3}}\b{2}\b{1}, {\red 3}\b{2}\b{1}, \b{2}{\red 3}\b{1}, \b{2}\b{1}{\red 3},$ \\
  & $2\b{1}{\red 3}, 2{\red 3}\b{1}, {\red 3}2\b{1}, \b{{\red 3}}2\b{1}, 2\b{{\red 3}}\b{1}, 2\b{1}\b{{\red 3}},$ \\
  & $\b{1}2\b{{\red 3}}, \b{1}\b{{\red 3}}2, \b{{\red 3}}\b{1}2, {\red 3}\b{1}2, \b{1}{\red 3}2, \b{1}2{\red 3}\;{}$ \\
\bottomrule
\end{tabular}}
& \hspace{5mm}
\raisebox{-\height}{\input{graphics/permBcycle}} \\ & \\
(a) & (b)
\end{tabular}
\end{center}
\caption{(a)~Our new cyclic Gray code for signed permutations for $n=1,2,3$; (b)~visualization of the listing for $n=3$ as a Hamiltonian cycle on the $B$-permutahedron.}
\label{fig:permB3cycle}
\end{figure}

For any positively signed element~$i\in[n]$ we write~$\ol{i}:=-i$ for its negatively signed counterpart.
Note that $\ol{\ol{i}}=i$.
We define~$[\ol{n}]:=\{\ol{1},\ol{2},\ldots,\ol{n}\}$, and we denote a signed permutation as a string $\pi=(a_1,\ldots,a_n)$ such that~$a_i\in\pmn$ for $i=1,\ldots,n$ and $\{|a_1|,\ldots,|a_n|\}=[n]$.
It is useful in this context to introduce the so-called \defn{full notation} of a signed permutation~$\pi=(a_1,\ldots,a_n)$, which is the string~$\h{\pi}:=(\ol{a_n},\ol{a_{n-1}},\ldots,\ol{a_1},a_1,a_2,\ldots,a_n)$ of length~$2n$.
Note that an adjacent transposition in~$\pi$ corresponds to a symmetric pair of adjacent transpositions in~$\h{\pi}$.
Furthermore, a sign change of the first entry of~$\pi$ corresponds to an adjacent transposition of the middle two entries of~$\h{\pi}$.
Therefore, our Gray code for signed permutations can be described greedily, using the full notation, as follows:
Start with the identity permutation~$(\ol{n},\ol{n-1},\ldots,\ol{1},1,2,\ldots,n)$, and then repeatedly apply an adjacent transposition\footnote{possibly together with a forced symmetric transposition} to the last permutation in the list that involves the largest possible value so as to create a new permutation.
This algorithm can be implemented looplessly, i.e., in time $\cO(1)$ per generated signed permutation, while using only $\cO(n)$ memory.
We prepared an implementation in C++, available for download and experimentation on the Combinatorial Object Server website~\cite{cos_cperm}.

\subsubsection{Type~$A$ subarrangements and acyclic orientations of graphs}

Recall from Section~\ref{sec:arrangements} that the graphic arrangement~$\cH$ of a graph $F=([n],E)$ is a subarrangement of the braid arrangement, obtained by retaining only the hyperplanes corresponding to the edges of $F$, with normal vectors $\{\rvec{e_i} - \rvec{e_j} \mid \{i,j\}\in E\}$.
It is well-known and not hard to see that the regions of~$\cH$ are in bijection with acyclic orientations of the graph~$F$, and that an adjacency in~$G(\cH)$ corresponds to reversing a single arc in the acyclic orientation of~$F$.
Stanley~\cite[Prop.~2.8]{MR0309815} proved that the graphic arrangement of~$F$ is supersolvable if and only if $F$ is a \defn{chordal graph}, i.e., $F$ has no induced cycle of length~4 or more.
Applying Theorem~\ref{thm:super-ham}, we thus obtain the Gray codes for acyclic orientations of chordal graphs studied in~\cite{MR1267311} and~\cite{MR4614413}.

Via Theorem~\ref{thm:super-quotient} we also recover the Gray codes for quotients of acyclic reorientation lattices of chordal graphs, whose existence was proved by Cardinal, Hoang, Merino, Mi\v{c}ka, and M\"utze~\cite{MR4614413}, addressing a problem raised by Pilaud~\cite[Prob.~52]{MR4827881}.

\subsubsection{Type~$B$ subarrangements and acyclic orientations of signed graphs}

Zaslavsky~\cite{MR0676405} showed that subarrangements of the type~$B$ Coxeter arrangement can be interpreted as the type~$B$ analogue of graphic arrangements but for \defn{signed graphs}, in which every edge carries a sign.
The conditions for supersolvability of this arrangement are known; see~\cite{MR1808091,MR3886267}.
In particular, these conditions identify a class of signed graphs that are the signed analogues of chordal graphs, those having a \defn{signed perfect elimination ordering}.
We refer to Sections~\ref{sec:sao} and \ref{sec:supersolvable} for the definitions.
Theorem~\ref{thm:super-ham} directly yields new Gray codes for acyclic orientations of those signed graphs, and Hamiltonicity of the corresponding polytopes.

\begin{corollary}
For any signed graph that has a signed perfect elimination ordering, the corresponding signed graphic zonotope has a Hamiltonian cycle.
\end{corollary}

From Theorem~\ref{thm:super-quotient} we also directly obtain Gray codes for quotients of acyclic reorientation lattices of signed graphs, generalizing the result of~\cite{MR4614413}.

\subsubsection{Hamiltonian paths on type~$A$ quotientopes}

Recall from Section~\ref{sec:poly} that the quotientopes introduced by Pilaud and Santos are the polytopes whose skeletons are the cover graphs of quotients of lattice congruences of the weak order of type~$A$.
Hoang and M\"utze~\cite{MR4344032} proved that all quotientopes admit a Hamiltonian path, a result that can be recovered as a special case from our Theorem~\ref{thm:super-quotient}.
In particular, we recover a Hamiltonian path on the associahedron (via the sylvester congruence), which coincides with the Gray codes for binary trees by tree rotations due to Lucas, Roelants van Baronaigien, and Ruskey~\cite{MR1239499}.
Furthermore, we recover Hamiltonian cycles on rectangulotopes, which coincide with the cyclic Gray codes for diagonal and generic rectangulations due to Merino and M\"utze~\cite{MR4598046}.

\subsubsection{Hamiltonian paths on type~$B$ quotientopes}

As mentioned before, type~$B$ quotientopes were defined by Padrol, Pilaud, and Ritter~\cite{MR4584712} as the polytopes whose skeletons are the cover graphs of quotients of lattice congruences of the weak order of type~$B$.
Applying Theorem~\ref{thm:super-quotient} to the type~$B$ Coxeter arrangement, we thus obtain the following new result.

\begin{corollary}
\label{cor:quotientopeB}
The skeleton of any type~$B$ quotientope has a Hamiltonian path.
\end{corollary}

In particular, we obtain a Hamiltonian cycle on the $B$-associahedron~\cite{MR1979780}.
Via a suitable bijection, this gives a new cyclic Gray code for point-symmetric triangulations of a convex $(2n+2)$-gon, in which any two consecutive triangulations differ in one flip of the edge through the center or a pair of symmetric flips not involving the center-edge; see Figures~\ref{fig:triang} and~\ref{fig:gc-triang} (see Theorems~\ref{thm:h-bij} and~\ref{thm:2b31} below).

\begin{figure}[h!]
\includegraphics[page=1,scale=0.5]{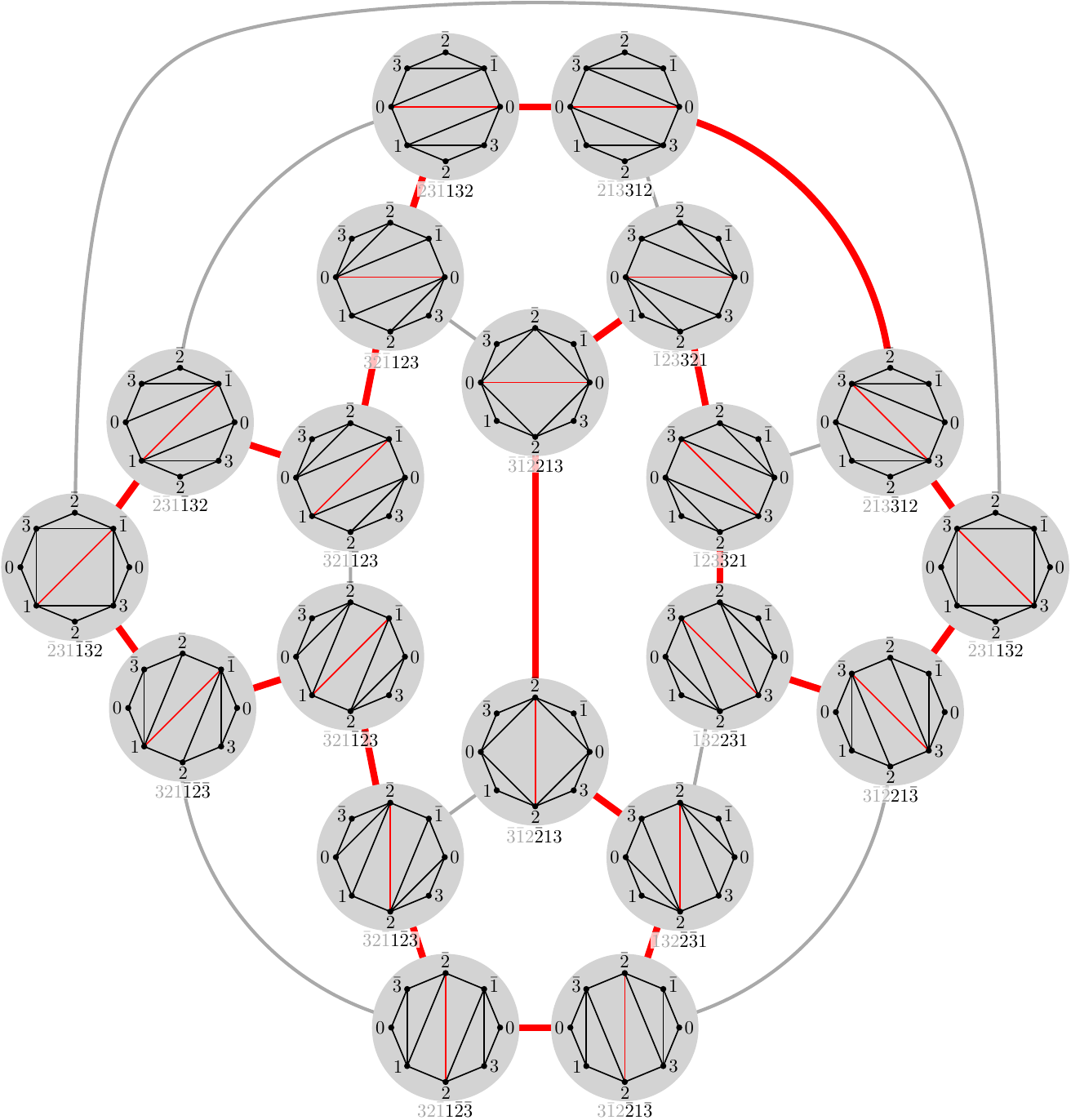}
\caption{The skeleton of the $B$-associahedron, i.e., the cover graph of the lattice quotient shown in Figure~\ref{fig:congB}, with the pattern-avoiding permutations in full notation.
The Hamiltonian cycle corresponding to the Gray code in Figure~\ref{fig:gc-triang} for $n=3$ is highlighted.}
\label{fig:triang}
\end{figure}

\begin{figure}
\makebox[0cm]{ % artificial box to center the picture
\begin{tabular}{c}
\includegraphics[page=2,scale=0.5]{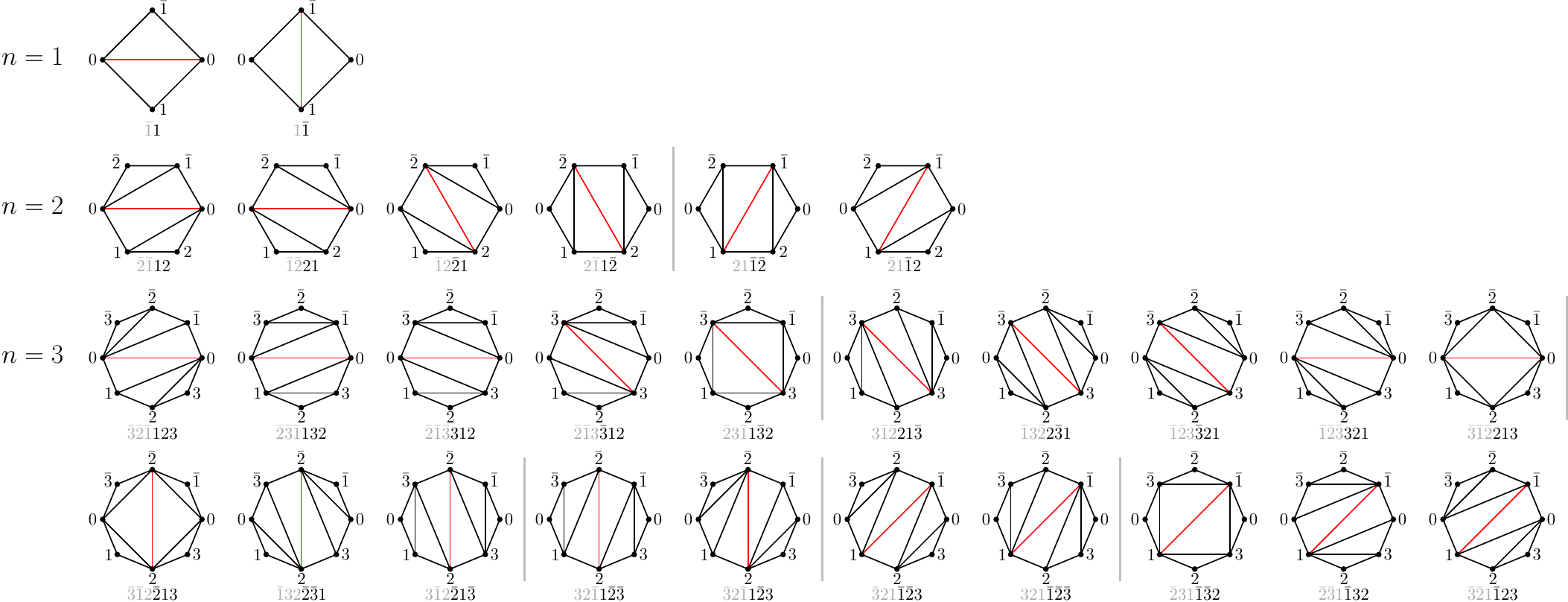} \\
\includegraphics[page=3,scale=0.5]{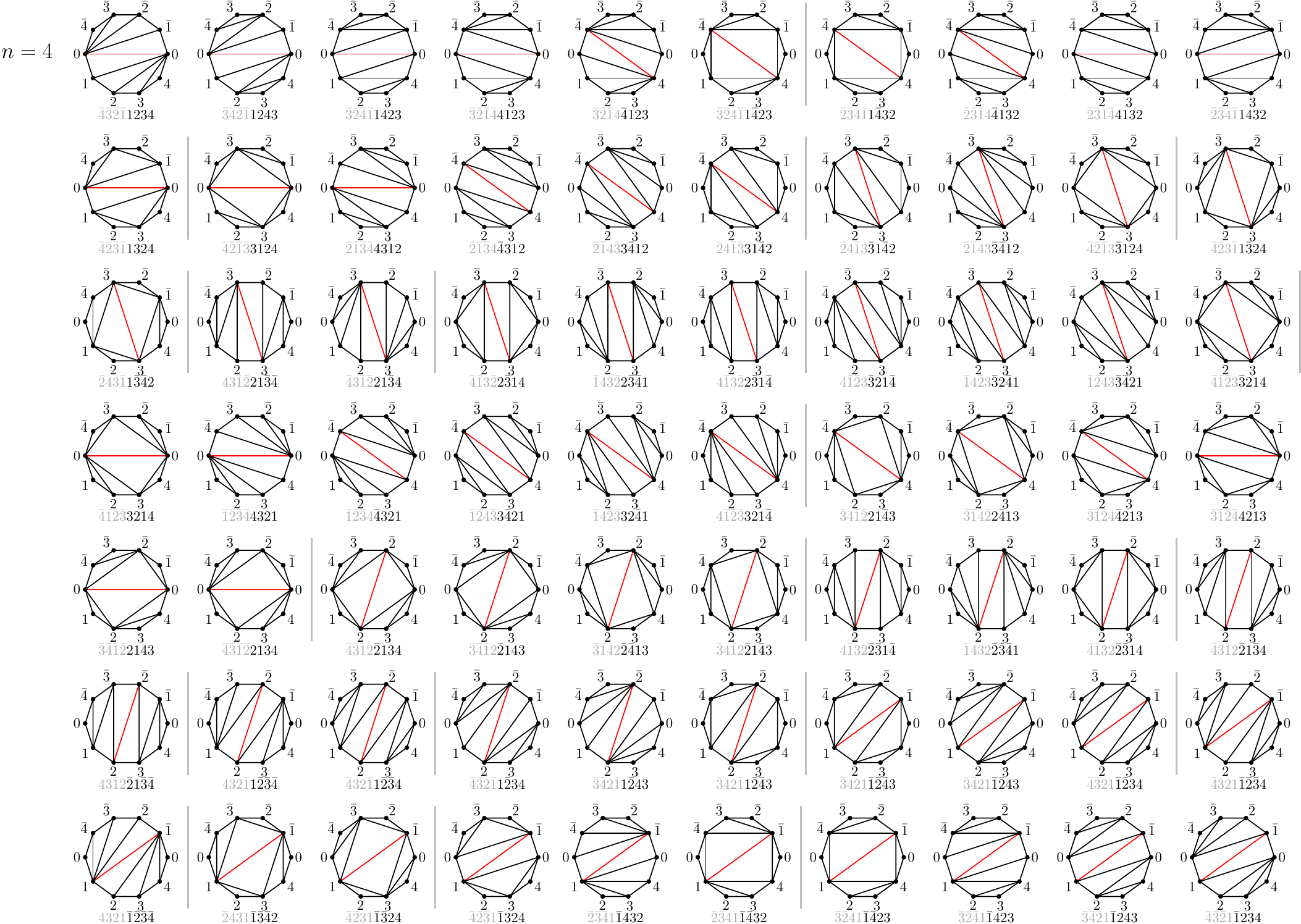}
\end{tabular}
}
\caption{Our new cyclic Gray code for point-symmetric triangulations of a convex $(2n+2)$-gon for $n=1,2,3,4$, given by Theorems~\ref{thm:h-bij} and~\ref{thm:2b31}.
Below each triangulation is the corresponding pattern-avoiding signed permutation.
The vertical bars separate groups of descendants that are derived from the same parent permutation/triangulation in the previous listing.}
\label{fig:gc-triang}
\end{figure}

It is not known whether every cover graph of a quotient of the lattice of regions of a supersolvable hyperplane arrangement can be realized as the skeleton of a polytope.
If so, then its skeleton admits a Hamiltonian path according to Theorem~\ref{thm:super-quotient}.
In fact, Padrol, Pilaud, and Ritter~\cite[Conj.~143]{MR4584712} conjecture that a polytopal realization is possible if the lattice of regions is congruence uniform.

\subsubsection{Zigzag languages of signed permutations}

As another contribution, we develop a theory of zigzag languages of signed permutations, and obtain the type~$B$ analogue of the methods described in Section~\ref{sec:zigzag} (see Theorem~\ref{thm:zigzag} below).
This yields Gray codes for generating several different pattern-avoiding signed permutations and the corresponding combinatorial objects in bijection to them.

\subsection{Outline of the paper}

In Section~\ref{sec:background} we set up additional terminology and notation needed for the proofs of our two main theorems.
The proofs of those theorems are presented in Section~\ref{sec:proofs}.
In Section~\ref{sec:zigzag-signed} we develop the theory of zigzag languages of signed permutations.
In Section~\ref{sec:except}, we describe three exceptional signed permutation patterns and establish bijections to other combinatorial objects.
In particular, we describe our new Gray code for point-symmetric triangulations of a convex polygon, and the corresponding Hamiltonian cycle on the $B$-associahedron.

\section{Background and terminology}
\label{sec:background}

\subsection{Partial orders and lattices}
\label{sec:poset}

A \defn{poset} $(P,\leq)$ is a set~$P$ together with a partial order~$\leq$ on~$P$.
We write~$<$ for the corresponding strict partial order.
An element~$x\in P$ is \defn{minimal} if there is no~$y\in P$ with $y<x$.
Similarly, $x$ is \defn{maximal} if there is no~$y\in P$ with $x<y$.
A \defn{cover} is a pair~$(x,y)$ of elements of~$P$ such that $x<y$ and there does not exist~$z$ such that $x<z<y$.
We write $x\lessdot y$ if $(x,y)$ is a cover.
The cover relation~$\lessdot$ is the reflexive, transitive reduction of the order relation~$\leq$.
Conversely, the order relation $\leq$ is the reflexive, transitive closure of the cover relation~$\lessdot$.
The \defn{cover graph of~$P$} is the (undirected) graph that has the elements of~$P$ as vertices, and as edges exactly the covers.
A poset~$P$ is often visualized by a \defn{Hasse diagram}, which is a drawing of the cover graph in the plane in which for every cover~$(x,y)$, the vertex~$y$ is embedded higher than~$x$.
A \defn{chain} in~$P$ is a sequence of covers $x_0 \lessdot x_1 \lessdot\cdots \lessdot x_\ell$.

The \defn{downset} and \defn{upset} of an element~$x\in P$ are defined as $x^\downarrow:=\{y\mid y\leq x\}$ and~$x^\uparrow:=\{y\mid x\leq y\}$.
The \defn{interval} between two elements $x,y\in P$ is the subposet~$[x,y]:=x^\uparrow\cap y^\downarrow=\{z \mid  x\leq z\leq y\}$.
A poset~$P$ is \defn{graded} if there is a \defn{rank function}~$\rank(x)$ mapping elements of~$P$ to integers such that $\rank(x)<\rank(y)$ if~$x<y$ and $\rank(y)=\rank(x)+1$ if~$x\lessdot y$.

Given two elements $x,y\in P$, their \defn{meet} $x\meet y$ is the maximum of~$x^\downarrow\cap y^\downarrow$, if it exists and is unique.
Dually, their \defn{join} $x\join y$ is the minimum of~$x^\uparrow\cap y^\uparrow$, if it exists and is unique.
The poset~$P$ is called a \defn{lattice} if meets and joins exist for any pair of elements.
It follows directly from this definition that any lattice has a unique minimum and maximum.

\subsection{Lattice congruences}
\label{sec:congr}

Let $(L,\leq)$ be a lattice.
For an equivalence relation~$\equiv$ on~$L$ and any~$x\in L$, we write $[x]:=\{y\in L\mid x\equiv y\}$ for the equivalence class of~$x$, and we write~$L/{\equiv}:=\{[x]\mid x\in L\}$ for the set of all equivalence classes.

A \defn{lattice congruence} is an equivalence relation~$\equiv$ on~$L$ such that for all $x,x',y,y'\in L$, if $x\equiv x'$ and $y\equiv y'$, then $x\meet y\equiv x'\meet y'$ and $x\join y\equiv x'\join y'$.
The set of equivalence classes $L/{\equiv}$ can be ordered by defining $X\leq Y$ for $X,Y\in L/{\equiv}$ if there are elements~$x\in X$ and~$y\in Y$ with $x\leq y$.
Furthermore, this order on~$L/{\equiv}$ is a lattice, namely if $[x]$ and $[y]$ are equivalence classes, then their meet and join are $[x\meet y]$ and $[x\join y]$, respectively.
We refer to $L/{\equiv}$ with this lattice structure as the \defn{quotient lattice}.

We recall some fundamental results about lattice congruences; see \cite[\S 9-5]{MR3645055} for proofs.

\begin{lemma}
\label{lem:cong}
Let $\equiv$ be an equivalence relation on a finite lattice~$L$.
\begin{enumerate}[label=(\roman*),leftmargin=8mm]
\item Every equivalence class is an interval.
\item We have $X\lessdot Y$ in $L/{\equiv}$ if and only if there are elements~$x\in X$ and $y\in Y$ such that $x\lessdot y$ in~$L$.
\end{enumerate}
\end{lemma}

\subsection{Hyperplane arrangements}

We introduce the required terminology regarding hyperplane arrangements; for more details see~\cite{MR1217488} and~\cite{MR2383131}.
A \defn{hyperplane arrangement} in $\R^n$ is a nonempty finite set~$\cH$ of hyperplanes through the origin.
The \defn{rank} of an arrangement~$\cH$ is the dimension of the space spanned by the normal vectors of the hyperplanes, or equivalently, the codimension of the intersection of all the hyperplanes.

The \defn{regions} $\cR=\cR(\cH)$ of an arrangement~$\cH$ are the connected components of $\R^n\setminus\cH$.
The \defn{graph of regions} $G(\cH)$ is the adjacency graph of the set of regions, i.e., it has $\cR(\cH)$ as its vertex set, and an edge between any two regions that are separated by exactly one hyperplane from~$\cH$.

A \defn{flat} of an arrangement~$\cH$ in~$\R^n$ is a subspace of~$\R^n$ obtained by intersecting some subset of hyperplanes of~$\cH$.
The \defn{intersection lattice} $\cL=\cL(\cH)$ of the arrangement~$\cH$ is the set of flats ordered in reverse inclusion order, completed by a minimal element $\h{0} = \R^n$.
The intersection lattice is a geometric lattice, namely the matroid lattice of the matroid defined by the arrangements of normal vectors of the hyperplanes of~$\cH$.
Its \defn{atoms} are the hyperplanes, and the rank~$\rank(X)$ of an element $X\in \cL$ is equal to its \defn{codimension} $\codim (X) = n - \dim(X)$, so that $\rank (\h{0}) = 0$ and $\rank (H) = 1$ for all $H\in\cH$.
Its maximum is given by~$\bigcap_{H\in \cH} H$, and its rank in the lattice is precisely the rank of the arrangement~$\cH$.
The meet of two elements $X, Y\in \cL$ is equal to $X\meet Y=\bigcap \{Z\in \cL \mid Z\supseteq (X\cup Y)\}$, and the join is $X\join Y=X\cap Y$.

\subsection{Zonotopes}

With every hyperplane arrangement~$\cH$, one can associate the \defn{zonotope} $Z(\cH)$ defined as the Minkowski sum of line segments in the directions of the normal vectors of the hyperplanes in~$\cH$.
The graph of regions~$G(\cH)$ is isomorphic to the skeleton of~$Z(\cH)$.

\subsection{Coxeter arrangements}

Interesting hyperplane arrangements are obtained by letting the normal vectors form a root system of a finite Coxeter group.
The standard type~$A$, $B$, and~$D$ root systems (the type~$C$ is essentially equivalent to the type $B$ for our application) consist respectively of the following vectors; see~\cite{MR2133266} and~\cite{MR2383126}:
\begin{description}
\item[type $A$] $\{ \rvec{e_i} - \rvec{e_j} \mid i\neq j\in [n]\}$,
\item[type $B$] $\{ \rvec{e_i} \pm \rvec{e_j} \mid i\neq j\in [n]\}\cup\{ \pm\rvec{e_i} \mid i\in [n]\}$,
\item[type $D$] $\{ \rvec{e_i} \pm \rvec{e_j} \mid i\neq j\in [n]\}$.
\end{description}

\subsection{Graphic arrangements}

Given a simple graph~$F=([n],E)$, the \defn{graphic arrangement} $\cH(F)$ of~$F$ is the arrangement of hyperplanes with the normal vectors~$\{\rvec{e_i} - \rvec{e_j} \mid \{i,j\}\in E\}$.
The graphic arrangement is therefore a subarrangement of the type~$A$ Coxeter arrangement.
These normal vectors form a realization of the \defn{graphic matroid} of~$F$, and the intersection lattice of $\cH(F)$ is the lattice of flats of this matroid.
An \defn{acyclic orientation} of~$G$ is an assignment of a direction to each edge in~$E$ that does not create any directed cycles.
The following is a well-known observation from Greene (see~\cite{MR712251}).

\begin{lemma}
\label{lem:gaoa}
The regions of the graphic arrangement~$\cH(F)$ of a graph~$F$ are in bijection with the acyclic orientations of~$F$.
\end{lemma}

\subsection{Signed graphic arrangements}
\label{sec:sao}

A \defn{signed graph} is a simple graph $F=([n],E)$ in which every edge in~$E$ is given a sign from~$\{+,-\}$, i.e., we obtain a partition $E=E^+ \cup E^-$, where $E^+$ and~$E^-$ are the sets of edges with a positive and negative sign, respectively.
The \defn{signed graphic arrangement} $\cH(F)$ of~$F$ is a subarrangement of the type~$D$ root system consisting of the hyperplanes with normal vectors
\[
\{\rvec{e_i} - \rvec{e_j} \mid ij\in E^+\} \cup \{\rvec{e_i} + \rvec{e_j} \mid ij\in E^-\}.
\]
Note that if~$E=E^+$, then the signed graphic arrangement is a graphic arrangement.
Also note that Zaslavsky~\cite{MR0676405,MR1120422} considered signed graphs with loops and half-edges, whose graphic arrangements are type~$B$ subarrangements.
For simplicity, we restrict our definition to simple signed graphs.

An \defn{orientation} of an edge $e=ij$ of a signed graph consists of assigning a direction to the two half-edges composing~$e$.
If $e\in E^+$, then these two directions must be the same, either both towards~$i$ or both towards~$j$.
If $e\in E^-$, then the two directions must be opposite, exactly one towards~$i$ and one towards~$j$.
An \defn{orientation} of the signed graph~$F$ is defined by orienting each of its edges.

\begin{figure}[h!]
\makebox[0cm]{ % artificial box to center the picture
\includegraphics{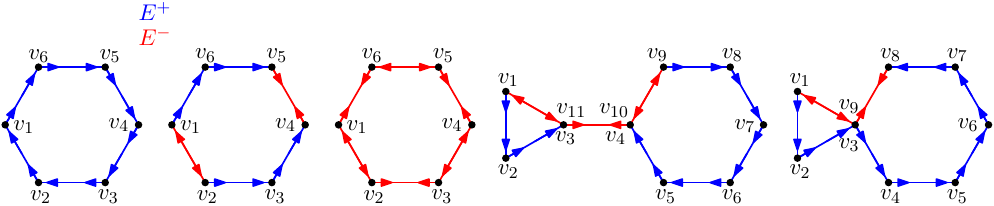}
}
\caption{Different types of cycles in oriented signed graphs.}
\label{fig:signed}
\end{figure}

A \defn{cycle} in an oriented signed graph is an inclusion-minimal subgraph whose vertices and edges can be arranged in an alternating sequence $(v_1,e_1,v_2,e_2,\ldots,v_\ell,e_{\ell})$ (possibly with repetitions) where for $i=1,\ldots,\ell$, the edge~$e_i$ has the end vertices~$v_i$ and~$v_{i+1}$, and from the two half-edges of~$e_{i-1}$ and~$e_i$ incident with~$v_i$ exactly one is oriented towards~$v_i$ (indices are considered cyclically modulo~$\ell$); see Figure~\ref{fig:signed}.
An orientation of a signed graph is \defn{acyclic} if it does not contain any cycle.

A cycle in an oriented signed graph may consist of a simple cycle in the underlying unsigned graph, in which case it must be \defn{balanced}, i.e., contain an even number of negative edges.
However, unlike for unsigned graphs, a cycle in an oriented signed graph may reuse edges and vertices, and may hence not be a cycle in the underlying unsigned graph.
Specifically, it may consist of two unbalanced cycles connected by a path or sharing exactly one vertex~\cite{MR1120422}.

Zaslavsky~\cite{MR1120422} proved the following analogue of Lemma~\ref{lem:gaoa}.

\begin{lemma}
\label{lem:sgaoa}
The regions of the signed graphic arrangement~$\cH(F)$ of a signed graph~$F$ are in bijection with the acyclic orientations of~$F$.
\end{lemma}

\subsection{Supersolvable arrangements}
\label{sec:supersolvable}

An arrangement of hyperplanes~$\cH$ is supersolvable if its intersection lattice~$\cL(\cH)$ is supersolvable in the sense of Stanley~\cite{MR0309815}.
We will instead use the following characterization of supersolvable arrangements given by Bj\"orner, Edelman, and Ziegler~\cite[Thm.~4.3]{MR1036875} as our definition.
An arrangement~$\cH$ of rank~$n$ is \defn{supersolvable} if either $n\leq 2$, or $n\geq 3$ and $\cH$ is the disjoint union of two nonempty arrangements $\cH_0$ and $\cH_1$, where $\cH_0$ is a supersolvable arrangement of rank~$n-1$, and for any pair of distinct hyperplanes $H, H'\in \cH_1$, there exists a hyperplane $H''\in\cH_0$ such that $H\cap H'\seq H''$.

We proceed to discuss which of the aforementioned hyperplane arrangements are supersolvable.
Among the Coxeter arrangements, those of types~$A$ and~$B$ are known to be the only supersolvable arrangements.

For graphic arrangements, supersolvability of the arrangement~$\cH(F)$ is equivalent to chordality of the graph~$F$.
Formally, a graph is \defn{chordal} if it does not contain any induced cycle of length~$>3$.
One can also characterize chordal graphs by the existence of a certain elimination ordering on their vertices.
Specifically, a vertex~$v$ in a graph~$F$ is \defn{simplicial} if the neighbors of~$v$ induce a complete subgraph of~$F$.
An ordering $v_1,v_2,\ldots ,v_n$ of the vertices of~$F$ is a \defn{perfect elimination ordering}
if $v_i$ is simplicial in $F\setminus \{v_1,\ldots ,v_{i-1}\}$ for all $i=1,\ldots,n$.
Note that a (reverse) perfect elimination ordering is obtained by repeatedly removing a simplicial vertex from~$F$.
It is well-known that a graph~$F$ is chordal if and only if it has a perfect elimination ordering.
Stanley~\cite[Prop.~2.8]{MR0309815} proved the following; see also \cite[Cor.~4.10]{MR2383131}.

\begin{theorem}
\label{thm:supergraphic}
The graphic arrangement $\cH(F)$ of a graph~$F$ is supersolvable if and only if $F$ is chordal.
\end{theorem}

A similar result is known for signed graphic arrangements.
Specifically, a vertex~$v$ of a signed graph~$F$ with edge set~$E=E^+\cup E^-$ is \defn{signed simplicial} if for any two distinct neighbors~$x,y$ of~$v$, there is an edge between~$x$ and~$y$ that completes a balanced triangle, formally:
\begin{itemize}[leftmargin=5mm]
\item if $vx, vy\in E^+$ or $vx, vy\in E^-$, then we have $xy\in E^+$;
\item if $vx\in E^+$ and~$vy\in E^-$, then we have $xy\in E^-$.
\end{itemize}
An ordering $v_1,v_2,\ldots ,v_n$ on the vertices of~$F$ is a \defn{signed perfect elimination ordering} if $v_i$ is signed simplicial in~$F\setminus \{v_1,\ldots ,v_{i-1}\}$.
Note that a (reverse) signed perfect elimination ordering is obtained by repeatedly removing a signed simplicial vertex from~$F$.

Zaslavsky~\cite[Thm.~2.2]{MR1808091} proved the following; see also~\cite[Thm.~4.19]{MR3886267} for the explicit statement on signed graphs.

\begin{theorem}
\label{thm:ssupergraphic}
If $F$ is a signed graph graph with a signed perfect elimination ordering, then the signed graphic arrangement~$\cH(F)$ is supersolvable.
\end{theorem}

Note that Zaslavsky actually gave a complete characterization (if and only if) of signed graphs whose arrangements are supersolvable, which includes specific additional graphs.
We omit the details here.

\subsection{Structure of the graph of regions~$G(\cH)$}

For the following lemmas, we consider a supersolvable arrangement~$\cH$ of rank~$n\geq 3$ and the decomposition of~$\cH$ into two subarrangements~$\cH_0$ and~$\cH_1$ given by the definition.
We define the surjective map
\begin{equation}
\label{eq:rho}
\rho : \cR(\cH) \to \cR(\cH_0)
\end{equation}
that maps every region of~$\cH$ to the region of~$\cH_0$ that it is contained in.

We recall the following basic properties of supersolvable hyperplane arrangements that can be found, for instance, in~\cite{MR737888, MR1036875}.

\begin{lemma}
\label{lem:regions}
Every hyperplane $H\in\cH_1$ has the following properties:
\begin{enumerate}[label=(\roman*),leftmargin=8mm]
\item We have $\rank(\cH_0 \cup \{H\})=n>\rank(\cH_0)=n-1$.
\item $H$ splits every region in~$\cR(\cH_0)$ into exactly two regions.
\end{enumerate}

\begin{proof}
This proof is illustrated in Figure~\ref{fig:planes}.

\begin{figure}[h!]
\includegraphics[page=1]{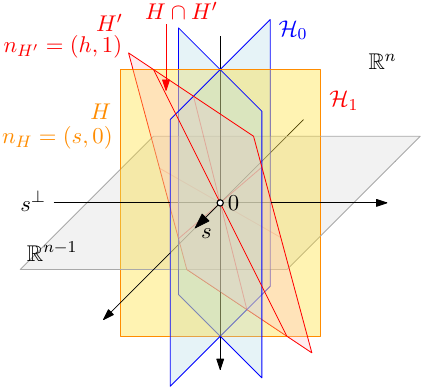}
\caption{Illustration of the proof of Lemma~\ref{lem:regions}.}
\label{fig:planes}
\end{figure}

We assume w.l.o.g.\ that the dimension of the ambient space equals~$n=\rank(\cH)$.
Let $N(\cH_0)$ denote the set of normal vectors of the hyperplanes in~$\cH_0$.
For vectors $x,y\in\R^n$, we write $x\cdot y$ for the standard scalar product of~$x$ and~$y$.
By the definition of supersolvability, we know that $\rank(\cH_0)=n-1<\rank(\cH)=n$.
We may choose a basis of $\R^n$ such that $\vspan(N(\cH_0)) \seq \R^{n-1} \times \{0\}$.

To prove~(i), consider a hyperplane~$H \in \cH_1$, and suppose for the sake of contradiction that $\rank(\cH_0 \cup \{H\})=n-1$, i.e., $H$ has a normal vector $n_H=(s,0)\in\R^n$ with $s\in\R^{n-1}$.
Let $H' \in \cH_1$ with normal vector $n_{H'} \notin \R^{n-1} \times \{0\}$.
Such a hyperplane exists in~$\cH_1$ since $\rank(\cH)=n>\rank(\cH_0)=n-1$.
We may assume w.l.o.g.\ that $n_{H'}=(h,1)$ for some $h \in \R^{n-1}$.
Then we have
\[ H \cap H' = \{(x,\lambda) \in \R^{n-1} \times \R \mid x\cdot s=0 \wedge x\cdot h+\lambda=0\}. \]
In particular, the map $H \cap H' \to s^\perp$ defined by $(x,\lambda) \mapsto x$ is surjective.
Here, $s^\perp \seq \R^{n-1}$ denotes the set of vectors in~$\R^{n-1}$ orthogonal to~$s$.
Since $\cH$ is supersolvable, there exists a hyperplane $H''\in \cH_0$ with $H\cap H' \seq H''$.
Let $n_{H''}=(t,0)\in\R^n$ for some $t \in \R^{n-1}$ be one of the normal vectors of~$H''$.
Let $x \in s^\perp$.
By the above, there is $\lambda \in \R$ such that $(x,\lambda) \in H \cap H' \seq H''$.
Hence $(x,\lambda)\cdot (t,0)=x\cdot t=0$.
This shows $s^\perp \seq t^\perp$, so $s^\perp=t^\perp$.
But this implies $H=H''\in \cH_0$, a contradiction, so (i) is proved.

To prove~(ii), note that every region $R\in\cR(\cH_0)$ is of the form $R=B \times \R$ for some $B \seq \R^{n-1}$ as $N(\cH_0) \seq \R^{n-1} \times \{0\}$.
Let $H \in \cH_1$ with normal vector $n_H \in \R^n$.
By~(i), we may assume that $n_H=(h,1)$ for some $h \in \R^{m-1}$.
Then $\{x \in \R^n \mid x\cdot n_H > 0\}\cap R = \{(x,\lambda) \in B\times \R \mid x\cdot h + \lambda > 0\}$ is a proper subset of $R$.
Hence $H$ splits $R$ into two regions.
\end{proof}
\end{lemma}

The following key lemma captures the recursive structure of the graph of regions~$G(\cH)$; see Figure~\ref{fig:GH}.

\begin{figure}
\includegraphics[page=13]{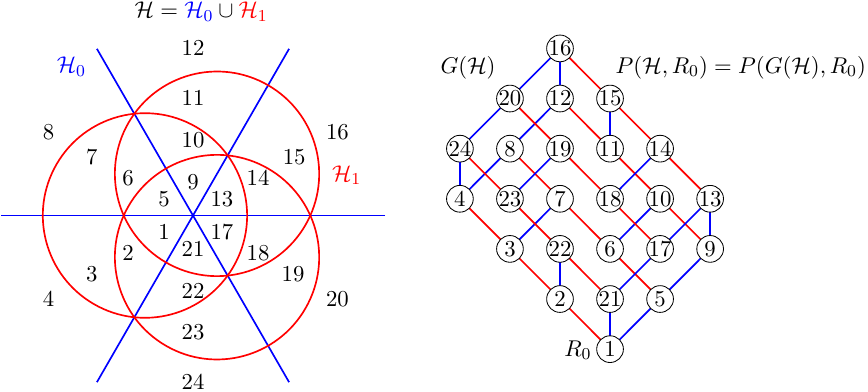}
\caption{The graph of regions~$G(\cH)$ of the arrangement of Figure~\ref{fig:permA4}.
The lattice of regions~$P(\cH,R_0)$ is obtained by picking the region labeled~1 as canonical base region~$R_0$.
The resulting lattice is isomorphic to the one shown in Figure~\ref{fig:congA}\,(a).
}
\label{fig:GH}
\end{figure}

\begin{lemma}
\label{lem:GH-structure}
The graph of regions~$G(\cH)$ has the following properties:
\begin{enumerate}[label=(\roman*),leftmargin=8mm]
\item For any $R\in\cR(\cH_0)$, the subgraph of~$G(\cH)$ induced by the regions in~$\rho^{-1}(R)$ is a path of length~$|\cH_1|$.
\item If two distinct regions $R,R'\in\cR(\cH_0)$ are adjacent in~$G(\cH_0)$, then in~$G(\cH)$ the only edges between the two paths~$\rho^{-1}(R)$ and~$\rho^{-1}(R')$ are between the first two and the last two vertices, and possibly between other pairs of vertices at the same distance from the end vertices.
\item If two distinct regions $R,R'\in\cR(\cH_0)$ are not adjacent in~$G(\cH_0)$, then in~$G(\cH)$ there are no edges between the two paths~$\rho^{-1}(R)$ and~$\rho^{-1}(R')$.
\end{enumerate}
\end{lemma}

\begin{proof}
We assume w.l.o.g.\ that the dimension of the ambient space equals~$n=\rank(\cH)$.

We first prove~(i), and this part of the proof is illustrated in Figure~\ref{fig:cone}~(a).
For every region~$R\in\cR(\cH_0)$ we can choose normal vectors~$n_H\in\R^n$ for all $H\in\cH_0$ such that $R$ is the convex cone $R=\{x\in\R^n\mid x\cdot n_H>0 \text{ for all } H\in \cH_0\}$.
By Lemma~\ref{lem:regions}~(ii), every hyperplane $H\in\cH_1$ splits $R$ into two regions.
For any $x\in R$, consider the ray $r_x(\lambda):=\lambda x\in R$, where $\lambda$ is a real-valued parameter in the range~$\lambda\in(0,\infty)$, which intersects each of the hyperplanes~$H\in\cH_1$.
Specifically, for very small $\lambda>0$ the point $r_x(\lambda)$ lies on one side of~$H$, and for very large~$\lambda$ it lies on the other side of~$H$.
Consequently, the ray $r_x(\lambda)$ defines a total ordering of the hyperplanes in~$\cH_1$, namely the order in which they are intersected by $r_x(\lambda)$ as the parameter $\lambda$ increases.
We claim that this total ordering is independent of the direction vector~$x\in R$.
Indeed, suppose for the sake of contradiction that there are points~$x,y\in R$, values $0<\lambda_{x,1}<\lambda_{x,2}$, $0<\lambda_{y,1}<\lambda_{y,2}$ and two hyperplanes~$H,H'\in \cH_1$ such that $r_x(\lambda_{x,1})\in H$, $r_x(\lambda_{x,2})\in H'$, $r_y(\lambda_{y,1})\in H'$, and $r_y(\lambda_{y,2})\in H$.
Then for the direction vector $z=\nu x+(1-\nu)y\in R$ for some suitable $\nu\in(0,1)$ and $\lambda>0$ we would have $r_z(\lambda)\in H\cap H'$, i.e., some point in the intersection~$H\cap H'$ that is not contained in any hyperplane from~$H_0$, a contradiction.
It follows that for every~$x\in R$, the ray $r_x(\lambda)$ intersects the hyperplanes of~$\cH_1$ in the same order, independent of~$x$.
We distinguish two special regions~$R_0,R_\infty\in\rho^{-1}(R)$, namely the one intersected by~$r_x(\lambda)$ for very small~$\lambda>0$ and the one intersected for very large~$\lambda$, respectively.
By what we said before, each region in~$\rho^{-1}(R)\setminus\{R_0,R_\infty\}$ is bounded by exactly two hyperplanes from~$\cH_1$, and $R_0,R_\infty$ are bounded by exactly one hyperplane from~$\cH_1$.
This completes the proof of~(i).

\begin{figure}
\includegraphics[page=2]{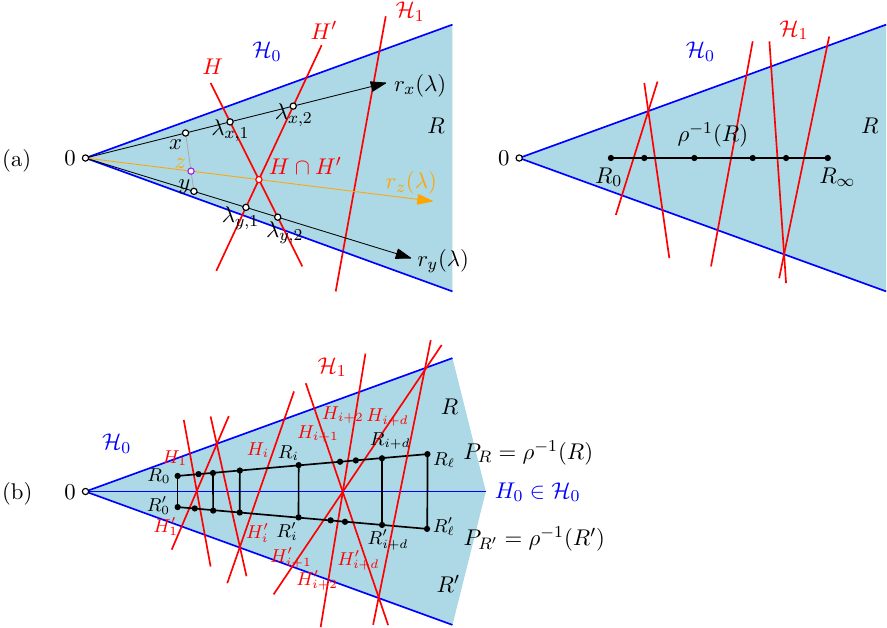}
\caption{Illustration of the proof of Lemma~\ref{lem:GH-structure}.}
\label{fig:cone}
\end{figure}

The proof of~(ii) is illustrated in Figure~\ref{fig:cone}~(b).
Let $R,R' \in \cR(\cH_0)$ be adjacent, i.e., these two regions are separated by exactly one hyperplane from~$\cH_0$, which we denote by~$H_0$.
By (i), $P_R:=\rho^{-1}(R)$ and $P_{R'}:=\rho^{-1}(R')$ are paths of length~$\ell:=|\cH_1|$.
Let $P_R=(R_0,\ldots,R_\ell)$ and $P_{R'}=(R_0',\ldots,R_\ell')$ be the sequences of regions along each of the paths.
Each edge of the paths corresponds to a hyperplane from~$\cH_1$, so for $i=1,\ldots,\ell$ we let $H_i\in \cH_1$ denote the hyperplane separating~$R_{i-1}$ and~$R_i$, and $H_i'\in \cH_1$ the hyperplane separating~$R_{i-1}'$ and~$R_i'$.
As any region~$R_i$, $i\in\{0,\ldots,\ell\}$, is separated from any region~$R_j'$, $j\in\{0,\ldots,\ell\}$, by $H_0\in\cH_0$, we have that $R_i$ and~$R_j'$ are adjacent in~$G(\cH)$ if and only if they are not separated by any hyperplanes from~$\cH_1$.

We observe the following for regions $R_i$ and~$R_i'$ that are adjacent in the graph~$G(\cH)$.
First of all we must have $\{H_1,\ldots,H_i\}=\{H_1',\ldots,H_i'\}$ and $\{H_{i+1},\ldots,H_\ell\}=\{H_{i+1}',\ldots,H_\ell'\}$.
Now let $d>0$ be the smallest integer such that $H_{i+d}=H_{i+1}'$.
If $d>1$, then as $H_{i+1}$ and~$H_{i+d}$ appear in reversed order along~$P_R$ and~$P_{R'}$, then these two hyperplanes must intersect, and their intersection lies in~$H_0$.
Furthermore, by a similar argument and the linearity of the hyperplanes we must have $(H_{i+1},H_{i+2},\ldots,H_{i+d})=(H_{i+d}',\ldots,H_{i+2}',H_{i+1}')$, i.e., the next $d$ hyperplanes appear in the opposite order along the two paths~$P_R$ and~$P_{R'}$.
Note that this conclusion is trivially true if $d=1$.
For any $d>0$ we obtain that $R_{i+d}$ and~$R_{i+d}'$ are separated by no hyperplane from~$\cH_1$ (only by~$H_0\in\cH_0$), and therefore they are adjacent in the graph~$G(\cH)$.
Furthermore, one can check that $R_j$ for $j=i+1,\ldots,i+d$, is separated from~$R_k'$, $k=0,\ldots,i+d-1$, by at least one hyperplane from~$\cH_1$, in addition to~$H_0\in\cH_0$, and therefore $R_j$ and~$R_k'$ are not adjacent in the graph~$G(\cH)$.
Observing that $R_0$ and~$R_0'$ are adjacent in~$G(\cH)$, and applying the aforementioned observation inductively for $i=0,\ldots,\ell$ proves~(ii).

To prove~(iii), note that as $R,R'$ are not adjacent in $G(\cH_0)$, they are separated by at least two hyperplanes $H,H'\in\cH_0$.
It follows that any two regions $P \in \rho^{-1}(R)$ and $P' \in \rho^{-1}(R')$ are also separated by both~$H$ and~$H'$, and therefore $P$ and~$P'$ are not adjacent in~$G(\cH)$.
\end{proof}

\subsection{Lattice of regions}
\label{sec:regions}

For any graph~$G$ and any two vertices~$x,y$ in~$G$, we write $d_G(x,y)$ for the \defn{distance} between~$x$ and~$y$ in~$G$, i.e., for the length of the shortest path connecting~$x$ and~$y$.
Let $G=(V,E)$ be a connected bipartite graph, and let~$x_0\in V$ be one of its vertices.
We let $P(G,x_0)$ be the graded poset on the ground set~$V$ whose cover graph is obtained from~$G$ by orienting all edges away from~$x_0$.
Formally, we define the rank of any $y\in V$ in~$P(G,x_0)$ as $\rank(y):=d_G(x_0,y)$.
As $G$ is assumed to be bipartite, there are no odd cycles and hence no edges between vertices of the same rank, so this is well-defined.

Using this definition, by choosing one region~$R_0\in\cR$ as \defn{base region}, one can orient the edges of the graph of regions~$G(\cH)$ away from~$R_0$.
As $G(\cH)$ is always connected and bipartite, this defines a graded poset~$P(\cH,R_0):=P(G(\cH),R_0)$, called the \defn{poset of regions}; see Figure~\ref{fig:GH}.
It is easy to see that in this poset the rank of any~$R\in\cR$ is given by the number of hyperplanes from~$\cH$ that separate~$R_0$ and~$R$, i.e.,
\[
\rank(R)=|\{H\in\cH\mid R_0 \text{ and } R \text{ lie on opposite sides of } H\}.
\]
Bj\"orner, Edelman, and Ziegler~\cite{MR1036875} proved that for a supersolvable arrangement~$\cH$, there always exists a region~$R_0\in\cR(\cH)$ such that the poset~$P(\cH,R_0)$ is a lattice.

To state the next lemma, we define the notion of a \defn{canonical} region for a supersolvable arrangement~$\cH$ of rank~$n\geq 1$.
If $n\leq 2$, then every region of~$\cR(\cH)$ is canonical.
If $n\geq 3$, we consider the map~$\rho$ defined in~\eqref{eq:rho}, and for any~$R\in\cR(\cH)$ we define the \defn{fiber} $\varphi(R):=\rho^{-1}(\rho(R))\seq \cR(\cH)$.
Then a region $R_0\in\cR(\cH)$ is canonical if $\rho(R_0)$ is canonical for~$\cH_0$ and if $R_0$ is the first or last vertex of the path induced by the fiber~$\varphi(R_0)$ in~$G(\cH)$ (recall Lemma~\ref{lem:GH-structure}~(i)).

\begin{lemma}[{\cite[Thm.~4.6]{MR1036875}}]
\label{lem:P-lattice}
If $R_0$ is a canonical region, then the poset~$P(\cH, R_0)$ is a lattice.
\end{lemma}

\section{Proofs of Theorems~\ref{thm:super-ham} and~\ref{thm:super-quotient}}
\label{sec:proofs}

\subsection{Graphic suspension}

For a graph~$G=(V,E)$ and subset of vertices~$U\seq V$ we write~$G[U]$ for the subgraph of~$G$ induced by~$U$.
Given two graphs~$H=(V,E)$ and~$F=(V',E')$, the \defn{Cartesian product}~$H\times F$ is the graph on the vertex set~$V\times V'$ with edge set
\[\big\{\{(x,x'),(y,y')\}\mid (\{x,y\}\in E \wedge x'=y') \vee (x=y \wedge \{x',y'\}\in E')\big\}.\]
We write~$P_\ell$ for the path of length~$\ell$ on the vertices $0,1,\ldots,\ell$, and we write $G\simeq H$ to indicate that two graphs~$G$ and~$H$ are isomorphic.

Let $G$ and~$H=(V,E)$ be graphs, and let $\ell\geq 1$ be an integer.
We say that~$G$ is \defn{$\ell$-suspended} over~$H$ if $G$ is a spanning subgraph of the Cartesian product~$H\times P_\ell$ such that $G[\{v\}\times P_\ell]\simeq P_\ell$ for all $v\in V$, and furthermore $G[V\times\{0\}]\simeq H$ and~$G[V\times\{\ell\}]\simeq H$.
In words, $G$ is a spanning subgraph of the Cartesian product of~$H$ with a path of length~$\ell$ in which all copies of the path are present, and the first and last copy of~$H$ along the paths are present, whereas edges of the other intermediate copies of~$H$ along the paths may or may not be present.
The notion of suspension is introduced to provide a graphic translation and clarification of the structure described in~\cite{MR1036875}.
Specifically, using this definition, Lemma~\ref{lem:GH-structure} can be restated more concisely as follows.

\begin{lemma}
\label{lem:susp}
The graph~$G(\cH)$ is (isomorphic to a graph that is) $\ell$-suspended over~$G(\cH_0)$ for $\ell:=|\cH_1|$.
\end{lemma}

\begin{lemma}
\label{lem:susp-dist}
Suppose that $G$ is $\ell$-suspended over~$H$.
Then for any two vertices~$x,y$ of~$H$ and any~$i\in\{0,\ldots,\ell\}$ we have $d_G((x,0),(y,i))=d_G((x,\ell),(y,\ell-i))=d_H(x,y)+i$.
\end{lemma}

\begin{proof}
This is immediate from the definition of suspension via the Cartesian product.
\end{proof}

We also need a second (non-equivalent) definition of suspension, which does not use Cartesian products.
Given a sequence~$(\lambda(x))_{x\in V}$ of integers~$\lambda(x)\geq 1$, we say that $G$ is \defn{$\lambda$-suspended} over~$H=(V,E)$ if $G$ is the union of the paths $((x,0),(x,1),\ldots,(x,\lambda(x)))$ and the edges~$\{(x,0),(y,0)\}$ and~$\{(x,\lambda(x)),(y,\lambda(y))\}$ for all $\{x,y\}\in E$.
In words, $G$ is $\lambda$-suspended over~$H$ if it consists of two disjoint copies of~$H$, which are connected via disjoint paths between corresponding pairs of vertices, where the lengths of the paths are specified by the function~$\lambda$.
Note that if $G$ is $\lambda$-suspended for the constant function~$\lambda:=\ell$ for all $x\in V$, then it is $\ell$-suspended, and furthermore it is the $\ell$-suspended graph with the fewest possible edges.

The following easy lemma asserts that suspended graphs behave particularly nicely with regards to Hamiltonian paths/cycles.

\begin{lemma}
\label{lem:susp-ham}
Suppose that $G$ is $\lambda$-suspended over~$H$.
If $H$ has a Hamiltonian path, then $G$ has a Hamiltonian path.

Furthermore, if $G$ is $\ell$-suspended over~$H$ and $H$ has a Hamiltonian cycle of even length, then~$G$ has a Hamiltonian cycle of even length.
\end{lemma}

\begin{figure}
\includegraphics[page=14]{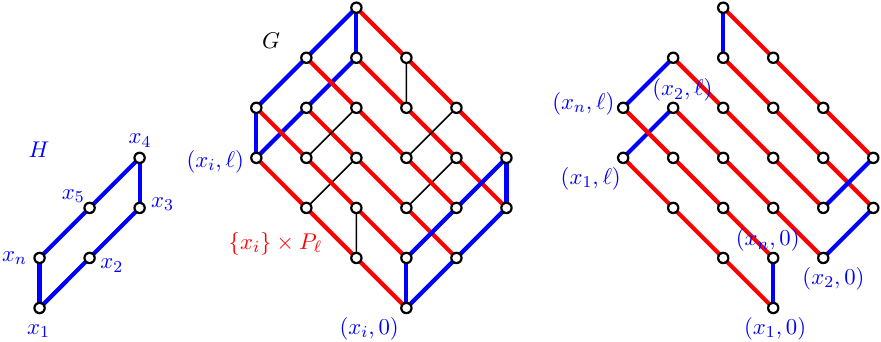}
\caption{Illustration of Lemma~\ref{lem:susp-ham}, continuing the example from Figure~\ref{fig:GH} (also see Figure~\ref{fig:zigzag}).}
\label{fig:susp-ham}
\end{figure}

Lemma~\ref{lem:susp-ham} and its proof are illustrated in Figure~\ref{fig:susp-ham}.

\begin{proof}
Let $V$ be the vertex set of~$H$ and let $(x_1,\ldots,x_n)$ be a Hamiltonian path in~$H$.
Then for each odd index~$k$, replace $x_k$ by the path~$((x_k,0),(x_k,1),\ldots,(x_k,\lambda(x_k)))$ in~$G$.
Similarly, for each even index~$k$, replace $x_k$ by the path~$((x_k,\lambda(x_k)),(x_k,\lambda(x_k)-1),\ldots,(x_k,0))$ in~$G$.
Consecutive pairs of end vertices from these subpaths are of the form $\{(x_k,0),(x_{k+1},0)\}$ if $k$ is even and $\{(x_k,\lambda(x_k)),(x_{k+1},\lambda(x_{k+1}))\}$ if $k$ is odd, respectively, i.e., these are edges in~$G$.
We conclude that the resulting sequence of vertices in~$G$ is a Hamiltonian path.

To prove the second part of the lemma assume that $n$ is even and that $(x_1,\dots,x_n)$ is Hamiltonian cycle in~$H$, i.e., a Hamiltonian path such that~$x_n$ and~$x_1$ are adjacent vertices.
Then the above argument yields a Hamiltonian path from~$(x_1,0)$ to $(x_n,0)$ in~$G$, so the pair~$\{(x_1,0),(x_n,0)\}$ is an edge in~$G[V\times\{0\}]\simeq H$.
We thus also obtain a Hamiltonian cycle in~$G$ of length~$(\ell+1)\cdot n$, which is even.
\end{proof}

\subsection{Congruences on lattices with suspended cover graphs}

The following crucial proposition describes how a congruence acts on a lattice whose cover graph is suspended.

\begin{proposition}
\label{prop:susp-cong}
Suppose that $G$ is $\ell$-suspended over~$H=(V,E)$, and let~$x_0$ be a vertex of~$H$ such that $L_H:=P(H,x_0)$ and $L_G:=P(G,(x_0,0))$ are both lattices.
Then $L_G[V\times\{0\}]$ and~$L_G[V\times\{\ell\}]$ are two sublattices isomorphic to~$L_H$, and $(x,0)\lessdot (x,1)\lessdot \cdots\lessdot (x,\ell)$ is a chain for each $x\in V$.
Furthermore, for any lattice congruence~$\equiv$ on~$L_G$ we have the following:
\begin{enumerate}[label=(\roman*),leftmargin=8mm]
\item The restriction ${\equiv^*}:=\{(x,y)\mid (x,0)\equiv (y,0)\}$ is a lattice congruence on~$L_H$.
\item For any equivalence class~$X$ of~$\equiv$, the projection~$p(X):=\{x\mid (x,i)\in X\}$ is an equivalence class of~$\equiv^*$.
\item The quantity $\lambda(x):=|\{[(x,i)]\mid i=0,\ldots,\ell\}|-1$, defined for any~$x\in V$, is either equal to~0 for all~$x\in V$ or strictly positive for all~$x\in V$.
\item If the cover graph of~$L_H/{\equiv^*}$ has a Hamiltonian path, then the cover graph of~$L_G/{\equiv}$ has a Hamiltonian path.
\end{enumerate}
\end{proposition}

Of course, by symmetry the poset~$P(G,(x_0,\ell))$ is also a lattice.

\begin{proof}
For any~$x\in V$ and~$r\in\{0,\ldots,\ell\}$ we define $C_\lambda(x):=\{(x,i)\mid 0\leq i\leq \lambda\}$ and~$C(x):=C_\ell(x)$.
The fact that $L_G[V\times\{0\}]$ and~$L_G[V\times\{\ell\}]$ are sublattices isomorphic to~$L_H$ and that the elements of~$C(x)$ form the chain~$(x,0)\lessdot (x,1)\lessdot \cdots\lessdot (x,\ell)$ for each~$x\in V$ follows directly from Lemma~\ref{lem:susp-dist} and our definition of rank in~$L_G=P(G,(x_0,0))$.

We now prove~(i).
For any two $x,y\in V$ we have
\begin{equation}
\label{eq:meet0}
(x,0)\meet (y,0)=(x\meet y,0).
\end{equation}
Now consider four elements~$x,x',y,y'\in V$ satisfying $x\equiv^* x'$ and~$y\equiv^* y'$.
From the definition of restriction, we have $(x,0)\equiv (x',0)$ and~$(y,0)\equiv (y',0)$.
Applying the definition of lattice congruence to~$\equiv$, we obtain that~$(x,0)\meet (y,0)\equiv (x',0)\meet (y',0)$.
Applying~\eqref{eq:meet0} to this relation yields~$(x\meet y,0)\equiv (x'\meet y',0)$, from which we obtain~$x\meet y\equiv^* x'\meet y'$ with the definition of restriction.
The argument for joins is dual, proving that $\equiv^*$ is indeed a lattice congruence of~$L_H$.

\begin{figure}[t!]
\makebox[0cm]{ % artificial box to center the picture
\includegraphics{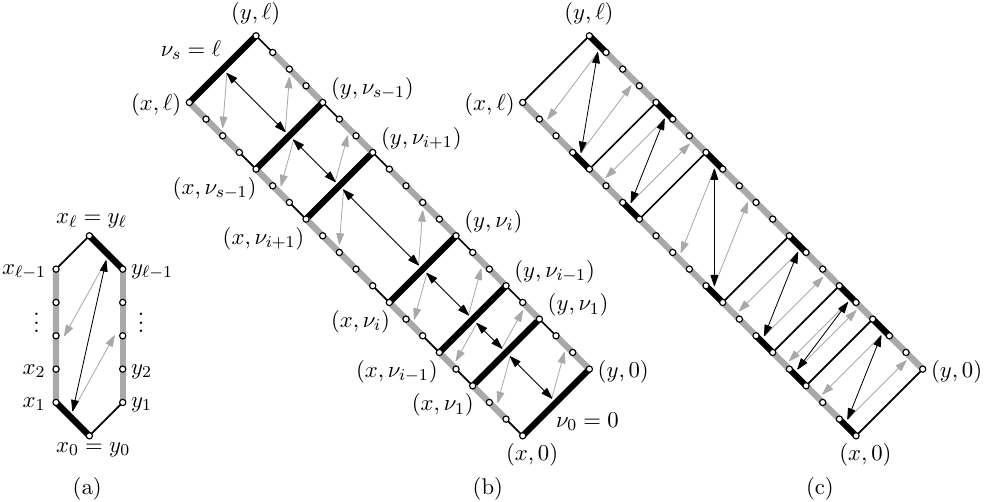}
}
\caption{Illustration of the proof of Proposition~\ref{prop:susp-cong}.
The bold lines indicate pairs of elements that are in the same equivalence class.
The arrows indicate (directed) forcing constraints between pairs of equivalent pairs.}
\label{fig:forcing}
\end{figure}

The proof of~(ii) and~(iii) needs the following auxiliary claim~(a), which is immediate from the definition of lattice congruence; see Figure~\ref{fig:forcing}~(a):
Consider the lattice~$L$ formed by two chains~$x_0\lessdot x_1\lessdot\cdots\lessdot x_\ell$ and $y_0\lessdot y_1\lessdot\cdots\lessdot y_\ell$ with $x_0=y_0$, $x_\ell=y_\ell$, and $x_i\neq y_j$ for any $i,j\in\{1,\ldots,\ell-1\}$, and let~$\equiv$ be a lattice congruence on~$L$.
If $x_0\equiv x_1$, then we have $y_\ell\equiv y_{\ell-1}\equiv \cdots\equiv y_1$.
If $y_\ell\equiv y_{\ell-1}$, then we have $x_0\equiv x_1\equiv \cdots\equiv x_{\ell-1}$.

Given $\{x,y\}\in E$ and a sequence~$\nu=(\nu_0,\ldots,\nu_s)$ where $s\geq 1$ and $0=\nu_0<\nu_1<\cdots<\nu_s=\ell$, where $s\geq 1$, the \defn{$(\ell,\nu)$-ladder} is the $\ell$-suspended graph over~$\{x,y\}$ in which precisely the edges~$\{(x,\nu_i),(y,\nu_i)\}$ are present for all $i=0,\ldots,s$; see Figure~\ref{fig:forcing}~(b).

To prove~(ii) and~(iii), consider a cover~$x\lessdot y$ in~$L_H$.
Then the interval~$[(x,0),(y,\ell)]$ in~$L_G$ contains exactly the elements~$C(x)\cup C(y)$ and its cover graph is an $(\ell,\nu)$-ladder~$L$ for some sequence~$\nu=(\nu_0,\ldots,\nu_s)$ determined by~$G$.
From our auxiliary claim~(a), we obtain the following about the ladder~$L$; see Figure~\ref{fig:forcing}~(b):
If $(x,\nu_i)\equiv (y,\nu_i)$, then we have $(x,\nu_{i-1})\equiv (y,\nu_{i-1})$ and $(x,\nu_{i+1})\equiv (y,\nu_{i+1})$, and furthermore $(x,\nu_i)\equiv (x,\nu_i-1)\equiv \cdots\equiv (x,\nu_{i-1}+1)$ and $(y,\nu_i)\equiv (y,\nu_i+1)\equiv \cdots \equiv (y,\nu_{i+1}-1)$.
Applying this observation exhaustively, we note that either $(x,\nu_i)\not\equiv (y,\nu_i)$ for all $i=0,\ldots,s$, or $(x,\nu_i)\equiv (y,\nu_i)$ for all $i=0,\ldots,s$, and in the latter case for every $i\in\{0,\ldots,\ell\}$ there is some $j\in\{0,\ldots,\ell\}$ with $(x,i)\equiv (y,j)$ and for every $j\in\{0,\ldots,\ell\}$ there is some $i\in\{0,\ldots,\ell\}$ with $(x,i)\equiv (y,j)$.
This proves property~(ii).

To prove~(iii), note that the auxiliary claim~(a) applied to the ladder~$L$ yields that $(x,0)\equiv (x,1)\equiv \cdots\equiv (x,\ell)$ if and only if $(y,0)\equiv (y,1)\equiv \cdots\equiv (y,\ell)$; see Figure~\ref{fig:forcing}~(c).
Consequently, the number of equivalence classes of~$\equiv$ along every chain~$C(x)$, $x\in V$, is either equal to~1 for all~$x\in V$ or strictly larger than~1 for all~$x\in V$, which proves~(iii).

It remains to prove~(iv).
If $\lambda(x)=0$ for all $x\in V$, then we have $L_G\simeq L_H$ and therefore $L_G/{\equiv}\simeq L_H/{\equiv^*}$, so the claim is trivial.
On the other hand, if $\lambda(x)\geq 1$ for all $x\in V$, which is the only other possibility by~(iii), then the cover graph of~$L_G/{\equiv}$ has an induced subgraph that is $\lambda$-suspended over~$L_H/{\equiv^*}$ by~(ii) and Lemma~\ref{lem:cong}.
Using the assumption that the cover graph of~$L_H/{\equiv^*}$ has a Hamiltonian path, applying the first part of Lemma~\ref{lem:susp-ham} proves that the cover graph of~$L_G/{\equiv}$ also has a Hamiltonian path.

This completes the proof.
\end{proof}

\subsection{Proofs of Theorems~\ref{thm:super-ham} and~\ref{thm:super-quotient}}

We are now ready to present the proof of our two main theorems.

\begin{proof}[Proof of Theorem~\ref{thm:super-ham}]
The proof is by induction on the rank~$n\geq 2$ of~$\cH$.
If $n=2$, then the graph of regions~$G(\cH)$ is a cycle of length~$2|\cH|$ and thus clearly has a Hamiltonian cycle of even length.
If $n\geq 3$, then consider the decomposition of~$\cH$ into two subarrangements~$\cH_0$ and~$\cH_1$.
As $\cH_0$ is a supersolvable arrangement of rank~$n-1$, we know by induction that~$G(\cH_0)$ has a Hamiltonian cycle of even length.
Applying Lemma~\ref{lem:susp} and the second part of Lemma~\ref{lem:susp-ham} yields that~$G(\cH)$ has a Hamiltonian cycle of even length.
\end{proof}

\begin{proof}[Proof of Theorem~\ref{thm:super-quotient}]
The proof is by induction on the rank~$n\geq 1$ of~$\cH$.
If $n=1$, then the cover graph of~$L$ is a single edge.
If $n=2$, then the cover graph of~$L$ is a cycle of length~$2|\cH|$.
By Lemma~\ref{lem:cong}~(i), the equivalence classes of~$\equiv$ are intervals, so in the two aforementioned cases the cover graph of $L/{\equiv}$ is either a cycle, a single edge, or a single vertex, hence it has a Hamiltonian path.

If $n\geq 3$, then consider the decomposition of~$\cH$ into two subarrangements~$\cH_0$ and~$\cH_1$.
Let $R_0\in\cR(\cH)$ be a canonical base region, then $\rho(R_0)$ is canonical for~$\cH_0$.
We know that $L':=P(\cH_0,\rho(R_0))=P(G(\cH_0),\rho(R_0))$ and $L=P(\cH,R_0)=P(G(\cH),R_0)$ are lattices by Lemma~\ref{lem:P-lattice}.
Furthermore, we know that $G(\cH)$ is (isomorphic to a graph that is) $\ell$-suspended over~$G(\cH_0)$ for $\ell:=|\cH_1|$ by Lemma~\ref{lem:susp}.

For any~$R\in\cR(\cH)$, we let $p(R)\in\cR(\cH)$ be the minimum of the chain on the fiber~$\varphi(R)$ in~$L$.
Let $\equiv$ be lattice congruence on~$L$.
By Proposition~\ref{prop:susp-cong}~(i) the restriction ${\equiv^*}:=\{(\rho(X),\rho(Y))\mid X=p(X)\text{ and } Y=p(Y)\text{ and } X\equiv Y\}$ is a lattice congruence on~$L'$.
As $\cH_0$ is a supersolvable arrangement of rank~$n-1$, we obtain by induction that the cover graph of~$L'/{\equiv^*}$ has a Hamiltonian path.
Therefore, applying Proposition~\ref{prop:susp-cong}~(iv) yields that the cover graph of~$L$ also has a Hamiltonian path.
This completes the proof.
\end{proof}

\section{Zigzag languages of signed permutations}
\label{sec:zigzag-signed}

In this section we develop a theory of zigzag languages of signed permutations, analogous to the theory for (unsigned) permutations presented in~\cite{MR4391718} and discussed in Section~\ref{sec:zigzag}.
We also show applications to generating pattern-avoiding signed permutations and the corresponding combinatorial objects.
With this paper we hope to encourage further study of pattern-avoiding signed permutations by the pattern-avoidance community.

\subsection{Patterns in permutations and signed permutations}

Recall from Section~\ref{sec:sperm} that we use bars to denote negative signs in a signed permutation.
We write~$B_n$ for the set of signed permutations of~$[n]$ in full notation (length~$2n$), and we will use this representation throughout.
We index the entries in a signed permutation~$\pi\in B_n$ by integers from~$\pmn$, i.e., we consider~$\pi$ as a string~$\pi=(a_{\ol{n}},\ldots,a_{\b{1}},a_1,\ldots,a_n)$, where $a_{\ol{i}}=\ol{a_i}$ holds for all $i\in\pmn$.
Notably, the index~$0$ does not exist.
Furthermore, we write~$S_n$ for the set of (unsigned) permutations of length~$n$.
We write $\varepsilon$ for the empty string, i.e., we have $B_0=\{\varepsilon\}$ and $S_0=\{\varepsilon\}$.
The \defn{identity} signed permutation is~$\id_n:=\ol{n}\cdots \b{2}\b{1}12\cdots n\in B_n$.

We say that a string~$\pi$ of integers \defn{contains a pattern~$\tau\in S_k$}, if there is a subsequence of $k$ (not necessarily consecutive) entries in~$\pi$ in the same relative order as in~$\tau$.
Otherwise we say that~$\pi$ \defn{avoids}~$\tau$.
We apply this definition in both cases when $\pi\in S_n$ is a permutation of length~$n$, or when $\pi\in B_n$ is a signed permutation in full notation, where in the second case the signs of entries are taken into account when comparing relative positions.
For example, $\pi=\b{4}\b{2}31\b{1}\b{3}24\in B_4$ contains the pattern~$\tau=231$, as witnessed by the subsequence~$\b{2}3\b{3}$, or alternatively by the subsequences~$\b{2}1\b{3}$ and~$\b{2}\b{1}\b{3}$.
We write~$S_n(\tau)$ for the set of all permutations from~$S_n$ that avoid~$\tau$, and similarly $B_n(\tau)$ for the set of all signed permutations from~$B_n$ that avoid~$\tau$.

Generalizing the notion of pattern-containment for signed permutations, we also allow underlining or overlining some of entries in a pattern~$\tau\in S_n$, with the interpretation that an underlined entry of the pattern has to match a strictly positive entry in the signed permutation, whereas an overlined entry of the pattern has to match a strictly negative entry in the signed permutation.
We refer to such a pattern as a \defn{marked} (signed) permutation pattern.
For example, in the signed permutation~$\pi$ from before, the first two occurrences of the pattern~$231$ are also occurrences of the pattern~$2\ul{3}1$, whereas the third occurrence is also an occurrence of the pattern~$2\b{3}1$.
Note that there are some trivial equivalences of patterns, in the sense that the set of avoiders are identical.
Specifically, if an entry in a marked pattern is underlined, then all larger entries may also be underlined.
Similarly, if an entry is overlined, then all smaller entries may also be overlined.
For example, the patterns $\ul{2}31$ and~$\ub{2}\ub{3}1$ are equivalent in this sense, and so are the patterns $2\b{3}1$, $\b{2}\b{3}1$, $2\b{3}\b{1}$ and $\b{2}\b{3}\b{1}$.

\subsection{Jumps in signed permutations}

Given a signed permutation~$\pi\in B_n$, a \defn{left hop} of the value~$a_i>0$, $i\neq \ol{n}$, is a transposition of the value~$a_i$ with the value~$b$ directly to its left, i.e., either with $b:=a_{i-1}$ if $i\neq 1$ or $b:=a_{\b{1}}$ if $i=1$, subject to the constraint that~$a_i>b$.
Similarly, a \defn{right hop} of the value~$a_i>0$, $i\neq n$, is a transposition of the value~$a_i$ with the value~$b$ directly to its right, i.e., either with $b:=a_{i+1}$ if $i\neq \b{1}$ or $b:=a_1$ if $i=\b{1}$, subject to the constraint that~$a_i>b$.
A \defn{left jump of~$a_i>0$ by $d$ steps} is sequence of $d$ consecutive left hops of the value~$a_i$.
Similarly, a \defn{right jump of~$a_i>0$ by $d$ steps} is a sequence of $d$ consecutive right hops of the value~$a_i$.
In all these operations the modifications are applied to the permutation in full notation in a symmetric way, so that the property $a_{\ol{i}}=\ol{a_i}$ for all $i\in\pmn$ is maintained.
For a subset $L_n\seq B_n$ and a permutation~$\pi\in L_n$, a jump in~$\pi$ is \defn{minimal} w.r.t.~$L_n$ if no jump of the same value by fewer steps yields a signed permutation from~$L_n$.
For example, if $L_4:=B_4(\ub{2}31)$, then in the permutation $\pi=\b{2}\b{1}43\b{3}\b{4}12\in L_4$ a right jump of the value~4 by 5 steps is minimal, yielding the permutation $\b{4}\b{2}\b{1}3\b{3}124\in L_4$ as each of the intermediate permutations $\b{2}\b{1}34\b{4}\b{3}12$, $\b{2}\b{1}3\b{4}4\b{3}12$, $\b{2}\b{1}\b{4}3\b{3}412$, $\b{2}\b{4}\b{1}3\b{3}142$ contains the marked pattern~$\ub{2}31$.

We propose the following algorithm to generate a set of signed permutations by minimal jumps.
An analogous algorithm for (unsigned) permutations was first presented in~\cite{MR4391718}.

\begin{algo}{Algorithm~J}{Greedy minimal jumps for signed permutations}
This algorithm attempts to greedily generate a set of signed permutations $L_n\seq B_n$ using minimal jumps starting from an initial signed permutation $\pi_0 \in L_n$.
\begin{enumerate}[label={\bfseries S\arabic*.}, leftmargin=8mm, noitemsep, topsep=3pt plus 3pt]
\item{} [Initialize] Visit the initial signed permutation~$\pi_0$.
\item{} [Jump] Generate an unvisited signed permutation from~$L_n$ by performing a minimal jump of the largest possible positive value in the most recently visited signed permutation.
If no such jump exists, or the jump direction is ambiguous, then terminate.
Otherwise visit this signed permutation and repeat~S2.
\end{enumerate}
\end{algo}

\subsection{Zigzag languages of signed permutations}

For any~$\pi\in B_{n-1}$ and integer~$i\in \pmn$ we write $c_i(\pi)\in B_n$ for the signed permutation obtained by inserting the largest value~$n$ and smallest value~$\ol{n}$ at positions~$i$ and~$\ol{i}$, respectively.
In particular, $c_{\ol{n}}(\pi)=n\,\pi\,\ol{n}$ and $c_n(\pi)=\ol{n}\,\pi\,n$.
Conversely, for any $\pi\in B_n$ we write $p(\pi)$ for the permutation in~$B_{n-1}$ obtained by deleting the extremal values~$n$ and~$\ol{n}$.
Clearly, we have $p(c_i(\pi))=\pi$.
For example, for $\pi\in 1\b{2}2\b{1}$ we have
\[ \big(c_{\b{3}}(\pi),c_{\b{2}}(\pi),c_{\b{1}}(\pi),c_1(\pi),c_2(\pi),c_3(\pi)\big)
=({\red 3}1\b{2}2\b{1}{\red\b{3}},1{\red 3}\b{2}2{\red\b{3}}\b{1},1\b{2}{\red 3}{\red\b{3}}2\b{1},1\b{2}{\red\b{3}}{\red 3}2\b{1},1{\red\b{3}}\b{2}2{\red 3}\b{1},{\red\b{3}}1\b{2}2\b{1}{\red 3}). \]

A set $L_n\seq B_n$ is a \defn{zigzag language}, if either $n=0$ and $L_0=\{\varepsilon\}$, or if $n\geq 1$ and $L_{n-1}:=\{p(\pi)\mid \pi \in L_n\}$ is a zigzag language such that for every $\pi\in L_{n-1}$ we have that $c_{\ol{n}}(\pi)=n\,\pi\,\ol{n}\in L_n$ and $c_n(\pi)=\ol{n}\,\pi\,n\in L_n$.

\begin{theorem}
\label{thm:zigzag}
Given any zigzag language of signed permutations~$L_n$ and initial permutation~$\pi_0=\id_n$, Algorithm~J visits every signed permutation from~$L_n$ exactly once.
\end{theorem}

The proof is analogous to the proof of Theorem~1 in~\cite{MR4391718}.

\begin{proof}
Given a zigzag language~$L_n$, we define a sequence~$J(L_n)$ of all signed permutations from~$L_n$, and we prove that Algorithm~J generates the permutations of~$L_n$ exactly in this order.
For any $\pi\in L_{n-1}$ we let $\rvec{c}(\pi)$ be the sequence of all $c_i(\pi)\in L_n$ for $i=\ol{n},\ol{n-1},\ldots,\b{1},1,\ldots,n$, starting with~$c_{\ol{n}}(\pi)$ and ending with~$c_n(\pi)$, and we let $\lvec{c}(\pi)$ denote the reverse sequence, i.e., it starts with~$c_n(\pi)$ and ends with~$c_{\ol{n}}(\pi)$.
In words, those sequences are obtained by inserting the new largest value~$n$ from left to right, or from right to left, respectively, into~$\pi$, in all possible positions that yield a signed permutation from~$L_n$, skipping the positions that yield a permutation that is not in~$L_n$.
The sequence~$J(L_n)$ is defined recursively as follows:
If $n=0$ then we define $J(L_0):=\varepsilon$, and if $n\geq 1$ then we consider the finite sequence $J(L_{n-1})=:(\pi_1,\pi_2,\ldots)$ and define
\begin{equation*}
J(L_n):=\lvec{c}(\pi_1),\rvec{c}(\pi_2),\lvec{c}(\pi_3),\rvec{c}(\pi_4),\ldots,
\end{equation*}
i.e., this sequence is obtained from the previous sequence by inserting the new largest value~$n$ in all possible positions alternatingly from right to left, or from left to right.

A straightforward induction shows that Algorithm~J generates the signed permutations in~$L_n$ precisely in the order~$J(L_n)$.
We omit the details.
The crucial observation is that a minimal jump in a signed permutation~$\pi\in L_{n-1}$ is also a minimal jump in the two permutations~$c_{\ol{n}}(\pi)=n\,\pi\,\ol{n}$ and $c_n(\pi)=\ol{n}\,\pi\,n$, because the values~$n$ and~$\ol{n}$ are inserted at the boundaries.
\end{proof}

\begin{remark}
\label{rem:cyclic}
It is easy to see that in the ordering of signed permutations~$J(L_n)$, the first and last permutation differ in a minimal jump if and only if $|L_i|$ is even for all $i=1,\ldots,n-1$.
Consequently, if this condition holds, then the Gray code produced by Algorithm~J is actually cyclic.
\end{remark}

The following lemma provides easily checkable sufficient conditions on a signed permutation pattern~$\tau$ so that $B_n(\tau)$ is a zigzag language.

We say that an infinite sequence of sets~$L_0,L_1,\ldots$ with $L_i\seq B_i$ for all $i\geq 0$ is \defn{hereditary} if $L_{i-1}=p(L_i)$ holds for all $i\geq 1$.
We say that a signed permutation pattern~$\tau$ is \defn{tame}, if $B_n(\tau)$, $n\geq 0$, is a hereditary sequence of zigzag languages.

\begin{lemma}
\label{lem:tame}
If in a marked pattern~$\tau\in S_k$, the smallest value~1 is underlined or not at the leftmost or rightmost position, and the largest value~$k$ is overlined or not at the leftmost or rightmost position, then $\tau$ is tame.
In particular, if $\tau\in S_k$ has neither the smallest value~1 nor the largest value~$k$ at the leftmost or rightmost position, then it is tame.
\end{lemma}

Four of the signed permutation patterns shown in Table~\ref{tab:pat} classify as tame by Lemma~\ref{lem:tame}.

\begin{proof}
Consider a marked pattern~$\tau\in S_k$ satisfying the conditions of the lemma.
We prove that $B_n(\tau)$, $n\geq 0$, is a hereditary sequence of zigzag languages by induction on~$n$.
The induction basis~$n=0$ is trivial, as $B_0(\tau)=\{\varepsilon\}$.
For the induction step, we assume that $B_{n-1}(\tau)$ is a zigzag language.
If $\pi\in B_n$ avoids~$\tau$, then clearly $p(\pi)\in B_{n-1}$ also avoids~$\tau$, as removing the values~$n$ and~$\ol{n}$ cannot create an occurrence of the pattern~$\tau$.
If follows that $B_{n-1}(\tau)\supseteq p(B_n(\tau))$.
We next show that if $\pi\in B_{n-1}$ avoids~$\tau$, then both $c_{\ol{n}}(\pi)=n\,\pi\,\ol{n}$ and~$c_n(\pi)=\ol{n}\,\pi\,n$ avoid~$\tau$.
As $p(c_{\ol{n}}(\pi))=\pi$ and $p(c_n(\pi))=\pi$, this proves that $B_{n-1}(\tau)=p(B_n(\tau))$ and it also establishes the zigzag property.
Indeed, any occurrence of~$\tau$ in one of the signed permutations $c_{\ol{n}}(\pi)$ and $c_n(\pi)$ must contain~$\ol{n}$ or~$n$, and those entries must match the smallest and largest values in~$\tau$ with the correct sign, respectively, which must be at the leftmost or rightmost position in~$\tau$.
This, however, is ruled out by the assumptions on~$\tau$ stated in the lemma.
\end{proof}

\newcommand{\dittotikz}{%
    \tikz{
        \draw [line width=0.12ex] (-0.2ex,0) -- +(0,0.8ex)
            (0.2ex,0) -- +(0,0.8ex);
        \draw [line width=0.08ex] (-0.6ex,0.4ex) -- +(-1.2em,0)
            (0.6ex,0.4ex) -- +(1.2em,0);
    }%
}
\begin{table}[h!]
\caption{Various signed permutation patterns~$\tau$ and corresponding counting sequences and combinatorial objects.
The third column indicates whether Algorithm~J succeeds to generate the set of all pattern-avoiding signed permutations.
A pattern being tame is sufficient for this property by Theorem~\ref{thm:zigzag}.}
\label{tab:pat}
\makebox[0cm]{ % artificial box to center the picture
\begin{tabular}{llllll}
$\tau$ & tame & Alg~J & $|B_n(\tau)|, n=0,1,\ldots,6$ & OEIS & combinatorial objects \\ \hline
$2\ub{1}$        & no  &     & $1,2,4,8,16,32,64$        & \OEIS{A000079} & binary strings \\
$\ub{2}\b{1}$    & no  & yes & $1,1,2,6,24,120,720$      & \OEIS{A000142} & permutations \textrightarrow{} Sec.~\ref{sec:pat2} \\
\hline
$23\ub{1}$       & yes & yes & $1,2,8,40,224,1344,8448$  & \OEIS{A151374}, \OEIS{A052701} & two-colored binary trees \textrightarrow{} Sec.~\ref{sec:pat1} \\

$213$            & no  & yes & $1,2,4,8,16,32,64$        & \OEIS{A000079} & binary strings \textrightarrow{} Sec.~\ref{sec:pat3} \\
$\ub{2}31$       & no  & yes & $1,2,6,20,70,252,924$     & \OEIS{A000984} & symmetric triangulations \textrightarrow{} Sec.~\ref{sec:pat4} \\
$\b{2}31$        & no  &     &  \dittotikz               & \dittotikz & \\

$2\ub{3}\b{1}$   & no  &     & $1,2,4,10,34,154,874$     & \OEIS{A003422} & \\
$\b{2}\ub{3}1$   & no  &     & $1,2,6,22,94,462,2606$    & \OEIS{A193763} & \\

$23\b{1}$        & no  &     & $1,2,4,9,23,65,197$       & \OEIS{A014137} & \\
$2\ub{3}1$       & no  &     & \dittotikz                & \dittotikz     & \\
\hline
$3142$           & yes & yes & $1,2,7,32,169,974,5947$   & \OEIS{A115197} & \\
$3412$           & yes & yes & $1,2,7,33,183,1118,7281$  & \OEIS{A086618} & \\
$\ub{3}41\b{2}$  & yes & yes & $1,2,7,34,209,1546,13327$ & \OEIS{A002720} & \\
\end{tabular}
}
\end{table}

\begin{remark}
So far we have only considered avoidance of a single pattern.
However, the theory generalizes straightforwardly to avoiding multiple patterns.
This is based on the observation that the intersection of two zigzag languages is again a zigzag language (also the hereditary property is preserved).
Consequently, if we have tame patterns~$\tau_1,\ldots,\tau_\ell$, then the set of signed permutations avoiding each of~$\tau_1,\ldots,\tau_\ell$ is also a zigzag language that can be generated by Algorithm~J.
\end{remark}

Note that Theorem~\ref{thm:zigzag} provides a sufficient condition for Algorithm~J to generate a set of signed permutations~$L_n\seq B_n$ exhaustively, namely to be a zigzag language.
To close this section, we consider one of the tame patterns mentioned in Table~\ref{tab:pat} and present the resulting listing of combinatorial objects obtained from applying Algorithm~J.

\subsubsection{The tame pattern~$23\protect\ub{1}$}
\label{sec:pat1}

Let $\tau:=23\protect\ub{1}$.
We present a bijection between signed permutations avoiding the pattern~$\tau$ and binary trees whose vertices are colored in one of two colors, red or blue, say.
We write $\cT_n$ for the sets of binary trees with $n$ vertices, and $\cT_n'$ for the set of two-colored binary trees with $n$ vertices.

\begin{theorem}
\label{thm:pat1}
Let $\tau:=23\ub{1}$.
There is a bijection~$g:\cT_n'\rightarrow B_n(\tau)$, and therefore $|B_n(\tau)|=2^n C_n$, where $C_n:=\frac{1}{n+1}\binom{2n}{n}$ are the Catalan numbers.
\end{theorem}

\begin{proof}
The bijection~$g$ uses as a building block the following well-known bijection $f:\cT_n\rightarrow S_n(231)$ between binary trees and 231-avoiding permutations; see Figure~\ref{fig:bij-basic}.
We consider each tree~$T\in\cT_n$ as a binary search tree, with the vertices labeled bijectively with integers~$1,\ldots,n$ such that the label of every vertex is larger than all labels in the left subtree and smaller than all labels in the right subtree.
Given a tree~$T$ with root~$r$ and left and right subtrees~$T^\uL$ and~$T^\uR$, we define $f(T):=(r,f(T^\uL),f(T^\uR))$, i.e., we first write the label of the root, then recursively all labels in the left subtree, then recursively all labels in the right subtree.
One can easily check that this mapping~$f$ is indeed a bijection between~$\cT_n$ and~$S_n(231)$, and that it induces a one-to-one correspondence between tree rotations and minimal jumps in 231-avoiding permutations (for details, see \cite[Sec.~3.3]{MR4391718}).

To construct the desired bijection~$g:\cT_n'\rightarrow B_n(\tau)$, note that $\tau=23\ub{1}=\ub{2}\ub{3}\ub{1}$, i.e., when searching for occurrences of this pattern in a signed permutation in full notation, one only needs to consider the positive entries, whereas all negative entries can be completely ignored.
Hence, the mapping~$g$ can be defined as follows:
Given a two-colored tree~$T'\in\cT_n'$, we first consider the underlying uncolored tree~$T\in\cT_n$, and construct the permutation $\pi:=f(T)\in S_n(231)$.
In the second step, we define which $n$ out of the $2n$ entries of~$g(T')$ receive a negative sign.
Specifically, the entry at position~$i\in[n]$ of~$g(T')$ receives a negative sign if vertex~$i$ is colored red in~$T'$, and a positive sign if vertex~$i$ is colored blue.
Note that in the first case this forces the entry at position~$\ol{i}$ to have a positive sign, and in the second case this forces entry at position~$\ol{i}$ to have a negative sign.
In the third step, we fill the string~$\pi$ into the~$n$ positions of~$g(T')$ with a positive sign from left to right, which defines the signed permutation~$g(T')$ completely (the values of the negatively signed entries are forced).
It can be checked directly that this mapping~$g$ is indeed a bijection between~$\cT_n'$ and~$B_n(\tau)$.

It is known that $|\cT_n|=C_n$, and since each tree with $n$ vertices can be colored in $2^n$ ways, this proves the claimed counting formula.
\end{proof}

As $\tau$ is tame by Lemma~\ref{lem:tame}, we may use Algorithm~J to generate~$B_n(\tau)$; see Figure~\ref{fig:trees}.
The resulting Gray code is cyclic by Remark~\ref{rem:cyclic}.
The minimal jumps performed by Algorithm~J translate to the following three types of operations on the two-colored trees~$\cT_n'$:
Vertex~1 switches its color; vertices~$i$ and~$i+1$ exchange their colors; a tree rotation combined with a cyclic rotation of colors among an interval of vertices.

\begin{figure}
\makebox[0cm]{ % artificial box to center the picture
\includegraphics[scale=0.7]{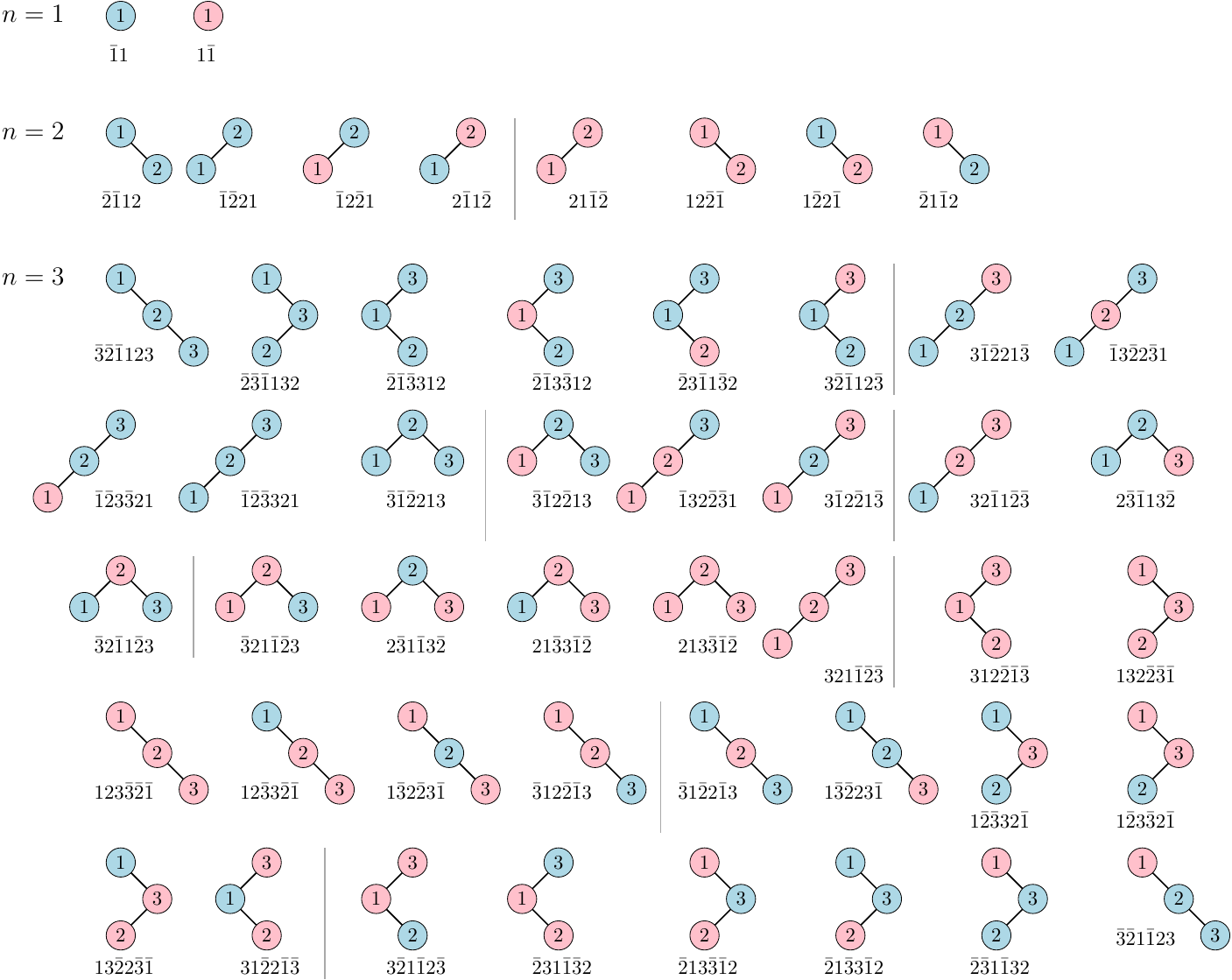}
}
\caption{Listings of signed permutations~$B_n(\tau)$ for $n=1,2,3$ for the tame pattern~$\tau=23\protect\ub{1}$ generated by Algorithm~J and corresponding two-colored binary trees.}
\label{fig:trees}
\end{figure}

\section{Three exceptional patterns}
\label{sec:except}

Interestingly, we found three patterns~$\tau$ that are \emph{not} tame and for which $B_n(\tau)$ is \emph{not} a zigzag language (i.e., Theorem~\ref{thm:zigzag} does not apply), and yet Algorithm~J succeeds in generating the set~$B_n(\tau)$ exhaustively.
We were not able to treat those exceptional patterns in a unified way, and we were also not able to find more such exceptional patterns for which Algorithm~J would work.
Consequently, in this section we describe each of these patterns separately.
The last one is particularly important, as it corresponds to the type~$B$ associahedron.

\subsection{The pattern~\texorpdfstring{$\protect\ub{2}\b{1}$}{21} and (unsigned) permutations}
\label{sec:pat2}

Let $\tau:=\ub{2}\b{1}$.
Signed permutations avoiding this pattern have all negative entries on the set of positions~$[\ol{n}]$, and all positive entries on the set of positions~$[n]$, and no other constraints, i.e., they are in bijection with (unsigned) permutations.
The pattern~$\tau$ is not tame and~$B_n(\tau)$ is not a zigzag language, and yet Algorithm~J succeeds in generating the set~$B_n(\tau)$ when initialized with $\pi_0:=\id_n$, and the corresponding ordering of permutations is the Steinhaus-Johnson-Trotter ordering.

\subsection{The pattern~$213$ and binary strings}
\label{sec:pat3}

Let $\tau:=213$.
We present a bijection between signed permutations avoiding the pattern~$\tau$ and binary strings of fixed length.

\begin{theorem}
\label{thm:pat3}
Let $\tau:=213$.
There is a bijection~$h:B_n(\tau)\rightarrow\{0,1\}^n$, and therefore $|B_n(\tau)|=2^n$.
\end{theorem}

\begin{proof}
The following properties can be easily verified by induction.
The structure of a signed permutation~$\pi\in B_{n-1}(\tau)$ is either
\[ \pi=\id_{n-1} \quad (\text{case~1}), \]
or for some $k\in\{0,1,\ldots,n-2\}$ we have
\[ \pi=(n-k-1,n-k,\ldots,n-2,n-1,\pi',\ol{n-1},\ol{n-2},\ldots,\ol{n-k-1}) \quad (\text{case~2}) \]
with $\pi'\in B_{n-k-2}(\tau)$.
For every such permutation~$\pi$ there are exactly two indices~$i\in\pmn$ such that $c_i(\pi)\in B_n(\tau)$, namely, in case~1
\begin{alignat*}{2}
c_{\ol{n}}(\pi)&=(n\,\pi,\ol{n}) \quad &&(\rightarrow \text{case~2}), \\
c_n(\pi)&=(\ol{n},\pi,n)=(\ol{n},\id_{n-1},n)=\id_n \quad &&(\rightarrow \text{case~1});
\end{alignat*}
and in case~2
\begin{alignat*}{2}
c_{\ol{n}}(\pi)&=(n,\pi,\ol{n}) \quad &&(\rightarrow \text{case~2}), \\
c_{\ol{n-k-1}}(\pi)&=(n-k-1,\ldots,n-1,n,\pi',\ol{n},\ol{n-1},\ldots,\ol{n-k-1}) \quad &&(\rightarrow \text{case~2}). \\
\end{alignat*}
In other words, in both cases the new largest value~$n$ can be inserted at the leftmost position, and directly after the maximum increasing prefix.

The bijection~$h:B_n(\tau)\rightarrow\{0,1\}^n$ can thus be defined as follows:
Given~$\pi\in B_n(\tau)$ and $i\in[n]$, consider $\pi':=p^{n-i}(\pi)\in B_i(\tau)$, and if the largest value~$i$ in~$\pi'$ is at the leftmost position, then the $i$th bit of~$h(\pi)$ is defined to be~1, and 0 otherwise.
\end{proof}

The pattern~$\tau$ is not tame and~$B_n(\tau)$ is not a zigzag language, and yet Algorithm~J succeeds in generating the set~$B_n(\tau)$ when initialized with $\pi_0:=\id_n$; see Figure~\ref{fig:bits}.
The resulting ordering of binary strings is the classical binary reflected Gray code, i.e., minimal jumps performed by the algorithm translate to single bitflips.

\begin{figure}
\makebox[0cm]{ % artificial box to center the picture
\includegraphics{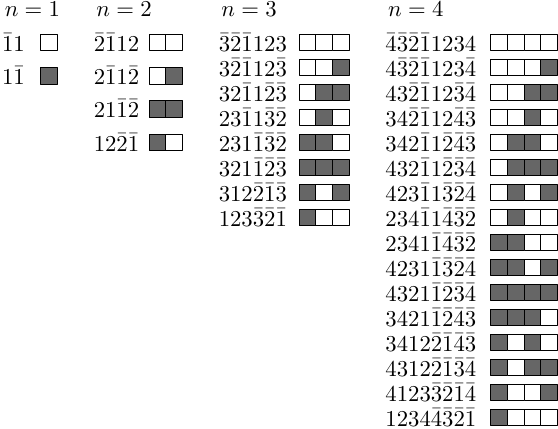}
}
\caption{Listings of signed permutations~$B_n(\tau)$ for $n=1,2,3,4$ for the non-tame pattern~$\tau=213$ generated by Algorithm~J and corresponding binary strings.
Each 0-bit is visualized as a white square, and each 1-bit as a black square.
The resulting ordering is the classical binary reflected Gray code.}
\label{fig:bits}
\end{figure}

However, there is no one-to-one correspondence between minimal jumps in $\tau$-avoiding permutations and bitflips in the hypercube.
In particular, the initial signed permutation~$\id_n$ admits only one possible minimal jump (the value~$n$ jumping all the way to the other side).
Consequently, the listing of $\tau$-avoiding signed permutations is not cyclic under minimal jumps, even though the binary reflected Gray code is cyclic, i.e., the last binary string differs only in a single bit from the first string.

\subsection{The pattern~\texorpdfstring{$\protect\ub{2}31$}{231} and point-symmetric triangulations}
\label{sec:pat4}

For the rest of this section, we consider the permutation pattern~$\tau:=\ub{2}31$.
This pattern is not tame and~$B_n(\tau)$ is not a zigzag language, and yet Algorithm~J can be used to generate the set~$B_n(\tau)$ exhaustively.
In order to prove this (Theorem~\ref{thm:2b31} below), we precisely describe the structure of $\tau$-avoiding signed permutations.
This also helps us to establish a bijection between signed permutations in~$B_n(\tau)$ and point-symmetric triangulations of a convex~$(2n+2)$-gon, with the property that minimal jumps in the permutations are in one-to-one correspondence to flips in the triangulations (see Section~\ref{sec:bij} and Theorem~\ref{thm:h-bij} therein).
As a consequence, Algorithm~J not only lists all signed permutation of~$B_n(\tau)$, but the listing actually corresponds to a Hamiltonian path on the $B$-associahedron \cite{MR1979780,MR2258260}; see Figures~\ref{fig:congB}, \ref{fig:triang} and~\ref{fig:gc-triang}.

The bijection is essentially the same as the one described in \cite[Sec.~7]{MR2258260}.
However, we describe it purely combinatorially, without reference to lattice-theoretic concepts.
Our goal is to give a simple, explicit description that can easily be exploited for the purpose of Gray codes.

\subsubsection{Triangulations}
\label{sec:triang}

We consider a set of $n+2$ points in the plane, placed equidistantly on the unit circle and labeled $0,1,\ldots,n,\b{0}$ in counterclockwise (ccw) order.
We write $\Delta_n$ for the set of triangulations on this point set.
A \defn{flip} in a triangulation removes one of its inner edges and replaces it by the other diagonal of the resulting empty quadrilateral.
It is well-known that there is a bijection between $\Delta_n$ and~$S_n(231)$, and that both families of objects are counted by the Catalan numbers, i.e., we have $|\Delta_n|=|S_n(231)|=C_n=\frac{1}{n+1}\binom{2n}{n}$.

The bijection~$g:\Delta_n\rightarrow S_n(231)$ will be used as a building block later, so we define it in the following; see Figure~\ref{fig:bij-basic}:
Given a triangulation~$T\in\Delta_n$, we consider the triangle in~$T$ seen through the edge~$0\b{0}$.
To obtain~$g(T)$, we process each of the $n$ the triangles one by one, recording one vertex of each triangle.
Specifically, for every (ccw) triangle~$abc$ seen through the edge~$ac$, we record the label of the third vertex~$b$, and we then recursively compute~$g(T^\uL)$ and~$g(T^\uR)$, where $T^\uL$ and~$T^\uR$ are the subtriangulations of~$T$ seen through the left edge~$ab$ and the right edge~$bc$, respectively.
Therefore, if $b$ is the third vertex of the triangle in~$T$ containing the edge~$0\b{0}$, then $g(T):=(b,g(T^\uL),g(T^\uR))$.

\begin{lemma}[\cite{MR4391718}]
\label{lem:g-bij}
The mapping~$g:\Delta_n\rightarrow S_n(231)$ is a bijection that induces a one-to-one correspondence between flips in triangulations and minimal jumps in 231-avoiding permutations.
\end{lemma}

Note that the composition $f^{-1}\circ g:\Delta_n\rightarrow \cT_n$, where $f:\cT_n\rightarrow S_n(231)$ is the bijection described in the proof of Theorem~\ref{thm:pat1}, is the geometric duality mapping that assigns to every triangulation its corresponding dual tree (seen through the edge~$0\b{0}$).
Furthermore, note that flips in triangulations are in one-to-one correspondence with tree rotations under this bijection.

\begin{figure}[h!]
\includegraphics[page=4]{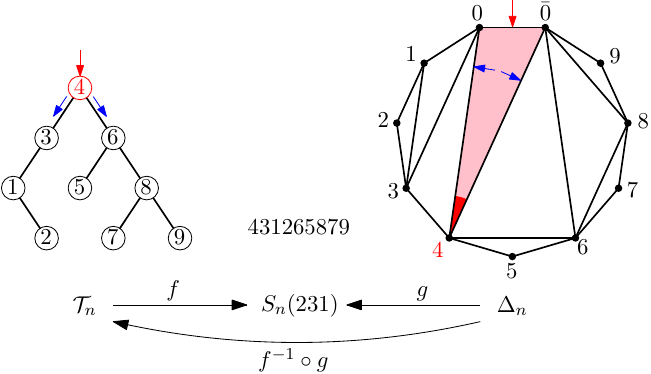}
\caption{Bijections between binary trees, 231-avoiding permutations and triangulations of a convex polygon.}
\label{fig:bij-basic}
\end{figure}

\subsubsection{Point-symmetric triangulations}

We consider a set of $2n+2$ points in the plane, placed equidistantly on the unit circle and labeled $0,1,\ldots,n,\b{0},\b{1},\ldots,\ol{n}$ in ccw order; see Figure~\ref{fig:symm-triang}.
Note that the points~$i$ and~$\ol{i}$ are opposite to each other on the circle.
A \defn{symmetric triangulation} is a triangulation of this point set that is point-symmetric w.r.t.\ the origin, i.e., whenever an edge~$uv$ is present, then the edge~$\ol{u}\,\ol{v}$ must also be present.
We write $\Delta_n^*$ for the set of all symmetric triangulations on this point set.
Note that every such triangulation has a unique edge~$u\ol{u}$ through the origin, which we call the \defn{central} edge.
A \defn{flip} in a symmetric triangulation either consists of a single flip of the central edge, or a pair of symmetric flips not involving the central edge.

\begin{figure}
\makebox[0cm]{ % artificial box to center the picture
\includegraphics[page=1]{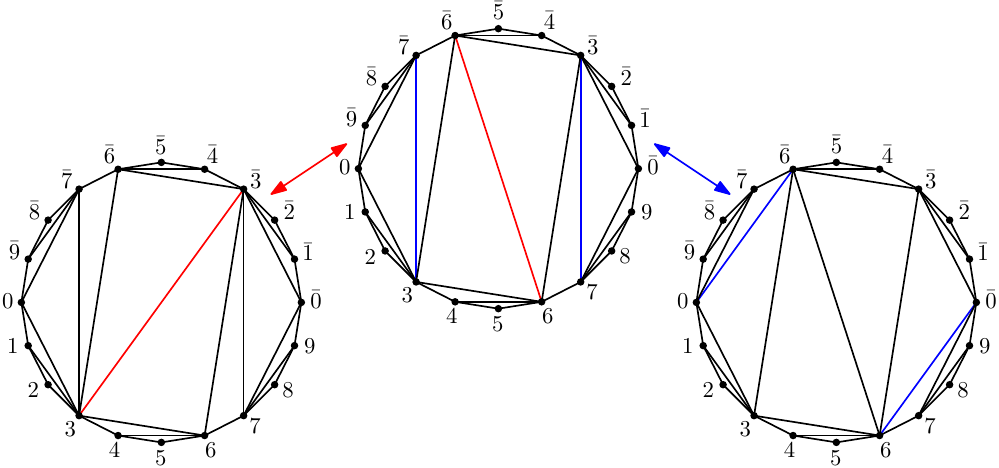}
}
\caption{A symmetric triangulation, and the two possible flip types.}
\label{fig:symm-triang}
\end{figure}

Observe that a symmetric triangulation is uniquely determined by the subtriangulation on one side of the central edge, and that there $n+1$ different central edges $i\ol{i}$, $i=0,\ldots,n$.
It follows that $|\Delta_n^*|=(n+1)|\Delta_n|=(n+1)C_n=\binom{2n}{n}$, i.e., symmetric triangulations are counted by the central binomial coefficients.

\subsubsection{Structure of \texorpdfstring{$\tau$}{tau}-avoiding signed permutations}

Our first lemma, illustrated in Figure~\ref{fig:matrix231}, describes the structure of 231-avoiding (unsigned) permutations.
In the following, for a string~$x=(x_1,\ldots,x_s)$, we write $|x|:=s$ for its length.

\begin{lemma}
\label{lem:struc-231}
For a permutation $\pi=(a_1,\ldots,a_n)\in S_n$, let $s\in[n]$ be the index of the value~1 in~$\pi$, and define $\pi^\uL=(x_s,\ldots,x_1):=(a_1,\ldots,a_s)$ and $\pi^\uR:=(a_{s+1},\ldots,a_n)$.
Then $\pi=(\pi^\uL,\pi^\uR)$ avoids~$231$ if and only if the following conditions hold:
\begin{enumerate}[label=(\roman*),leftmargin=8mm]
\item We have $x_s>x_{s-1}>\cdots>x_1=1$, i.e., the values to the left of~1 appear in decreasing order.
\item There is a decomposition $\pi^\uR=(\pi^\uR_1,\ldots,\pi^\uR_s)$ where $x_1<\pi^\uR_1<x_2<\pi^\uR_2<\cdots<x_{s-1}<\pi^\uR_{s-1}<x_s<\pi^\uR_s$, such that each of the subpermutations~$\pi^\uR_j$, $j=1,\ldots,s$, avoids 231.
\end{enumerate}
\end{lemma}

The inequality $x_j<\pi^\uR_j<x_{j+1}$ means that all values of the subpermutation $\pi^\uR_j$ are larger than $x_j$ and smaller than~$x_{j+1}$.

\begin{figure}[h!]
\includegraphics[page=1]{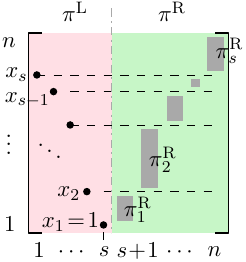}
\caption{Illustration of the structure of a $231$-avoiding permutation as captured by Lemma~\ref{lem:struc-231}.
Specifically, the figure shows the corresponding permutation matrix, i.e., a point in column~$i$ and row~$j$ indicates that the permutation has the value~$j$ at position~$i$.}
\label{fig:matrix231}
\end{figure}

\begin{proof}
We first assume that $\pi\in S_n$ avoids~231.
To prove~(i), note that if there are indices $i<j<s$ with $a_i<a_j$, then $a_ia_ja_s=a_ia_j1$ is an occurrence of the pattern~231, a contradiction.
To prove~(ii), we first establish the following auxiliary claim:
For any $j\in[s]$, we have that in~$\pi^\uR$ all entries smaller than~$x_j$ are to the left of all entries larger than~$x_j$.
Indeed, if there is a $j\in[s]$ and indices $s+1\leq k<\ell\leq n$ such that $a_k>x_j>a_\ell$, then $x_ja_ka_\ell$ is an occurrence of~231, a contradiction.
This proves the decomposition of~$\pi$ into subpermutations~$\pi^\uR_j$, $j\in[s]$, claimed in~(ii).
Clearly, each of these subpermutations must also avoid 231.

Now assume that $\pi\in S_n$ satisfies (i) and (ii).
If $a_ia_ja_k$, $i<j<k$, is an occurrence of~231 in~$\pi$, then we have $a_k<a_i<a_j$, and we distinguish three possible cases:
(1)~The values~$a_i$ and~$a_j$ are both in~$\pi^\uL$: this is impossible because of~(i).
(2)~The value~$a_i$ is in~$\pi^\uL$ and $a_j,a_k$ are in~$\pi^\uR$: this is impossible because of the decomposition into subpermutations stated in~(ii).
(3)~The values~$a_i,a_j,a_k$ are all in~$\pi^\uR$: because of the inequalities $\pi^\uR_1<\pi^\uR_2<\cdots<\pi^\uR_{s-1}<\pi^\uR_s$ and $a_i,a_j>a_k$ the values~$a_i$ and~$a_j$ must be belong to the same subpermutation as~$a_k$, i.e., $a_i,a_j,a_k$ are all in~$\pi^\uR_\hj$ for some $\hj\in[s]$, which is impossible because of the 231-avoidance condition stated in~(ii).
\end{proof}

The next lemma, illustrated in Figure~\ref{fig:matrix2b31}~(a), describes the structure of $\tau$-avoiding signed permutations.
For any string~$x=(a_1,\ldots,a_k)$ of integers, we define $x':=(\ol{a_k},\ldots,\ol{a_1})$, i.e., this operation reverses the string and complements all signs.

Let $\pi=(a_{\ol{n}},\ldots,a_{\b{1}},a_1,\ldots,a_n)\in B_n$.
We define $N(\pi):=\{i\in[n] \mid a_i<0\}$ and
\begin{subequations}
\label{eq:stpi}
\begin{equation}
\label{eq:spi}
s_\pi:=\begin{cases} \max N(\pi) & \text{if } N(\pi)\neq \emptyset, \\
                     0 & \text{else}.
   \end{cases}
\end{equation}
In words, $s_\pi$ is the largest positive index for which the corresponding entry of~$\pi$ is negative, or 0 if no such index exists.
Furthermore, we let $p_\pi\in\pmn$ be the position of the value~1 in~$\pi$.
Lastly, we define
\begin{equation}
\label{eq:tpi}
t_\pi:=\max\{s_\pi,p_\pi\}.
\end{equation}
\end{subequations}

\begin{figure}[h!]
\begin{tabular}{cc}
\includegraphics[page=2]{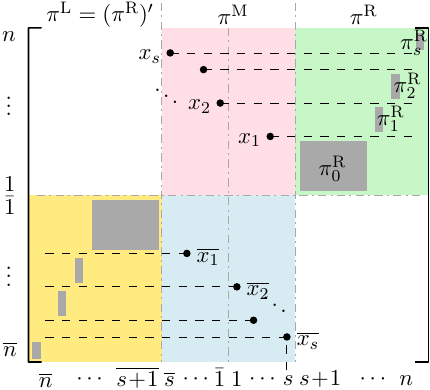} &
\includegraphics[page=3]{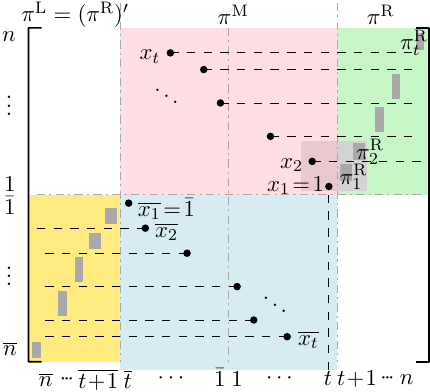} \\
(a) & (b)
\end{tabular}
\caption{Illustration of the structure of a $\tau$-avoiding signed permutation as described by (a) Lemma~\ref{lem:struc-2b31} and (b) Lemma~\ref{lem:struc-2b31p}.}
\label{fig:matrix2b31}
\end{figure}

\begin{lemma}
\label{lem:struc-2b31}
For a signed permutation $\pi=(a_{\ol{n}},\ldots,a_{\b{1}},a_1,\ldots,a_n)\in B_n$, let $s:=s_\pi$ be as defined in~\eqref{eq:spi}.
Define $\pi^\uM:=(a_{\ol{s}},\ldots,a_{\b{1}},a_1,\ldots,a_s)$, $\pi^\uR:=(a_{s+1},\ldots,a_n)$ and $\pi^\uL:=(a_{\ol{n}},\ldots,a_{\ol{s+1}})=(\pi^\uR)'$, and let $x_s,x_{s-1},\ldots,x_1$ denote the sequence of positive entries of~$\pi^\uM$ from left to right.
Then $\pi=(\pi^\uL,\pi^\uM,\pi^\uR)$ avoids~$\tau$ if and only if the following conditions hold:
\begin{enumerate}[label=(\roman*),leftmargin=8mm]
\item We have $x_s>x_{s-1}>\cdots>x_1$, i.e., the positive entries of~$\pi^\uM$ appear in decreasing order.
\item There is a decomposition $\pi^\uR=(\pi^\uR_0,\pi^\uR_1,\ldots,\pi^\uR_s)$ where $\pi^\uR_0<x_1<\pi^\uR_1<x_2<\pi^\uR_2<\cdots<x_{s-1}<\pi^\uR_{s-1}<x_s<\pi^\uR_s$, such that each of the subpermutations~$\pi^\uR_j$, $j=0,\ldots,s$, avoids 231.
\end{enumerate}
\end{lemma}

We emphasize that the pattern~$231$ mentioned in~(ii) has no underscore.
Note that~$|\pi^\uM|=2s_\pi$ and $|\pi^\uR|=|\pi^\uL|=n-s_\pi$.

\begin{proof}
We first assume that $\pi\in B_n$ avoids~$\tau$.
To prove~(i), note that if there are indices $i<j$ with $x_i<x_j$, then $x_ix_ja_s=x_ix_j\ol{x_s}$ is an occurrence of the pattern~$\tau$, a contradiction.
To prove~(ii), we first establish the following auxiliary claim:
For any $j\in[s]$, we have that in~$\pi^\uR$ all entries smaller than~$x_j$ are to the left of all entries larger than~$x_j$.
Indeed, if there is a $j\in[s]$ and indices $s+1\leq k<\ell\leq n$ such that $a_k>x_j>a_\ell$, then $x_ja_ka_\ell$ is an occurrence of~$\tau$, a contradiction.
This proves the decomposition of~$\pi$ into subpermutations~$\pi^\uR_j$, $j=0,\ldots,s$, claimed in~(ii).
As all values in~$\pi^\uR$ are positive by definition, any occurrence of~231 would also be an occurrence of~$\tau$, so each of the subpermutations~$\pi^\uR_j$, $j=0,\ldots,s$, must avoid~231.

We now assume that~$\pi\in B_n$ satisfies~(i) and~(ii).
If $a_ia_ja_k$, $i<j<k$, is an occurrence of~$\tau$ in~$\pi$, then we have $a_k<a_i<a_j$ and $a_i>0$, so $a_i,a_j,a_k$ are all in~$\pi^\uM$ or~$\pi^\uR$.
We only need to consider three possible cases:
(1)~The values~$a_i$ and~$a_j$ are both in~$\pi^\uM$: this is impossible because of~(i).
(2)~The value~$a_i$ is in~$\pi^\uM$ and $a_j,a_k$ are in~$\pi^\uR$: this is impossible because of the decomposition into subpermutations stated in~(ii).
(3)~The values~$a_i,a_j,a_k$ are all in~$\pi^\uR$: because of the inequalities $\pi^\uR_0<\pi^\uR_1<\pi^\uR_2<\cdots<\pi^\uR_{s-1}<\pi^\uR_s$ and $a_i,a_j>a_k$ the values $a_i$ and~$a_j$ must belong to the same subpermutation as~$a_k$, i.e., $a_i,a_j,a_k$ are all in $\pi^\uR_\hj$ for some $\hj\in\{0,\ldots,s\}$, which is impossible because of the 231-avoidance condition stated in~(ii).
\end{proof}

We also need the following alternative structural description of $\tau$-avoiding signed permutations; see Figure~\ref{fig:matrix2b31}~(b).

\begin{lemma}
\label{lem:struc-2b31p}
For a signed permutation $\pi=(a_{\ol{n}},\ldots,a_{\b{1}},a_1,\ldots,a_n)\in B_n$, let $s:=s_\pi$, $p:=p_\pi$ and~$t:=t_\pi$ be as defined in~\eqref{eq:stpi}.
Define $\pi^\uM:=(a_{\ol{t}},\ldots,a_{\b{1}},a_1,\ldots,a_t)$, $\pi^\uR:=(a_{t+1},\ldots,a_n)$ and $\pi^\uL:=(a_{\ol{n}},\ldots,a_{\ol{t+1}})=(\pi^\uR)'$, and let $x_t,x_{t-1},\ldots,x_1$ denote the sequence of positive entries of~$\pi^\uM$ from left to right.
Then $\pi=(\pi^\uL,\pi^\uM,\pi^\uR)$ avoids~$\tau$ if and only if the following conditions hold:
\begin{enumerate}[label=(\roman*),leftmargin=8mm]
\item We have $x_t>x_{t-1}>\cdots>x_1=1$, i.e., the positive entries of~$\pi^\uM$ appear in decreasing order, ending with~1.
\item In the decomposition $\pi^\uR=(\pi^\uR_1,\ldots,\pi^\uR_s)$ where $x_1<\pi^\uR_1<x_2<\pi^\uR_2<\cdots<x_{t-1}<\pi^\uR_{t-1}<x_t<\pi^\uR_t$, each of the subpermutations~$\pi^\uR_j$, $j\in[t]$, avoids 231.
\end{enumerate}
\end{lemma}

By Lemma~\ref{lem:struc-2b31p}, a $\tau$-avoiding signed permutation $\pi$ can be viewed as a 231-avoiding permutation~$\sigma$ (recall Lemma~\ref{lem:struc-231}), combined with the reversed and complemented permutation~$\sigma'$ such that the positive values to the left and including~1 of~$\sigma$ and their negated counterparts in~$\sigma'$ are interleaved arbitrarily (in a symmetric way).

\begin{proof}
Note that the value~1 is always in~$\pi^\uM$, because the positive values of~$\pi$ are all in either~$\pi^\uM$ or~$\pi^\uR$, and if the value~1 were in~$\pi^\uR$, then we would have $s>t$, a contradiction to the definition~\eqref{eq:tpi}.

If $s\geq p$, then the claims of the lemma are identical to Lemma~\ref{lem:struc-2b31}.
On the other hand, if $s<p$, then we apply Lemma~\ref{lem:struc-2b31} to obtain a decomposition $\pi=(\pi^{\uL^*},\pi^{\uM^*},\pi^{\uR^*})$ with $\pi^{\uR^*}=(\pi^{\uR^*}_0,\pi^{\uR^*}_1,\ldots,\pi^{\uR^*}_s)$, and we then apply Lemma~\ref{lem:struc-231} to further decompose~$\pi^{\uR^*}_0$, yielding a decomposition that has the desired properties.
\end{proof}

We note that $|\pi^\uM|=t_\pi$ and $|\pi^\uR|=|\pi^\uL|=n-t_\pi$.
Furthermore, as the value~1 belongs to~$\pi^\uM$ we have $\pi^\uM\neq \varepsilon$.

For $\pi\in B_n(\tau)$, we refer to $(\pi^\uL,\pi^\uM,\pi^\uR)$ and $\pi^\uR=(\pi^\uR_1,\ldots,\pi^\uR_t)$ as given by Lemma~\ref{lem:struc-2b31p} as \defn{staircase decomposition} of~$\pi$, and to the sequence~$x_t>\cdots>x_1$ as \defn{staircase values}.
Note that the length of~$\pi^\uM$ is~$2t$ and the length of~$\pi^\uR$ and~$\pi^\uL$ is~$n-t$.

From Lemma~\ref{lem:struc-2b31p}, we can directly deduce the following structure of minimal jumps inside a $\tau$-avoiding signed permutation; see Figure~\ref{fig:flips}.

\begin{lemma}
\label{lem:min-jumps}
Let $\pi=(a_{\ol{n}},\ldots,a_n)\in B_n(\tau)$, and consider its staircase decomposition $\pi=(\pi^\uL,\pi^\uM,\pi^\uR)$ and $\pi^\uR=(\pi^\uR_1,\ldots,\pi^\uR_t)$ with staircase values~$x_t>\cdots>x_1$, as given by Lemma~\ref{lem:struc-2b31p}.
A minimal jump of a value~$a_i>0$ in~$\pi$ is exactly one of the following:
\begin{enumerate}[label=(\roman*),leftmargin=8mm]
\item If $a_i$ is in~$\pi^\uM$ and its neighboring entry~$b$ in~$\pi^\uM$ is negative (either $b=a_{i-1}$ or $b=a_{i+1}$), it is an adjacent transposition $a_i\leftrightarrow b$ within~$\pi^\uM$.
\item If $a_i$ is in~$\pi^\uM$ and $a_{i+1}>0$, i.e., $a_i=x_j$ for some $j\in[t]$, then it is a right jump ending directly to the left of~$\pi^\uR_j$.
\item If $a_i$ is the leftmost entry of~$\pi^\uR_j$ for some $j\in[t]$, then it is a left jump ending directly to the left of~$x_j$.
\item If $a_i$ is in~$\pi^\uR_j$ for some $j\in[s]$, then it is a jump within the subpermutation~$\pi^R_j$ that is minimal w.r.t.\ avoidance of the pattern~231.
\end{enumerate}
\end{lemma}

Note that left and right jumps as in~(i) are inverse to each other, and the same is true for left and right jumps as in~(iv).
On the other hand, right jumps as in~(ii) are the inverse of the left jumps in~(iii).
We refer to a minimal jump as in~(i) or~(iv) as a \defn{close} minimal jump, and to a jump as in~(ii) or~(iii) as a \defn{far} minimal jump.
Note that close minimal jumps do not change the lengths of the three parts of the staircase decomposition, whereas a far jump either decreases or increases (in case~(ii) and~(iii), respectively) the length of the middle part~$\pi^\uM$ by~2.

For a signed permutation~$\pi\in B_{n-1}(\tau)$, it is easy to see (using Lemma~\ref{lem:struc-2b31p}) that $c_n(\pi)\in B_n(\tau)$, i.e., inserting $n$ at the rightmost position and $\ol{n}$ at the leftmost position yields again a permutation that avoids~$\tau$.
A similar statement is false for~$c_{\ol{n}}(\pi)$, which prevents~$B_n(\tau)$ from being a zigzag language.
For example, for $\pi=\b{2}\b{1}12\in B_2(\tau)$ we have $c_{\b{3}}(\pi)=3\b{2}\b{1}12\b{3}$, and the substring~$12\b{3}$ is an occurrence of~$\tau$.

The following lemma describes the leftmost possible position to insert the value~$n$ into a $\tau$-avoiding signed permutation~$\pi\in B_{n-1}(\tau)$, so that the resulting signed permutation is again $\tau$-avoiding.
Furthermore, importantly for Algorithm~J, minimal jumps are preserved under such insertions.

\begin{lemma}
\label{lem:min-jumps-preserve}
Let $\pi\in B_{n-1}(\tau)$, let $t:=t_\pi$ be as defined in~\eqref{eq:stpi}, and consider the staircase decomposition $\pi=(\pi^\uL,\pi^\uM,\pi^\uR)$ as given by Lemma~\ref{lem:struc-2b31p}.
Then the smallest integer~$i\in\pmn$ such that $c_i(\pi)$ avoids~$\tau$ equals $i=\ol{t+1}$, i.e., we have $\sigma:=c_{\ol{t+1}}(\pi)=(\pi^\uL,n,\pi^\uM,\ol{n},\pi^\uR)\in B_n(\tau)$.
Furthermore, the staircase decomposition of~$\sigma$ is $(\sigma^\uL,\sigma^\uM,\sigma^\uR)=(\pi^\uL,(n,\pi^\uM,\ol{n}),\pi^\uR)$, and any minimal jump in~$\pi$ is also a minimal jump in~$\sigma$.
\end{lemma}

\begin{proof}
Combine Lemmas~\ref{lem:struc-2b31p} and~\ref{lem:min-jumps}.
\end{proof}

We emphasize that a close minimal jump in~$\pi$ with $d$ steps is still a close minimal jump by $d$ steps in~$\sigma$, whereas a far minimal jump in~$\pi$ turns into a far minimal jump by $d+1$ steps in~$\sigma$, because the value~$\ol{n}$ is inserted at position~$t+1$.
For example, $\pi:=\b{3}\b{1}2\b{2}13\rightarrow \b{1}32\b{2}\b{3}1$ is a far minimal left jump of~3 by 4~steps in~$\pi$, and we have $t_\pi=2$, so $\sigma:=c_{\ol{t+1}}(\pi)=c_{\b{3}}(\pi)=\b{3}4\b{1}2\b{2}1\b{4}3$, and $\sigma=\b{3}4\b{1}2\b{2}1\b{4}3\rightarrow 4\b{1}32\b{2}\b{3}1\b{4}$ is a far minimal left jump of~3 by 5 steps in~$\sigma$.

\subsubsection{A bijection between $\Delta_n^*$ and $B_n(\tau)$}
\label{sec:bij}

The following definitions are illustrated in Figure~\ref{fig:bij}.

\begin{figure}[t!]
\makebox[0cm]{ % artificial box to center the picture
\includegraphics[page=2]{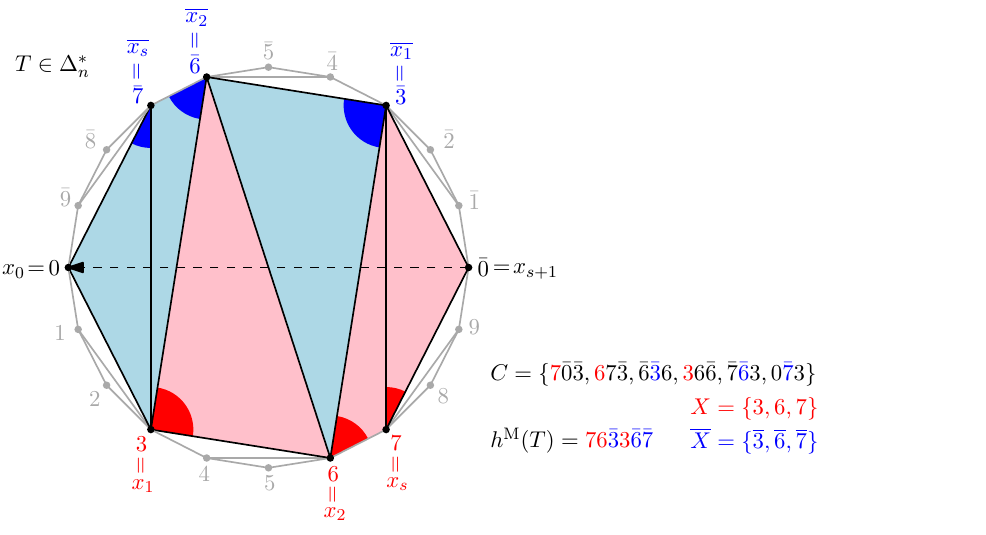}} \\
\makebox[0cm]{ % artificial box to center the picture
\includegraphics[page=3]{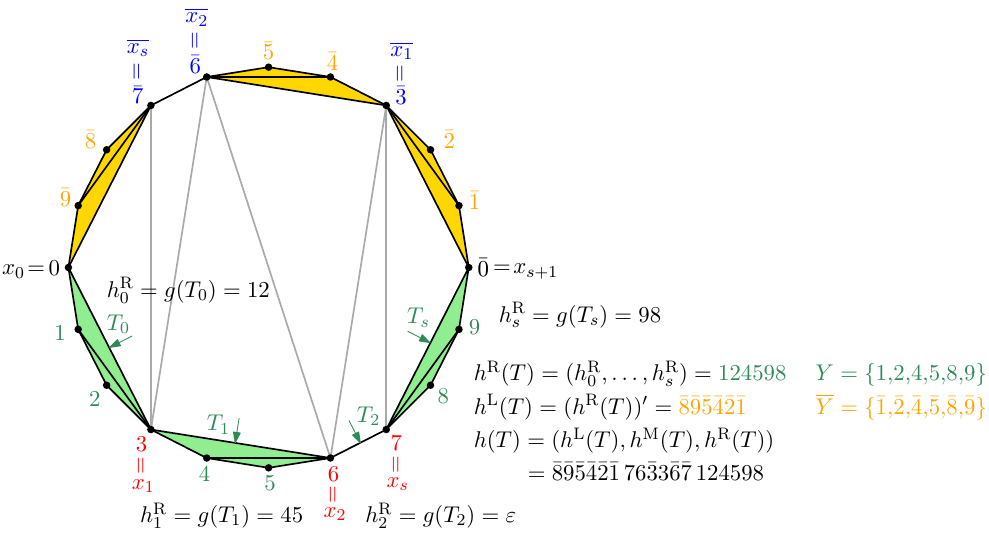}}
\caption{Definition of the bijection~$h$ between symmetric triangulations and $\tau$-avoiding signed permutations.
The triangulation is the same as in Figure~\ref{fig:symm-triang}.}
\label{fig:bij}
\end{figure}

\begin{figure}[t!]
\makebox[0cm]{ % artificial box to center the picture
\includegraphics[page=5,scale=0.8]{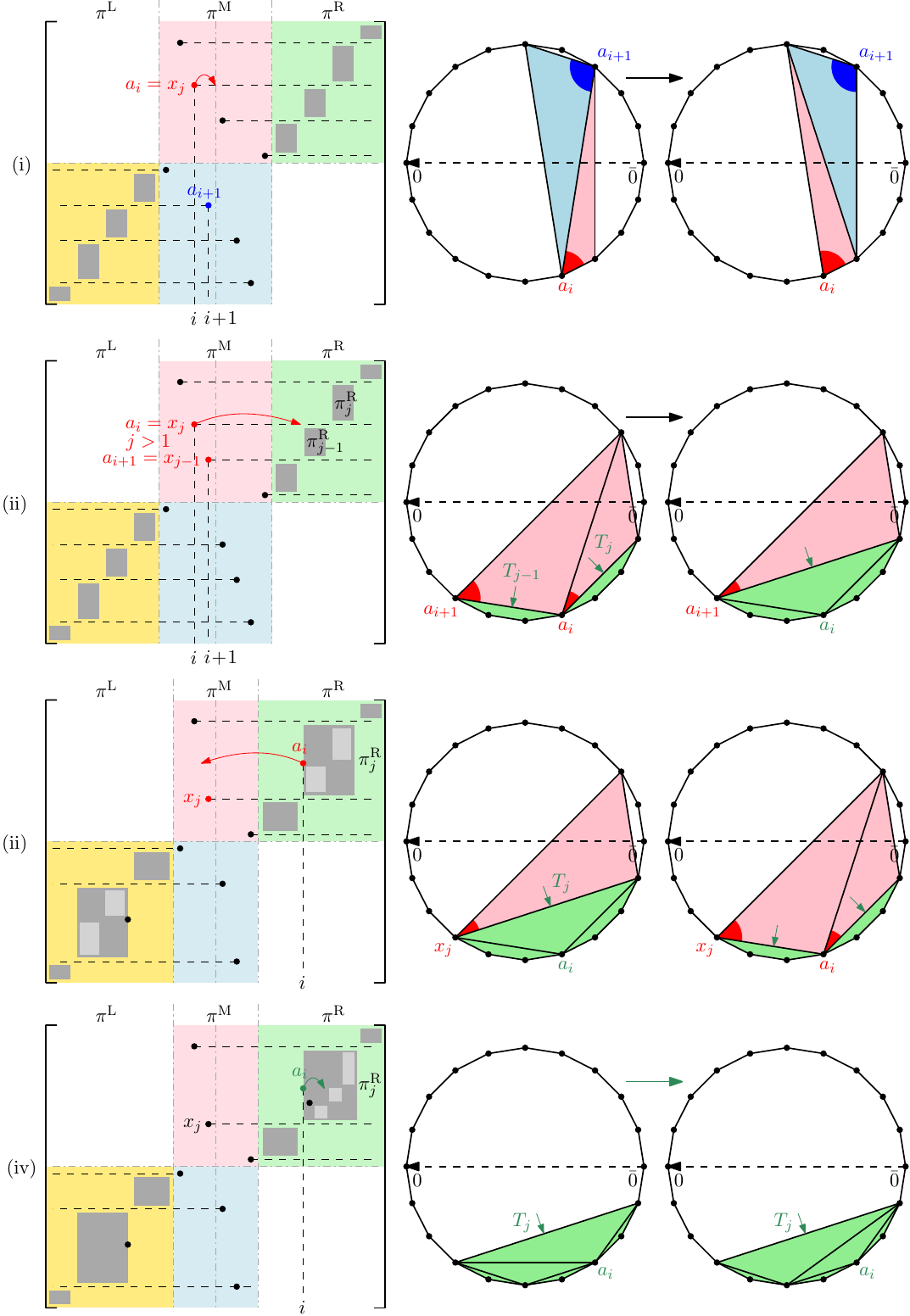}
}
\caption{Correspondence between minimal jumps in $\tau$-avoiding signed permutations as described by Lemma~\ref{lem:min-jumps} (left) and flips in symmetric triangulations (right) under the bijection~$h$.}
\label{fig:flips}
\end{figure}

Consider a symmetric triangulation~$T\in\Delta_n^*$.
Let~$C$ be the set of triangles in~$T$ that have points on both sides of the line~$\b{0}0$.
We refer to them as \defn{crossing} triangles of~$T$.
Let $X=\{x_1<x_2<\cdots<x_s\}$ be the set of positive vertices of triangles in~$C$, excluding the vertices~$x_0:=0$ and~$x_{s+1}:=\b{0}$.
Note that we have $|C|=2s$.

We first define a permutation~$h^\uM(T)$ of the integers in~$X\cup\ol{X}$ as follows.
We sweep through the triangles in~$C$ in the order in which they intersect the ray from~$\b{0}$ to~$0$, starting with the triangle~$x_s\ol{x_1}\b{0}$ and ending with the triangle~$0x_1\ol{x_s}$.
For every triangle in~$C$, we add one of its vertices to the string~$h^\uM(T)$, according to the following rule:
If the triangle is~$abc$ with $a,b\in[n]\cup\{\b{0}\}$, $b=\b{0}$ or $a<b$, and $c\in[\ol{n}]$, we add the vertex~$a$.
If the triangle is~$abc$ with $a,b\in[\ol{n}]\cup\{0\}$, $a=0$ or $a<b$, and $c\in[n]$, we add the vertex~$b$.
One can check that~$(h^\uM(T))'=h^\uM(T)$.
Furthermore, if central edge is~$u\ol{u}$, then $u$ is the smallest/rightmost positive value in the left half of~$h^\uM(T)$, and $\ol{u}$ is the largest/leftmost negative value in the right half of~$h^\uM(T)$.
In the special case were the central edge is~$0\b{0}$, we have $h^\uM(T)=\varepsilon$.

We now define a permutation~$h^\uR(T)$ of the integers in~$Y:=[n]\setminus X$ as follows.
For $j=0,\ldots,s$, we let $T_j$ denote the subtriangulation of~$T$ of non-crossing triangles seen through the edge~$x_jx_{j+1}$.
We define $h^\uR(T):=(h^\uR_0,\cdots,h^\uR_s)$ with $h^\uR_j:=g(T_j)$, where $g$ is the mapping defined in Section~\ref{sec:triang}, and in evaluating~$g(T_j)$ we consider the (sub)triangulation~$T_j$ through the edge~$x_jx_{j+1}$.
Furthermore, we define~$h^\uL(T):=(h^\uR(T))'$, which is a permutation of the integers in~$\ol{Y}=[\ol{n}]\setminus \ol{X}$.

Combining the results from the previous steps, we define $h(T):=(h^\uL(T),h^\uM(T),h^\uR(T))$.

\begin{theorem}
\label{thm:h-bij}
The mapping~$h:\Delta_n^*\rightarrow B_n(\ub{2}31)$ is a bijection that induces a one-to-one correspondence between flips in symmetric triangulations and minimal jumps in $\ub{2}31$-avoiding signed permutations.
\end{theorem}

The theorem shows that the two flip graphs on symmetric triangulations and $\tau$-avoiding signed permutations are isomorphic.
Figure~\ref{fig:triang} shows the $B$-associahedron for $n=3$, labeled with the symmetric triangulations and $\tau$-avoiding signed permutations corresponding to each other under the bijection~$h$.

\begin{proof}
The fact that~$h$ is a bijection can be verified directly, using Lemmas~\ref{lem:g-bij} and~\ref{lem:struc-2b31}.
The correspondence between flips and minimal jumps follows from Lemma~\ref{lem:min-jumps}.
Specifically, the flips corresponding to the four different types of jumps~(i)---(iv) described in the lemma are illustrated in Figure~\ref{fig:flips}~(i)---(iv), respectively.
\end{proof}

\begin{theorem}
\label{thm:2b31}
Given the set $L_n=B_n(\ub{2}31)$ and initial permutation~$\pi_0=\id_n$, Algorithm~J visits every signed permutation from~$L_n$ exactly once.
Furthermore, the last permutation differs from the first one in a minimal jump.
\end{theorem}

The resulting orderings of $\tau$-avoiding signed permutations and symmetric triangulations are shown in Figures~\ref{fig:triang} and~\ref{fig:gc-triang} for $n=1,2,3,4$.

\begin{proof}
The proof is a straightforward adaption of the proof of Theorem~\ref{thm:zigzag}.
We use again the property that if $\pi\in B_{n-1}(\tau)$, then $c_n(\pi)\in B_n(\tau)$ and any minimal jump in~$\pi$ is also a minimal jump in~$c_n(\tau)$.
As a similar statement does not hold for $c_{\ol{n}}(\pi)$, we apply Lemma~\ref{lem:min-jumps-preserve} instead.
Using that $|B_n(\tau)|=\binom{2n}{n}$ is even, an easy induction shows that the last permutation in $J(B_n(\tau))$ is $\ol{n}\cdots\b{3}\b{2}1\ol{1}23\cdots n$ (recall Remark~\ref{rem:cyclic}), i.e., it differs from the first permutation~$\id_n$ in a minimal jump of~1.
\end{proof}

\section*{Acknowledgements}

We thank the reviewers of the extended abstract of this paper for numerous helpful comments and suggestions.

\bibliographystyle{alpha}
\bibliography{refs}

\end{document}

%% file: graphics/permAcycle.tex
\begin{tikzpicture}%
[x={(0.105480cm, -0.368869cm)},
y={(0.994421cm, 0.039067cm)},
z={(0.000064cm, 0.928660cm)},
scale=1.000000,
backr/.style={color=red!30!white, thick},
backb/.style={color=blue!30!white, thick},
edger/.style={color=red!95!black, thick},
edgeb/.style={color=blue!95!black, thick},
facet/.style={fill=red!95!black,fill opacity=0.1},
vertex/.style={inner sep=1pt,rectangle,fill=white,thick,scale=.7}]
%
%
%% This TikZ-picture was produced with Sagemath version 10.3
%% with the command: ._tikz_3d_in_3d and parameters:
%% view = [-0.529100000000000, -0.475900000000000, -0.702600000000000]
%% angle = 104.020000000000
%%
%% Coordinate of the vertices:
%%
\coordinate (4132) at (-0.70711, -1.22474, -1.73205);
\coordinate (1432) at (-0.70711, -2.04124, -0.57735);
\coordinate (4312) at (-1.41421, 0.00000, -1.73205);
\coordinate (1342) at (-1.41421, -1.63299, 0.57735);
\coordinate (3412) at (-2.12132, 0.40825, -0.57735);
\coordinate (3142) at (-2.12132, -0.40825, 0.57735);
\coordinate (4123) at (0.70711, -1.22474, -1.73205);
\coordinate (1423) at (0.70711, -2.04124, -0.57735);
\coordinate (4321) at (-0.70711, 1.22474, -1.73205);
\coordinate (1324) at (-0.70711, -1.22474, 1.73205);
\coordinate (3421) at (-1.41421, 1.63299, -0.57735);
\coordinate (3124) at (-1.41421, 0.00000, 1.73205);
\coordinate (4213) at (1.41421, 0.00000, -1.73205);
\coordinate (1243) at (1.41421, -1.63299, 0.57735);
\coordinate (4231) at (0.70711, 1.22474, -1.73205);
\coordinate (1234) at (0.70711, -1.22474, 1.73205);
\coordinate (3241) at (-0.70711, 2.04124, 0.57735);
\coordinate (3214) at (-0.70711, 1.22474, 1.73205);
\coordinate (2413) at (2.12132, 0.40825, -0.57735);
\coordinate (2143) at (2.12132, -0.40825, 0.57735);
\coordinate (2431) at (1.41421, 1.63299, -0.57735);
\coordinate (2134) at (1.41421, 0.00000, 1.73205);
\coordinate (2341) at (0.70711, 2.04124, 0.57735);
\coordinate (2314) at (0.70711, 1.22474, 1.73205);
%%
%%
%% Drawing edges in the back
%%
\draw[backr] (4132) -- (1432);
\draw[backb] (4132) -- (4123);
\draw[backr] (4312) -- (3412);
\draw[backb] (4312) -- (4321);
\draw[backr] (3412) -- (3142);
\draw[backr] (3142) -- (3124);
\draw[backr] (4321) -- (3421);
\draw[backr] (3421) -- (3241);
%%
%%
%% Drawing vertices in the back
%%
\node[vertex] at ([xshift=5pt,yshift=5pt]3142)     {3142};
\node[vertex] at ([xshift=3pt]4132)     {4132};
\node[vertex] at (4312)     {4312};
\node[vertex] at ([xshift=3pt]3412)     {3412};
\node[vertex] at ([xshift=3pt]4321)     {4321};
\node[vertex] at (3421)     {3421};
%%
%%
%% Drawing the facets
%%
\fill[facet] (2134) -- (1234) -- (1243) -- (2143) -- cycle {};
\fill[facet, fill=blue] (2314) -- (3214) -- (3124) -- (1324) -- (1234) -- (2134) -- cycle {};
\fill[facet] (1234) -- (1324) -- (1342) -- (1432) -- (1423) -- (1243) -- cycle {};
\fill[facet] (2314) -- (3214) -- (3241) -- (2341) -- cycle {};
\fill[facet] (2431) -- (4231) -- (4213) -- (2413) -- cycle {};
\fill[facet] (2143) -- (1243) -- (1423) -- (4123) -- (4213) -- (2413) -- cycle {};
\fill[facet] (2314) -- (2134) -- (2143) -- (2413) -- (2431) -- (2341) -- cycle {};
\fill[facet, fill=blue] (4123) -- (4213) -- (4231) -- (4321) -- (4312) -- (4132) -- cycle {};
%%
%%
%% Drawing edges in the front
%%
\draw[edger] (1432) -- (1342);
\draw[edger] (1342) -- (1324);
\draw[edger] (4123) -- (1423);
\draw[edger] (1423) -- (1243);
\draw[edgeb] (1324) -- (3124);
\draw[edgeb] (4213) -- (4231);
\draw[edger] (4213) -- (2413);
\draw[edger] (1243) -- (1234);
\draw[edger] (4231) -- (2431);
\draw[edgeb] (1234) -- (2134);
\draw[edger] (3241) -- (3214);
\draw[edgeb] (3214) -- (2314);
\draw[edger] (2413) -- (2143);
\draw[edger] (2143) -- (2134);
\draw[edger] (2431) -- (2341);
\draw[edger] (2341) -- (2314);
%%
%%
%% Drawing the vertices in the front
%%
\node[vertex] at (1432)     {1432}; 
\node[vertex] at (1342)     {1342}; 
\node[vertex] at (4123)     {4123}; 
\node[vertex] at (1423)     {1423}; 
\node[vertex] at (1324)     {1324}; 
\node[vertex] at (3124)     {3124}; 
\node[vertex] at (4213)     {4213}; 
\node[vertex] at (1243)     {1243}; 
\node[vertex] at (4231)     {4231}; 
\node[vertex] at (1234)     {1234}; 
\node[vertex] at (3241)     {3241}; 
\node[vertex] at (3214)     {3214}; 
\node[vertex] at (2413)     {2413}; 
\node[vertex] at (2143)     {2143}; 
\node[vertex] at (2431)     {2431}; 
\node[vertex] at (2134)     {2134}; 
\node[vertex] at (2341)     {2341}; 
\node[vertex] at (2314)     {2314}; 
\end{tikzpicture}

%% file: graphics/cube.tex
\begin{tikzpicture}%
	[x={(-0.442768cm, -0.371341cm)},
	y={(0.896636cm, -0.183347cm)},
	z={(-0.000025cm, 0.910214cm)},
	scale=2.500000,
	back/.style={dotted, thin},
	edge/.style={color=blue!95!black, thick},
	facet/.style={fill=red!95!black,fill opacity=0.1},
	vertex/.style={inner sep=2pt,rectangle,fill=white,thick,scale=.7}]
%
%
%% This TikZ-picture was produce with Sagemath version 9.5
%% with the command: ._tikz_3d_in_3d and parameters:
%% view = [-0.318800000000000, -0.513000000000000, -0.797000000000000]
%% angle = 127.300000000000
%% scale = 1
%% edge_color = blue!95!black
%% facet_color = blue!95!black
%% opacity = 0.8
%% vertex_color = green
%% axis = False

%% Coordinate of the vertices:
%%
\coordinate (000) at (1.00000, 0.00000, 1.00000);
\coordinate (001) at (1.00000, 0.00000, 0.00000);
\coordinate (010) at (0.00000, 0.00000, 1.00000);
\coordinate (011) at (0.00000, 0.00000, 0.00000);
\coordinate (100) at (1.00000, 1.00000, 1.00000);
\coordinate (101) at (1.00000, 1.00000, 0.00000);
\coordinate (110) at (0.00000, 1.00000, 1.00000);
\coordinate (111) at (0.00000, 1.00000, 0.00000);
%%
%%
%% Drawing edges in the back
%%
\draw[edge,back] (011) -- (111);
\draw[edge,back] (011) -- (001);
\draw[edge,back] (011) -- (010);
%%
%%
%% Drawing vertices in the back
%%
\node[vertex] at (011)     {011};
%%
%%
%% Drawing the facets
%%
\fill[facet] (000) -- (010) -- (110) -- (100) -- cycle {};
\fill[facet] (101) -- (001) -- (000) -- (100) -- cycle {};
\fill[facet] (111) -- (101) -- (100) -- (110) -- cycle {};
%%
%%
%% Drawing edges in the front
%%
\draw[edge] (001) -- (000);
\draw[edge] (001) -- (101);
\draw[edge] (010) -- (000);
\draw[edge] (010) -- (110);
\draw[edge] (000) -- (100);
\draw[edge] (100) -- (101);
\draw[edge] (100) -- (110);
\draw[edge] (101) -- (111);
\draw[edge] (110) -- (111);
%%
%%
%% Drawing the vertices in the front
%%
\node[vertex] at (001)     {001};
\node[vertex] at (010)     {010};
\node[vertex] at (000)     {000};
\node[vertex] at (100)     {100};
\node[vertex] at (101)     {101};
\node[vertex] at (110)     {110};
\node[vertex] at (111)     {111};

\node[vertex] at (0,0,1.3) {\Large $G(\mathcal{H})$};
\end{tikzpicture}

%% file: graphics/permA.tex
\begin{tikzpicture}%
[x={(0.105480cm, -0.368869cm)},
y={(0.994421cm, 0.039067cm)},
z={(0.000064cm, 0.928660cm)},
scale=1.000000,
back/.style={dotted, thin},
edge/.style={color=blue!95!black, thick},
facet/.style={fill=red!95!black,fill opacity=0.1},
vertex/.style={inner sep=2pt,rectangle,fill=white,thick,scale=.7}]
%
%
%% This TikZ-picture was produced with Sagemath version 10.3
%% with the command: ._tikz_3d_in_3d and parameters:
%% view = [-0.529100000000000, -0.475900000000000, -0.702600000000000]
%% angle = 104.020000000000
%%
%% Coordinate of the vertices:
%%
\coordinate (4132) at (-0.70711, -1.22474, -1.73205);
\coordinate (1432) at (-0.70711, -2.04124, -0.57735);
\coordinate (4312) at (-1.41421, 0.00000, -1.73205);
\coordinate (1342) at (-1.41421, -1.63299, 0.57735);
\coordinate (3412) at (-2.12132, 0.40825, -0.57735);
\coordinate (3142) at (-2.12132, -0.40825, 0.57735);
\coordinate (4123) at (0.70711, -1.22474, -1.73205);
\coordinate (1423) at (0.70711, -2.04124, -0.57735);
\coordinate (4321) at (-0.70711, 1.22474, -1.73205);
\coordinate (1324) at (-0.70711, -1.22474, 1.73205);
\coordinate (3421) at (-1.41421, 1.63299, -0.57735);
\coordinate (3124) at (-1.41421, 0.00000, 1.73205);
\coordinate (4213) at (1.41421, 0.00000, -1.73205);
\coordinate (1243) at (1.41421, -1.63299, 0.57735);
\coordinate (4231) at (0.70711, 1.22474, -1.73205);
\coordinate (1234) at (0.70711, -1.22474, 1.73205);
\coordinate (3241) at (-0.70711, 2.04124, 0.57735);
\coordinate (3214) at (-0.70711, 1.22474, 1.73205);
\coordinate (2413) at (2.12132, 0.40825, -0.57735);
\coordinate (2143) at (2.12132, -0.40825, 0.57735);
\coordinate (2431) at (1.41421, 1.63299, -0.57735);
\coordinate (2134) at (1.41421, 0.00000, 1.73205);
\coordinate (2341) at (0.70711, 2.04124, 0.57735);
\coordinate (2314) at (0.70711, 1.22474, 1.73205);
%%
%%
%% Drawing edges in the back
%%
\draw[edge,back] (4132) -- (1432);
\draw[edge,back] (4132) -- (4312);
\draw[edge,back] (4132) -- (4123);
\draw[edge,back] (4312) -- (3412);
\draw[edge,back] (4312) -- (4321);
\draw[edge,back] (1342) -- (3142);
\draw[edge,back] (3412) -- (3142);
\draw[edge,back] (3412) -- (3421);
\draw[edge,back] (3142) -- (3124);
\draw[edge,back] (4321) -- (3421);
\draw[edge,back] (4321) -- (4231);
\draw[edge,back] (3421) -- (3241);
%%
%%
%% Drawing vertices in the back
%%
\node[vertex] at ([xshift=5pt,yshift=5pt]3142)     {3142};
\node[vertex] at ([xshift=3pt]4132)     {4132};
\node[vertex] at (4312)     {4312};
\node[vertex] at ([xshift=3pt]3412)     {3412};
\node[vertex] at ([xshift=3pt]4321)     {4321};
\node[vertex] at (3421)     {3421};
%%
%%
%% Drawing the facets
%%
\fill[facet] (2134) -- (1234) -- (1243) -- (2143) -- cycle {};
\fill[facet] (2314) -- (3214) -- (3124) -- (1324) -- (1234) -- (2134) -- cycle {};
\fill[facet] (1234) -- (1324) -- (1342) -- (1432) -- (1423) -- (1243) -- cycle {};
\fill[facet] (2314) -- (3214) -- (3241) -- (2341) -- cycle {};
\fill[facet] (2431) -- (4231) -- (4213) -- (2413) -- cycle {};
\fill[facet] (2143) -- (1243) -- (1423) -- (4123) -- (4213) -- (2413) -- cycle {};
\fill[facet] (2314) -- (2134) -- (2143) -- (2413) -- (2431) -- (2341) -- cycle {};
%%
%%
%% Drawing edges in the front
%%
\draw[edge] (1432) -- (1342);
\draw[edge] (1432) -- (1423);
\draw[edge] (1342) -- (1324);
\draw[edge] (4123) -- (1423);
\draw[edge] (4123) -- (4213);
\draw[edge] (1423) -- (1243);
\draw[edge] (1324) -- (3124);
\draw[edge] (1324) -- (1234);
\draw[edge] (3124) -- (3214);
\draw[edge] (4213) -- (4231);
\draw[edge] (4213) -- (2413);
\draw[edge] (1243) -- (1234);
\draw[edge] (1243) -- (2143);
\draw[edge] (4231) -- (2431);
\draw[edge] (1234) -- (2134);
\draw[edge] (3241) -- (3214);
\draw[edge] (3241) -- (2341);
\draw[edge] (3214) -- (2314);
\draw[edge] (2413) -- (2143);
\draw[edge] (2413) -- (2431);
\draw[edge] (2143) -- (2134);
\draw[edge] (2431) -- (2341);
\draw[edge] (2134) -- (2314);
\draw[edge] (2341) -- (2314);
%%
%%
%% Drawing the vertices in the front
%%
\node[vertex] at (1432)     {1432}; 
\node[vertex] at (1342)     {1342}; 
\node[vertex] at (4123)     {4123}; 
\node[vertex] at (1423)     {1423}; 
\node[vertex] at (1324)     {1324}; 
\node[vertex] at (3124)     {3124}; 
\node[vertex] at (4213)     {4213}; 
\node[vertex] at (1243)     {1243}; 
\node[vertex] at (4231)     {4231}; 
\node[vertex] at (1234)     {1234}; 
\node[vertex] at (3241)     {3241}; 
\node[vertex] at (3214)     {3214}; 
\node[vertex] at (2413)     {2413}; 
\node[vertex] at (2143)     {2143}; 
\node[vertex] at (2431)     {2431}; 
\node[vertex] at (2134)     {2134}; 
\node[vertex] at (2341)     {2341}; 
\node[vertex] at (2314)     {2314}; 
\end{tikzpicture}

%% file: graphics/permB.tex
\begin{tikzpicture}%
[x={(-0.801330cm, 0.247442cm)},
y={(-0.598223cm, -0.331531cm)},
z={(-0.000051cm, 0.910417cm)},
scale=.5,
back/.style={dotted, thin},
edge/.style={color=blue!95!black, thick},
facet/.style={fill=red!95!black,fill opacity=0.1},
vertex/.style={inner sep=1pt,rectangle,fill=white,thick,scale=.7}]
%
%
%% This TikZ-picture was produced with Sagemath version 10.3
%% with the command: ._tikz_3d_in_3d and parameters:
%% view = [0.177000000000000, -0.532900000000000, -0.827500000000000]
%% angle = 210.730000000000
%%
%% Coordinate of the vertices:
%%
\coordinate (213) at (2.41421, 1.00000, 3.82843);
\coordinate (123) at (1.00000, 2.41421, 3.82843);
\coordinate (132) at (1.00000, 3.82843, 2.41421);
\coordinate (21b3) at (2.41421, 1.00000, -3.82843);
\coordinate (12b3) at (1.00000, 2.41421, -3.82843);
\coordinate (1b32) at (1.00000, 3.82843, -2.41421);
\coordinate (312) at (2.41421, 3.82843, 1.00000);
\coordinate (b312) at (2.41421, 3.82843, -1.00000);
\coordinate (1b23) at (1.00000, -2.41421, 3.82843);
\coordinate (13b2) at (1.00000, -3.82843, 2.41421);
\coordinate (b213) at (2.41421, -1.00000, 3.82843);
\coordinate (1b2b3) at (1.00000, -2.41421, -3.82843);
\coordinate (1b3b2) at (1.00000, -3.82843, -2.41421);
\coordinate (b21b3) at (2.41421, -1.00000, -3.82843);
\coordinate (231) at (3.82843, 1.00000, 2.41421);
\coordinate (2b31) at (3.82843, 1.00000, -2.41421);
\coordinate (321) at (3.82843, 2.41421, 1.00000);
\coordinate (b321) at (3.82843, 2.41421, -1.00000);
\coordinate (b231) at (3.82843, -1.00000, 2.41421);
\coordinate (b2b31) at (3.82843, -1.00000, -2.41421);
\coordinate (31b2) at (2.41421, -3.82843, 1.00000);
\coordinate (b31b2) at (2.41421, -3.82843, -1.00000);
\coordinate (3b21) at (3.82843, -2.41421, 1.00000);
\coordinate (b3b21) at (3.82843, -2.41421, -1.00000);
\coordinate (2b13) at (-2.41421, 1.00000, 3.82843);
\coordinate (b123) at (-1.00000, 2.41421, 3.82843);
\coordinate (b132) at (-1.00000, 3.82843, 2.41421);
\coordinate (2b1b3) at (-2.41421, 1.00000, -3.82843);
\coordinate (b12b3) at (-1.00000, 2.41421, -3.82843);
\coordinate (b1b32) at (-1.00000, 3.82843, -2.41421);
\coordinate (3b12) at (-2.41421, 3.82843, 1.00000);
\coordinate (b3b12) at (-2.41421, 3.82843, -1.00000);
\coordinate (b1b23) at (-1.00000, -2.41421, 3.82843);
\coordinate (b13b2) at (-1.00000, -3.82843, 2.41421);
\coordinate (b2b13) at (-2.41421, -1.00000, 3.82843);
\coordinate (b1b2b3) at (-1.00000, -2.41421, -3.82843);
\coordinate (b1b3b2) at (-1.00000, -3.82843, -2.41421);
\coordinate (b2b1b3) at (-2.41421, -1.00000, -3.82843);
\coordinate (23b1) at (-3.82843, 1.00000, 2.41421);
\coordinate (2b3b1) at (-3.82843, 1.00000, -2.41421);
\coordinate (32b1) at (-3.82843, 2.41421, 1.00000);
\coordinate (b32b1) at (-3.82843, 2.41421, -1.00000);
\coordinate (b23b1) at (-3.82843, -1.00000, 2.41421);
\coordinate (b2b3b1) at (-3.82843, -1.00000, -2.41421);
\coordinate (3b1b2) at (-2.41421, -3.82843, 1.00000);
\coordinate (b3b1b2) at (-2.41421, -3.82843, -1.00000);
\coordinate (3b2b1) at (-3.82843, -2.41421, 1.00000);
\coordinate (b3b2b1) at (-3.82843, -2.41421, -1.00000);
%%
%%
%% Drawing edges in the back
%%
\draw[edge,back] (21b3) -- (12b3);
\draw[edge,back] (21b3) -- (b21b3);
\draw[edge,back] (21b3) -- (2b31);
\draw[edge,back] (1b23) -- (13b2);
\draw[edge,back] (13b2) -- (31b2);
\draw[edge,back] (13b2) -- (b13b2);
\draw[edge,back] (b213) -- (b231);
\draw[edge,back] (1b2b3) -- (1b3b2);
\draw[edge,back] (1b2b3) -- (b21b3);
\draw[edge,back] (1b2b3) -- (b1b2b3);
\draw[edge,back] (1b3b2) -- (b31b2);
\draw[edge,back] (1b3b2) -- (b1b3b2);
\draw[edge,back] (b21b3) -- (b2b31);
\draw[edge,back] (231) -- (b231);
\draw[edge,back] (2b31) -- (b321);
\draw[edge,back] (2b31) -- (b2b31);
\draw[edge,back] (b231) -- (3b21);
\draw[edge,back] (b2b31) -- (b3b21);
\draw[edge,back] (31b2) -- (b31b2);
\draw[edge,back] (31b2) -- (3b21);
\draw[edge,back] (b31b2) -- (b3b21);
\draw[edge,back] (3b21) -- (b3b21);
\draw[edge,back] (b1b2b3) -- (b1b3b2);
\draw[edge,back] (b1b2b3) -- (b2b1b3);
\draw[edge,back] (b1b3b2) -- (b3b1b2);
\draw[edge,back] (3b1b2) -- (b3b1b2);
\draw[edge,back] (b3b1b2) -- (b3b2b1);
%%
%%
%% Drawing vertices in the back
%%
\node[vertex] at (13b2)     {$13\bar{2}$};
\node[vertex] at (1b2b3)     {$1\bar{2}\bar{3}$};
\node[vertex] at ([xshift=-5pt,yshift=12pt]1b3b2)     {$1\bar{3}\bar{2}$};
\node[vertex] at ([xshift=10pt,yshift=10pt]b1b2b3)     {$\bar{1}\bar{2}\bar{3}$};
\node[vertex] at (b1b3b2)     {$\bar{1}\bar{3}\bar{2}$};
\node[vertex] at ([xshift=10pt,yshift=4pt]21b3)     {$21\bar{3}$};
\node[vertex] at (b21b3)     {$\bar{2}1\bar{3}$};
\node[vertex] at (31b2)     {$31\bar{2}$};
\node[vertex] at (b31b2)     {$\bar{3}1\bar{2}$};
\node[vertex] at (b3b1b2)     {$\bar{3}\bar{1}\bar{2}$};
\node[vertex] at ([xshift=10pt]b231)     {$\bar{2}31$};
\node[vertex] at ([xshift=10pt]2b31)     {$2\bar{3}1$};
\node[vertex] at (b2b31)     {$\bar{2}\bar{3}1$};
\node[vertex] at ([yshift=-9pt]3b21)     {$3\bar{2}1$};
\node[vertex] at ([yshift=10pt]b3b21)     {$\bar{3}\bar{2}1$};
%%
%%
%% Drawing the facets
%%
\fill[facet] (b132) -- (132) -- (123) -- (b123) -- cycle {};
\fill[facet] (b1b32) -- (1b32) -- (12b3) -- (b12b3) -- cycle {};
\fill[facet] (b2b13) -- (2b13) -- (b123) -- (123) -- (213) -- (b213) -- (1b23) -- (b1b23) -- cycle {};
\fill[facet] (b3b12) -- (b1b32) -- (1b32) -- (b312) -- (312) -- (132) -- (b132) -- (3b12) -- cycle {};
\fill[facet] (b23b1) -- (b2b13) -- (2b13) -- (23b1) -- cycle {};
\fill[facet] (b2b3b1) -- (b2b1b3) -- (2b1b3) -- (2b3b1) -- cycle {};
\fill[facet] (321) -- (312) -- (132) -- (123) -- (213) -- (231) -- cycle {};
\fill[facet] (b321) -- (b312) -- (312) -- (321) -- cycle {};
\fill[facet] (32b1) -- (3b12) -- (b132) -- (b123) -- (2b13) -- (23b1) -- cycle {};
\fill[facet] (b32b1) -- (b3b12) -- (b1b32) -- (b12b3) -- (2b1b3) -- (2b3b1) -- cycle {};
\fill[facet] (b32b1) -- (b3b12) -- (3b12) -- (32b1) -- cycle {};
\fill[facet] (3b2b1) -- (b23b1) -- (b2b13) -- (b1b23) -- (b13b2) -- (3b1b2) -- cycle {};
\fill[facet] (b3b2b1) -- (b2b3b1) -- (2b3b1) -- (b32b1) -- (32b1) -- (23b1) -- (b23b1) -- (3b2b1) -- cycle {};
%%
%%
%% Drawing edges in the front
%%
\draw[edge] (213) -- (123);
\draw[edge] (213) -- (b213);
\draw[edge] (213) -- (231);
\draw[edge] (123) -- (132);
\draw[edge] (123) -- (b123);
\draw[edge] (132) -- (312);
\draw[edge] (132) -- (b132);
\draw[edge] (12b3) -- (1b32);
\draw[edge] (12b3) -- (b12b3);
\draw[edge] (1b32) -- (b312);
\draw[edge] (1b32) -- (b1b32);
\draw[edge] (312) -- (b312);
\draw[edge] (312) -- (321);
\draw[edge] (b312) -- (b321);
\draw[edge] (1b23) -- (b213);
\draw[edge] (1b23) -- (b1b23);
\draw[edge] (231) -- (321);
\draw[edge] (321) -- (b321);
\draw[edge] (2b13) -- (b123);
\draw[edge] (2b13) -- (b2b13);
\draw[edge] (2b13) -- (23b1);
\draw[edge] (b123) -- (b132);
\draw[edge] (b132) -- (3b12);
\draw[edge] (2b1b3) -- (b12b3);
\draw[edge] (2b1b3) -- (b2b1b3);
\draw[edge] (2b1b3) -- (2b3b1);
\draw[edge] (b12b3) -- (b1b32);
\draw[edge] (b1b32) -- (b3b12);
\draw[edge] (3b12) -- (b3b12);
\draw[edge] (3b12) -- (32b1);
\draw[edge] (b3b12) -- (b32b1);
\draw[edge] (b1b23) -- (b13b2);
\draw[edge] (b1b23) -- (b2b13);
\draw[edge] (b13b2) -- (3b1b2);
\draw[edge] (b2b13) -- (b23b1);
\draw[edge] (b2b1b3) -- (b2b3b1);
\draw[edge] (23b1) -- (32b1);
\draw[edge] (23b1) -- (b23b1);
\draw[edge] (2b3b1) -- (b32b1);
\draw[edge] (2b3b1) -- (b2b3b1);
\draw[edge] (32b1) -- (b32b1);
\draw[edge] (b23b1) -- (3b2b1);
\draw[edge] (b2b3b1) -- (b3b2b1);
\draw[edge] (3b1b2) -- (3b2b1);
\draw[edge] (3b2b1) -- (b3b2b1);
%%
%%
%% Drawing the vertices in the front
%%
\node[vertex] at ([xshift=-8pt,yshift=7pt]213)     {$213$};
\node[vertex] at (123)     {$123$};
\node[vertex] at (132)     {$132$}; 
\node[vertex] at ([xshift=-3pt,yshift=-3pt]12b3)     {$12\bar{3}$};
\node[vertex] at (1b32)     {$1\bar{3}2$};
\node[vertex] at (312)     {$312$};
\node[vertex] at (b312)     {$\bar{3}12$};
\node[vertex] at (1b23)     {$1\bar{2}3$};
\node[vertex] at (b213)     {$\bar{2}13$};
\node[vertex] at (231)     {$231$};
\node[vertex] at (321)     {$321$};
\node[vertex] at (b321)     {$\bar{3}21$};
\node[vertex] at (2b13)     {$2\bar{1}3$};
\node[vertex] at (b123)     {$\bar{1}23$}; 
\node[vertex] at (b132)     {$\bar{1}32$};
\node[vertex] at (2b1b3)     {$2\bar{1}\bar{3}$};
\node[vertex] at (b12b3)     {$\bar{1}2\bar{3}$};
\node[vertex] at (b1b32)     {$\bar{1}\bar{3}2$};
\node[vertex] at (3b12)     {$3\bar{1}2$};
\node[vertex] at (b3b12)     {$\bar{3}\bar{1}2$};
\node[vertex] at (b1b23)     {$\bar{1}\bar{2}3$};
\node[vertex] at ([xshift=9pt,yshift=9pt]b13b2)     {$\bar{1}3\bar{2}$};
\node[vertex] at (b2b13)     {$\bar{2}\bar{1}3$};
\node[vertex] at ([xshift=9pt,yshift=-9pt]b2b1b3)     {$\bar{2}\bar{1}\bar{3}$};
\node[vertex] at (23b1)     {$23\bar{1}$};
\node[vertex] at (2b3b1)     {$2\bar{3}\bar{1}$};
\node[vertex] at (32b1)     {$32\bar{1}$};
\node[vertex] at (b32b1)     {$\bar{3}2\bar{1}$};
\node[vertex] at (b23b1)     {$\bar{2}3\bar{1}$};
\node[vertex] at (b2b3b1)     {$\bar{2}\bar{3}\bar{1}$};
\node[vertex] at ([xshift=8pt,yshift=8pt]3b1b2)     {$3\bar{1}\bar{2}$};
\node[vertex] at (3b2b1)     {$3\bar{2}\bar{1}$};
\node[vertex] at (b3b2b1)     {$\bar{3}\bar{2}\bar{1}$};
\end{tikzpicture}

%% file: graphics/type_a_asso_perm.tex
\begin{tikzpicture}%
[x={(0.105480cm, -0.368869cm)},
y={(0.994421cm, 0.039067cm)},
z={(0.000064cm, 0.928660cm)},
scale=0.800000,
back/.style={dotted, thin},
edge/.style={color=blue!95!black, thick},
edge2/.style={color=red!95!black, thick},
facet/.style={fill=white,fill opacity=0.550000},
facet2/.style={fill=red,fill opacity=0.250000},
vertex/.style={inner sep=0.3pt,circle}]
%
%
%% This TikZ-picture was produced with Sagemath version 10.4
%% with the command: ._tikz_3d_in_3d and parameters:
%% view = [-0.529100000000000, -0.475900000000000, -0.702600000000000]
%% angle = 104.020000000000
%% scale = 1
%% edge_color = blue!95!black
%% facet_color = white
%% opacity = 0.8
%% vertex_color = green
%% axis = False
%%
%% Coordinate of the vertices:
%%
\coordinate (1.41421, -1.63299, 0.57735) at (1.41421, -1.63299, 0.57735);
\coordinate (0.70711, -1.22474, 1.73205) at (0.70711, -1.22474, 1.73205);
\coordinate (2.12132, -0.40825, 0.57735) at (2.12132, -0.40825, 0.57735);
\coordinate (1.41421, 0.00000, 1.73205) at (1.41421, 0.00000, 1.73205);
\coordinate (-2.12132, -1.22474, 1.73205) at (-2.12132, -1.22474, 1.73205);
\coordinate (2.12132, 1.22474, -1.73205) at (2.12132, 1.22474, -1.73205);
\coordinate (0.70711, -2.04124, -0.57735) at (0.70711, -2.04124, -0.57735);
\coordinate (0.70711, -1.22474, -1.73205) at (0.70711, -1.22474, -1.73205);
\coordinate (-0.70711, -1.22474, -1.73205) at (-0.70711, -1.22474, -1.73205);
\coordinate (-0.70711, -2.04124, -0.57735) at (-0.70711, -2.04124, -0.57735);
\coordinate (0.70711, 3.67423, -1.73205) at (0.70711, 3.67423, -1.73205);
\coordinate (-3.53553, 1.22474, 1.73205) at (-3.53553, 1.22474, 1.73205);
\coordinate (-1.41421, 0.00000, -1.73205) at (-1.41421, 0.00000, -1.73205);
\coordinate (-1.41421, 4.89898, 1.73205) at (-1.41421, 4.89898, 1.73205);
%%
%%
%% Drawing edges in the back
%%
\draw[edge,back] (0.70711, -1.22474, -1.73205) -- (-0.70711, -1.22474, -1.73205);
\draw[edge,back] (-0.70711, -1.22474, -1.73205) -- (-0.70711, -2.04124, -0.57735);
\draw[edge,back] (-0.70711, -1.22474, -1.73205) -- (-1.41421, 0.00000, -1.73205);
\draw[edge,back] (0.70711, 3.67423, -1.73205) -- (-1.41421, 0.00000, -1.73205);
\draw[edge,back] (-3.53553, 1.22474, 1.73205) -- (-1.41421, 0.00000, -1.73205);
%%
%%
%% Drawing vertices in the back
%%
\node[vertex] at (-0.70711, -1.22474, -1.73205)     {};
\node[vertex] at (-1.41421, 0.00000, -1.73205)     {};
%%
%% Drawing the vertices in the front
%%
\node[vertex] at (1.41421, -1.63299, 0.57735)     {};
\node[vertex] at (0.70711, -1.22474, 1.73205)     {};
\node[vertex] at (2.12132, -0.40825, 0.57735)     {};
\node[vertex] at (1.41421, 0.00000, 1.73205)     {};
\node[vertex] at (-2.12132, -1.22474, 1.73205)     {};
\node[vertex] at (2.12132, 1.22474, -1.73205)     {};
\node[vertex] at (0.70711, -2.04124, -0.57735)     {};
\node[vertex] at (0.70711, -1.22474, -1.73205)     {};
\node[vertex] at (-0.70711, -2.04124, -0.57735)     {};
\node[vertex] at (0.70711, 3.67423, -1.73205)     {};
\node[vertex] at (-3.53553, 1.22474, 1.73205)     {};
\node[vertex] at (-1.41421, 4.89898, 1.73205)     {};
%%
%%

%%%% Permutahedron coordinates %%%%
\coordinate (-0.70711, -1.22474, -1.73205) at (-0.70711, -1.22474, -1.73205);
\coordinate (-0.70711, -2.04124, -0.57735) at (-0.70711, -2.04124, -0.57735);
\coordinate (-1.41421, 0.00000, -1.73205) at (-1.41421, 0.00000, -1.73205);
\coordinate (-1.41421, -1.63299, 0.57735) at (-1.41421, -1.63299, 0.57735);
\coordinate (-2.12132, 0.40825, -0.57735) at (-2.12132, 0.40825, -0.57735);
\coordinate (-2.12132, -0.40825, 0.57735) at (-2.12132, -0.40825, 0.57735);
\coordinate (0.70711, -1.22474, -1.73205) at (0.70711, -1.22474, -1.73205);
\coordinate (0.70711, -2.04124, -0.57735) at (0.70711, -2.04124, -0.57735);
\coordinate (-0.70711, 1.22474, -1.73205) at (-0.70711, 1.22474, -1.73205);
\coordinate (-0.70711, -1.22474, 1.73205) at (-0.70711, -1.22474, 1.73205);
\coordinate (-1.41421, 1.63299, -0.57735) at (-1.41421, 1.63299, -0.57735);
\coordinate (-1.41421, 0.00000, 1.73205) at (-1.41421, 0.00000, 1.73205);
\coordinate (1.41421, 0.00000, -1.73205) at (1.41421, 0.00000, -1.73205);
\coordinate (1.41421, -1.63299, 0.57735) at (1.41421, -1.63299, 0.57735);
\coordinate (0.70711, 1.22474, -1.73205) at (0.70711, 1.22474, -1.73205);
\coordinate (0.70711, -1.22474, 1.73205) at (0.70711, -1.22474, 1.73205);
\coordinate (-0.70711, 2.04124, 0.57735) at (-0.70711, 2.04124, 0.57735);
\coordinate (-0.70711, 1.22474, 1.73205) at (-0.70711, 1.22474, 1.73205);
\coordinate (2.12132, 0.40825, -0.57735) at (2.12132, 0.40825, -0.57735);
\coordinate (2.12132, -0.40825, 0.57735) at (2.12132, -0.40825, 0.57735);
\coordinate (1.41421, 1.63299, -0.57735) at (1.41421, 1.63299, -0.57735);
\coordinate (1.41421, 0.00000, 1.73205) at (1.41421, 0.00000, 1.73205);
\coordinate (0.70711, 2.04124, 0.57735) at (0.70711, 2.04124, 0.57735);
\coordinate (0.70711, 1.22474, 1.73205) at (0.70711, 1.22474, 1.73205);
%%
%%
%% Drawing edges in the back
%%
%\draw[edge2,back] (-0.70711, -1.22474, -1.73205) -- (-0.70711, -2.04124, -0.57735);
%\draw[edge2,back] (-0.70711, -1.22474, -1.73205) -- (-1.41421, 0.00000, -1.73205);
%\draw[edge2,back] (-0.70711, -1.22474, -1.73205) -- (0.70711, -1.22474, -1.73205);
%\draw[edge2,back] (-1.41421, 0.00000, -1.73205) -- (-2.12132, 0.40825, -0.57735);
%\draw[edge2,back] (-1.41421, 0.00000, -1.73205) -- (-0.70711, 1.22474, -1.73205);
\draw[edge2,back] (-1.41421, -1.63299, 0.57735) -- (-2.12132, -0.40825, 0.57735);
\draw[edge2,back] (-2.12132, 0.40825, -0.57735) -- (-2.12132, -0.40825, 0.57735);
\draw[edge2,back] (-2.12132, 0.40825, -0.57735) -- (-1.41421, 1.63299, -0.57735);
\draw[edge2,back] (-2.12132, -0.40825, 0.57735) -- (-1.41421, 0.00000, 1.73205);
\draw[edge2,back] (-0.70711, 1.22474, -1.73205) -- (-1.41421, 1.63299, -0.57735);
\draw[edge2,back] (-0.70711, 1.22474, -1.73205) -- (0.70711, 1.22474, -1.73205);
\draw[edge2,back] (-1.41421, 1.63299, -0.57735) -- (-0.70711, 2.04124, 0.57735);
%%
%%
%% Drawing vertices in the back
%%
\node[vertex] at (-2.12132, -0.40825, 0.57735)     {};
\node[vertex] at (-0.70711, -1.22474, -1.73205)     {};
\node[vertex] at (-1.41421, 0.00000, -1.73205)     {};
\node[vertex] at (-2.12132, 0.40825, -0.57735)     {};
\node[vertex] at (-0.70711, 1.22474, -1.73205)     {};
\node[vertex] at (-1.41421, 1.63299, -0.57735)     {};
%%
%%
%% Drawing the facets
%%
\fill[facet2] (1.41421, 0.00000, 1.73205) -- (0.70711, -1.22474, 1.73205) -- (1.41421, -1.63299, 0.57735) -- (2.12132, -0.40825, 0.57735) -- cycle {};
\fill[facet2] (0.70711, 1.22474, 1.73205) -- (-0.70711, 1.22474, 1.73205) -- (-1.41421, 0.00000, 1.73205) -- (-0.70711, -1.22474, 1.73205) -- (0.70711, -1.22474, 1.73205) -- (1.41421, 0.00000, 1.73205) -- cycle {};
\fill[facet2] (0.70711, -1.22474, 1.73205) -- (-0.70711, -1.22474, 1.73205) -- (-1.41421, -1.63299, 0.57735) -- (-0.70711, -2.04124, -0.57735) -- (0.70711, -2.04124, -0.57735) -- (1.41421, -1.63299, 0.57735) -- cycle {};
\fill[facet2] (0.70711, 1.22474, 1.73205) -- (-0.70711, 1.22474, 1.73205) -- (-0.70711, 2.04124, 0.57735) -- (0.70711, 2.04124, 0.57735) -- cycle {};
\fill[facet2] (1.41421, 1.63299, -0.57735) -- (0.70711, 1.22474, -1.73205) -- (1.41421, 0.00000, -1.73205) -- (2.12132, 0.40825, -0.57735) -- cycle {};
\fill[facet2] (2.12132, -0.40825, 0.57735) -- (1.41421, -1.63299, 0.57735) -- (0.70711, -2.04124, -0.57735) -- (0.70711, -1.22474, -1.73205) -- (1.41421, 0.00000, -1.73205) -- (2.12132, 0.40825, -0.57735) -- cycle {};
\fill[facet2] (0.70711, 1.22474, 1.73205) -- (1.41421, 0.00000, 1.73205) -- (2.12132, -0.40825, 0.57735) -- (2.12132, 0.40825, -0.57735) -- (1.41421, 1.63299, -0.57735) -- (0.70711, 2.04124, 0.57735) -- cycle {};
%%
%%
%%
%% Drawing the associahedron facets
%%
\fill[facet] (-1.41421, 4.89898, 1.73205) -- (1.41421, 0.00000, 1.73205) -- (2.12132, -0.40825, 0.57735) -- (2.12132, 1.22474, -1.73205) -- (0.70711, 3.67423, -1.73205) -- cycle {};
\fill[facet] (0.70711, -1.22474, -1.73205) -- (2.12132, 1.22474, -1.73205) -- (2.12132, -0.40825, 0.57735) -- (1.41421, -1.63299, 0.57735) -- (0.70711, -2.04124, -0.57735) -- cycle {};
\fill[facet] (1.41421, 0.00000, 1.73205) -- (0.70711, -1.22474, 1.73205) -- (1.41421, -1.63299, 0.57735) -- (2.12132, -0.40825, 0.57735) -- cycle {};
\fill[facet] (-1.41421, 4.89898, 1.73205) -- (1.41421, 0.00000, 1.73205) -- (0.70711, -1.22474, 1.73205) -- (-2.12132, -1.22474, 1.73205) -- (-3.53553, 1.22474, 1.73205) -- cycle {};
\fill[facet] (-0.70711, -2.04124, -0.57735) -- (-2.12132, -1.22474, 1.73205) -- (0.70711, -1.22474, 1.73205) -- (1.41421, -1.63299, 0.57735) -- (0.70711, -2.04124, -0.57735) -- cycle {};
%%
%%
%% Drawing edges in the front
%%
\draw[edge2] (-0.70711, -2.04124, -0.57735) -- (-1.41421, -1.63299, 0.57735);
\draw[edge2] (-0.70711, -2.04124, -0.57735) -- (0.70711, -2.04124, -0.57735);
\draw[edge2] (-1.41421, -1.63299, 0.57735) -- (-0.70711, -1.22474, 1.73205);
\draw[edge2] (0.70711, -1.22474, -1.73205) -- (0.70711, -2.04124, -0.57735);
\draw[edge2] (0.70711, -1.22474, -1.73205) -- (1.41421, 0.00000, -1.73205);
\draw[edge2] (0.70711, -2.04124, -0.57735) -- (1.41421, -1.63299, 0.57735);
\draw[edge2] (-0.70711, -1.22474, 1.73205) -- (-1.41421, 0.00000, 1.73205);
\draw[edge2] (-0.70711, -1.22474, 1.73205) -- (0.70711, -1.22474, 1.73205);
\draw[edge2] (-1.41421, 0.00000, 1.73205) -- (-0.70711, 1.22474, 1.73205);
\draw[edge2] (1.41421, 0.00000, -1.73205) -- (0.70711, 1.22474, -1.73205);
\draw[edge2] (1.41421, 0.00000, -1.73205) -- (2.12132, 0.40825, -0.57735);
\draw[edge2] (1.41421, -1.63299, 0.57735) -- (0.70711, -1.22474, 1.73205);
\draw[edge2] (1.41421, -1.63299, 0.57735) -- (2.12132, -0.40825, 0.57735);
\draw[edge2] (0.70711, 1.22474, -1.73205) -- (1.41421, 1.63299, -0.57735);
\draw[edge2] (0.70711, -1.22474, 1.73205) -- (1.41421, 0.00000, 1.73205);
\draw[edge2] (-0.70711, 2.04124, 0.57735) -- (-0.70711, 1.22474, 1.73205);
\draw[edge2] (-0.70711, 2.04124, 0.57735) -- (0.70711, 2.04124, 0.57735);
\draw[edge2] (-0.70711, 1.22474, 1.73205) -- (0.70711, 1.22474, 1.73205);
\draw[edge2] (2.12132, 0.40825, -0.57735) -- (2.12132, -0.40825, 0.57735);
\draw[edge2] (2.12132, 0.40825, -0.57735) -- (1.41421, 1.63299, -0.57735);
\draw[edge2] (2.12132, -0.40825, 0.57735) -- (1.41421, 0.00000, 1.73205);
\draw[edge2] (1.41421, 1.63299, -0.57735) -- (0.70711, 2.04124, 0.57735);
\draw[edge2] (1.41421, 0.00000, 1.73205) -- (0.70711, 1.22474, 1.73205);
\draw[edge2] (0.70711, 2.04124, 0.57735) -- (0.70711, 1.22474, 1.73205);
%%
%%
%% Drawing edges in the front
%%
\draw[edge] (1.41421, -1.63299, 0.57735) -- (0.70711, -1.22474, 1.73205);
\draw[edge] (1.41421, -1.63299, 0.57735) -- (2.12132, -0.40825, 0.57735);
\draw[edge] (1.41421, -1.63299, 0.57735) -- (0.70711, -2.04124, -0.57735);
\draw[edge] (0.70711, -1.22474, 1.73205) -- (1.41421, 0.00000, 1.73205);
\draw[edge] (0.70711, -1.22474, 1.73205) -- (-2.12132, -1.22474, 1.73205);
\draw[edge] (2.12132, -0.40825, 0.57735) -- (1.41421, 0.00000, 1.73205);
\draw[edge] (2.12132, -0.40825, 0.57735) -- (2.12132, 1.22474, -1.73205);
\draw[edge] (1.41421, 0.00000, 1.73205) -- (-1.41421, 4.89898, 1.73205);
\draw[edge] (-2.12132, -1.22474, 1.73205) -- (-0.70711, -2.04124, -0.57735);
\draw[edge] (-2.12132, -1.22474, 1.73205) -- (-3.53553, 1.22474, 1.73205);
\draw[edge] (2.12132, 1.22474, -1.73205) -- (0.70711, -1.22474, -1.73205);
\draw[edge] (2.12132, 1.22474, -1.73205) -- (0.70711, 3.67423, -1.73205);
\draw[edge] (0.70711, -2.04124, -0.57735) -- (0.70711, -1.22474, -1.73205);
\draw[edge] (0.70711, -2.04124, -0.57735) -- (-0.70711, -2.04124, -0.57735);
\draw[edge] (0.70711, 3.67423, -1.73205) -- (-1.41421, 4.89898, 1.73205);
\draw[edge] (-3.53553, 1.22474, 1.73205) -- (-1.41421, 4.89898, 1.73205);
%%
%%
%% Drawing the vertices in the front
%%
\node[vertex] at (-0.70711, -2.04124, -0.57735)     {};
\node[vertex] at (-1.41421, -1.63299, 0.57735)     {};
\node[vertex] at (0.70711, -1.22474, -1.73205)     {};
\node[vertex] at (0.70711, -2.04124, -0.57735)     {};
\node[vertex] at (-0.70711, -1.22474, 1.73205)     {};
\node[vertex] at (-1.41421, 0.00000, 1.73205)     {};
\node[vertex] at (1.41421, 0.00000, -1.73205)     {};
\node[vertex] at (1.41421, -1.63299, 0.57735)     {};
\node[vertex] at (0.70711, 1.22474, -1.73205)     {};
\node[vertex] at (0.70711, -1.22474, 1.73205)     {};
\node[vertex] at (-0.70711, 2.04124, 0.57735)     {};
\node[vertex] at (-0.70711, 1.22474, 1.73205)     {};
\node[vertex] at (2.12132, 0.40825, -0.57735)     {};
\node[vertex] at (2.12132, -0.40825, 0.57735)     {};
\node[vertex] at (1.41421, 1.63299, -0.57735)     {};
\node[vertex] at (1.41421, 0.00000, 1.73205)     {};
\node[vertex] at (0.70711, 2.04124, 0.57735)     {};
\node[vertex] at (0.70711, 1.22474, 1.73205)     {};
\end{tikzpicture}

%% file: graphics/type_b_asso_perm.tex
\begin{tikzpicture}%
[x={(-0.801330cm, 0.247442cm)},
y={(-0.598223cm, -0.331531cm)},
z={(-0.000051cm, 0.910417cm)},
scale=0.350000,
back/.style={dotted, thin},
edge/.style={color=blue!95!black, thick},
edge2/.style={color=red!95!black, thick},
facet/.style={fill=white,fill opacity=0.550000},
facet2/.style={fill=red,fill opacity=0.250000},
vertex/.style={inner sep=0.3pt,circle}]
%
%
%% This TikZ-picture was produced with Sagemath version 10.4
%% with the command: ._tikz_3d_in_3d and parameters:
%% view = [0.177000000000000, -0.532900000000000, -0.827500000000000]
%% angle = 210.730000000000
%% scale = 1
%% edge_color = blue!95!black
%% facet_color = white
%% opacity = 0.8
%% vertex_color = green
%% axis = False
%%
%% Coordinate of the vertices:
%%
\coordinate (3.82843, 2.41421, -1.00000) at (3.82843, 2.41421, -1.00000);
\coordinate (3.82843, 2.41421, 1.00000) at (3.82843, 2.41421, 1.00000);
\coordinate (3.82843, 1.00000, -2.41421) at (3.82843, 1.00000, -2.41421);
\coordinate (2.41421, 3.82843, -1.00000) at (2.41421, 3.82843, -1.00000);
\coordinate (2.41421, 3.82843, 1.00000) at (2.41421, 3.82843, 1.00000);
\coordinate (1.00000, 3.82843, 2.41421) at (1.00000, 3.82843, 2.41421);
\coordinate (3.82843, -1.00000, -2.41421) at (3.82843, -1.00000, -2.41421);
\coordinate (3.82843, -0.41421, 3.82843) at (3.82843, -0.41421, 3.82843);
\coordinate (1.00000, 2.41421, 3.82843) at (1.00000, 2.41421, 3.82843);
\coordinate (3.82843, -7.24264, 3.82843) at (3.82843, -7.24264, 3.82843);
\coordinate (2.41421, 1.00000, -3.82843) at (2.41421, 1.00000, -3.82843);
\coordinate (-0.41421, 3.82843, -3.82843) at (-0.41421, 3.82843, -3.82843);
\coordinate (2.41421, -1.00000, -3.82843) at (2.41421, -1.00000, -3.82843);
\coordinate (-1.00000, 3.82843, 2.41421) at (-1.00000, 3.82843, 2.41421);
\coordinate (-1.00000, 2.41421, 3.82843) at (-1.00000, 2.41421, 3.82843);
\coordinate (-7.24264, 3.82843, -3.82843) at (-7.24264, 3.82843, -3.82843);
\coordinate (-2.41421, -13.48528, 3.82843) at (-2.41421, -13.48528, 3.82843);
\coordinate (-10.07107, -13.48528, -3.82843) at (-10.07107, -13.48528, -3.82843);
\coordinate (-2.41421, 1.00000, 3.82843) at (-2.41421, 1.00000, 3.82843);
\coordinate (-10.07107, 1.00000, -3.82843) at (-10.07107, 1.00000, -3.82843);
%%
%%
%% Drawing edges in the back
%%
\draw[edge,back] (3.82843, 2.41421, -1.00000) -- (3.82843, 1.00000, -2.41421);
\draw[edge,back] (3.82843, 1.00000, -2.41421) -- (3.82843, -1.00000, -2.41421);
\draw[edge,back] (3.82843, 1.00000, -2.41421) -- (2.41421, 1.00000, -3.82843);
\draw[edge,back] (3.82843, -1.00000, -2.41421) -- (3.82843, -7.24264, 3.82843);
\draw[edge,back] (3.82843, -1.00000, -2.41421) -- (2.41421, -1.00000, -3.82843);
\draw[edge,back] (2.41421, 1.00000, -3.82843) -- (-0.41421, 3.82843, -3.82843);
\draw[edge,back] (2.41421, 1.00000, -3.82843) -- (2.41421, -1.00000, -3.82843);
\draw[edge,back] (2.41421, -1.00000, -3.82843) -- (-10.07107, -13.48528, -3.82843);
%%
%%
%% Drawing vertices in the back
%%
\node[vertex] at (3.82843, 1.00000, -2.41421)     {};
\node[vertex] at (3.82843, -1.00000, -2.41421)     {};
\node[vertex] at (2.41421, 1.00000, -3.82843)     {};
\node[vertex] at (2.41421, -1.00000, -3.82843)     {};
%%
%%
%% Coordinate of the vertices:
%%
\coordinate (2.41421, 1.00000, 3.82843) at (2.41421, 1.00000, 3.82843);
\coordinate (1.00000, 2.41421, 3.82843) at (1.00000, 2.41421, 3.82843);
\coordinate (1.00000, 3.82843, 2.41421) at (1.00000, 3.82843, 2.41421);
\coordinate (2.41421, 1.00000, -3.82843) at (2.41421, 1.00000, -3.82843);
\coordinate (1.00000, 2.41421, -3.82843) at (1.00000, 2.41421, -3.82843);
\coordinate (1.00000, 3.82843, -2.41421) at (1.00000, 3.82843, -2.41421);
\coordinate (2.41421, 3.82843, 1.00000) at (2.41421, 3.82843, 1.00000);
\coordinate (2.41421, 3.82843, -1.00000) at (2.41421, 3.82843, -1.00000);
\coordinate (1.00000, -2.41421, 3.82843) at (1.00000, -2.41421, 3.82843);
\coordinate (1.00000, -3.82843, 2.41421) at (1.00000, -3.82843, 2.41421);
\coordinate (2.41421, -1.00000, 3.82843) at (2.41421, -1.00000, 3.82843);
\coordinate (1.00000, -2.41421, -3.82843) at (1.00000, -2.41421, -3.82843);
\coordinate (1.00000, -3.82843, -2.41421) at (1.00000, -3.82843, -2.41421);
\coordinate (2.41421, -1.00000, -3.82843) at (2.41421, -1.00000, -3.82843);
\coordinate (3.82843, 1.00000, 2.41421) at (3.82843, 1.00000, 2.41421);
\coordinate (3.82843, 1.00000, -2.41421) at (3.82843, 1.00000, -2.41421);
\coordinate (3.82843, 2.41421, 1.00000) at (3.82843, 2.41421, 1.00000);
\coordinate (3.82843, 2.41421, -1.00000) at (3.82843, 2.41421, -1.00000);
\coordinate (3.82843, -1.00000, 2.41421) at (3.82843, -1.00000, 2.41421);
\coordinate (3.82843, -1.00000, -2.41421) at (3.82843, -1.00000, -2.41421);
\coordinate (2.41421, -3.82843, 1.00000) at (2.41421, -3.82843, 1.00000);
\coordinate (2.41421, -3.82843, -1.00000) at (2.41421, -3.82843, -1.00000);
\coordinate (3.82843, -2.41421, 1.00000) at (3.82843, -2.41421, 1.00000);
\coordinate (3.82843, -2.41421, -1.00000) at (3.82843, -2.41421, -1.00000);
\coordinate (-2.41421, 1.00000, 3.82843) at (-2.41421, 1.00000, 3.82843);
\coordinate (-1.00000, 2.41421, 3.82843) at (-1.00000, 2.41421, 3.82843);
\coordinate (-1.00000, 3.82843, 2.41421) at (-1.00000, 3.82843, 2.41421);
\coordinate (-2.41421, 1.00000, -3.82843) at (-2.41421, 1.00000, -3.82843);
\coordinate (-1.00000, 2.41421, -3.82843) at (-1.00000, 2.41421, -3.82843);
\coordinate (-1.00000, 3.82843, -2.41421) at (-1.00000, 3.82843, -2.41421);
\coordinate (-2.41421, 3.82843, 1.00000) at (-2.41421, 3.82843, 1.00000);
\coordinate (-2.41421, 3.82843, -1.00000) at (-2.41421, 3.82843, -1.00000);
\coordinate (-1.00000, -2.41421, 3.82843) at (-1.00000, -2.41421, 3.82843);
\coordinate (-1.00000, -3.82843, 2.41421) at (-1.00000, -3.82843, 2.41421);
\coordinate (-2.41421, -1.00000, 3.82843) at (-2.41421, -1.00000, 3.82843);
\coordinate (-1.00000, -2.41421, -3.82843) at (-1.00000, -2.41421, -3.82843);
\coordinate (-1.00000, -3.82843, -2.41421) at (-1.00000, -3.82843, -2.41421);
\coordinate (-2.41421, -1.00000, -3.82843) at (-2.41421, -1.00000, -3.82843);
\coordinate (-3.82843, 1.00000, 2.41421) at (-3.82843, 1.00000, 2.41421);
\coordinate (-3.82843, 1.00000, -2.41421) at (-3.82843, 1.00000, -2.41421);
\coordinate (-3.82843, 2.41421, 1.00000) at (-3.82843, 2.41421, 1.00000);
\coordinate (-3.82843, 2.41421, -1.00000) at (-3.82843, 2.41421, -1.00000);
\coordinate (-3.82843, -1.00000, 2.41421) at (-3.82843, -1.00000, 2.41421);
\coordinate (-3.82843, -1.00000, -2.41421) at (-3.82843, -1.00000, -2.41421);
\coordinate (-2.41421, -3.82843, 1.00000) at (-2.41421, -3.82843, 1.00000);
\coordinate (-2.41421, -3.82843, -1.00000) at (-2.41421, -3.82843, -1.00000);
\coordinate (-3.82843, -2.41421, 1.00000) at (-3.82843, -2.41421, 1.00000);
\coordinate (-3.82843, -2.41421, -1.00000) at (-3.82843, -2.41421, -1.00000);
%%
%%
%% Drawing edges in the back
%%
%\draw[edge2,back] (2.41421, 1.00000, -3.82843) -- (1.00000, 2.41421, -3.82843);
%\draw[edge2,back] (2.41421, 1.00000, -3.82843) -- (2.41421, -1.00000, -3.82843);
%\draw[edge2,back] (2.41421, 1.00000, -3.82843) -- (3.82843, 1.00000, -2.41421);
\draw[edge2,back] (1.00000, -2.41421, 3.82843) -- (1.00000, -3.82843, 2.41421);
\draw[edge2,back] (1.00000, -3.82843, 2.41421) -- (2.41421, -3.82843, 1.00000);
\draw[edge2,back] (1.00000, -3.82843, 2.41421) -- (-1.00000, -3.82843, 2.41421);
\draw[edge2,back] (2.41421, -1.00000, 3.82843) -- (3.82843, -1.00000, 2.41421);
\draw[edge2,back] (1.00000, -2.41421, -3.82843) -- (1.00000, -3.82843, -2.41421);
%\draw[edge2,back] (1.00000, -2.41421, -3.82843) -- (2.41421, -1.00000, -3.82843);
\draw[edge2,back] (1.00000, -2.41421, -3.82843) -- (-1.00000, -2.41421, -3.82843);
\draw[edge2,back] (1.00000, -3.82843, -2.41421) -- (2.41421, -3.82843, -1.00000);
\draw[edge2,back] (1.00000, -3.82843, -2.41421) -- (-1.00000, -3.82843, -2.41421);
%\draw[edge2,back] (2.41421, -1.00000, -3.82843) -- (3.82843, -1.00000, -2.41421);
\draw[edge2,back] (3.82843, 1.00000, 2.41421) -- (3.82843, -1.00000, 2.41421);
%\draw[edge2,back] (3.82843, 1.00000, -2.41421) -- (3.82843, 2.41421, -1.00000);
%\draw[edge2,back] (3.82843, 1.00000, -2.41421) -- (3.82843, -1.00000, -2.41421);
\draw[edge2,back] (3.82843, -1.00000, 2.41421) -- (3.82843, -2.41421, 1.00000);
%\draw[edge2,back] (3.82843, -1.00000, -2.41421) -- (3.82843, -2.41421, -1.00000);
\draw[edge2,back] (2.41421, -3.82843, 1.00000) -- (2.41421, -3.82843, -1.00000);
\draw[edge2,back] (2.41421, -3.82843, 1.00000) -- (3.82843, -2.41421, 1.00000);
\draw[edge2,back] (2.41421, -3.82843, -1.00000) -- (3.82843, -2.41421, -1.00000);
\draw[edge2,back] (3.82843, -2.41421, 1.00000) -- (3.82843, -2.41421, -1.00000);
\draw[edge2,back] (-1.00000, -2.41421, -3.82843) -- (-1.00000, -3.82843, -2.41421);
\draw[edge2,back] (-1.00000, -2.41421, -3.82843) -- (-2.41421, -1.00000, -3.82843);
\draw[edge2,back] (-1.00000, -3.82843, -2.41421) -- (-2.41421, -3.82843, -1.00000);
\draw[edge2,back] (-2.41421, -3.82843, 1.00000) -- (-2.41421, -3.82843, -1.00000);
\draw[edge2,back] (-2.41421, -3.82843, -1.00000) -- (-3.82843, -2.41421, -1.00000);
%%
%%
%% Drawing vertices in the back
%%
\node[vertex] at (1.00000, -3.82843, 2.41421)     {};
\node[vertex] at (1.00000, -2.41421, -3.82843)     {};
\node[vertex] at (1.00000, -3.82843, -2.41421)     {};
\node[vertex] at (-1.00000, -2.41421, -3.82843)     {};
\node[vertex] at (-1.00000, -3.82843, -2.41421)     {};
\node[vertex] at (2.41421, 1.00000, -3.82843)     {};
\node[vertex] at (2.41421, -1.00000, -3.82843)     {};
\node[vertex] at (2.41421, -3.82843, 1.00000)     {};
\node[vertex] at (2.41421, -3.82843, -1.00000)     {};
\node[vertex] at (-2.41421, -3.82843, -1.00000)     {};
\node[vertex] at (3.82843, -1.00000, 2.41421)     {};
\node[vertex] at (3.82843, 1.00000, -2.41421)     {};
\node[vertex] at (3.82843, -1.00000, -2.41421)     {};
\node[vertex] at (3.82843, -2.41421, 1.00000)     {};
\node[vertex] at (3.82843, -2.41421, -1.00000)     {};
%%
%%
%% Drawing the permutahedron facets
%%
\fill[facet2] (-1.00000, 3.82843, 2.41421) -- (1.00000, 3.82843, 2.41421) -- (1.00000, 2.41421, 3.82843) -- (-1.00000, 2.41421, 3.82843) -- cycle {};
\fill[facet2] (-1.00000, 3.82843, -2.41421) -- (1.00000, 3.82843, -2.41421) -- (1.00000, 2.41421, -3.82843) -- (-1.00000, 2.41421, -3.82843) -- cycle {};
\fill[facet2] (-2.41421, -1.00000, 3.82843) -- (-2.41421, 1.00000, 3.82843) -- (-1.00000, 2.41421, 3.82843) -- (1.00000, 2.41421, 3.82843) -- (2.41421, 1.00000, 3.82843) -- (2.41421, -1.00000, 3.82843) -- (1.00000, -2.41421, 3.82843) -- (-1.00000, -2.41421, 3.82843) -- cycle {};
\fill[facet2] (-2.41421, 3.82843, -1.00000) -- (-1.00000, 3.82843, -2.41421) -- (1.00000, 3.82843, -2.41421) -- (2.41421, 3.82843, -1.00000) -- (2.41421, 3.82843, 1.00000) -- (1.00000, 3.82843, 2.41421) -- (-1.00000, 3.82843, 2.41421) -- (-2.41421, 3.82843, 1.00000) -- cycle {};
\fill[facet2] (-3.82843, -1.00000, 2.41421) -- (-2.41421, -1.00000, 3.82843) -- (-2.41421, 1.00000, 3.82843) -- (-3.82843, 1.00000, 2.41421) -- cycle {};
\fill[facet2] (-3.82843, -1.00000, -2.41421) -- (-2.41421, -1.00000, -3.82843) -- (-2.41421, 1.00000, -3.82843) -- (-3.82843, 1.00000, -2.41421) -- cycle {};
\fill[facet2] (3.82843, 2.41421, 1.00000) -- (2.41421, 3.82843, 1.00000) -- (1.00000, 3.82843, 2.41421) -- (1.00000, 2.41421, 3.82843) -- (2.41421, 1.00000, 3.82843) -- (3.82843, 1.00000, 2.41421) -- cycle {};
\fill[facet2] (3.82843, 2.41421, -1.00000) -- (2.41421, 3.82843, -1.00000) -- (2.41421, 3.82843, 1.00000) -- (3.82843, 2.41421, 1.00000) -- cycle {};
\fill[facet2] (-3.82843, 2.41421, 1.00000) -- (-2.41421, 3.82843, 1.00000) -- (-1.00000, 3.82843, 2.41421) -- (-1.00000, 2.41421, 3.82843) -- (-2.41421, 1.00000, 3.82843) -- (-3.82843, 1.00000, 2.41421) -- cycle {};
\fill[facet2] (-3.82843, 2.41421, -1.00000) -- (-2.41421, 3.82843, -1.00000) -- (-1.00000, 3.82843, -2.41421) -- (-1.00000, 2.41421, -3.82843) -- (-2.41421, 1.00000, -3.82843) -- (-3.82843, 1.00000, -2.41421) -- cycle {};
\fill[facet2] (-3.82843, 2.41421, -1.00000) -- (-2.41421, 3.82843, -1.00000) -- (-2.41421, 3.82843, 1.00000) -- (-3.82843, 2.41421, 1.00000) -- cycle {};
\fill[facet2] (-3.82843, -2.41421, 1.00000) -- (-3.82843, -1.00000, 2.41421) -- (-2.41421, -1.00000, 3.82843) -- (-1.00000, -2.41421, 3.82843) -- (-1.00000, -3.82843, 2.41421) -- (-2.41421, -3.82843, 1.00000) -- cycle {};
\fill[facet2] (-3.82843, -2.41421, -1.00000) -- (-3.82843, -1.00000, -2.41421) -- (-3.82843, 1.00000, -2.41421) -- (-3.82843, 2.41421, -1.00000) -- (-3.82843, 2.41421, 1.00000) -- (-3.82843, 1.00000, 2.41421) -- (-3.82843, -1.00000, 2.41421) -- (-3.82843, -2.41421, 1.00000) -- cycle {};
%%
%%
%% Drawing the associahedron facets
%%
\fill[facet] (2.41421, 3.82843, 1.00000) -- (3.82843, 2.41421, 1.00000) -- (3.82843, 2.41421, -1.00000) -- (2.41421, 3.82843, -1.00000) -- cycle {};
\fill[facet] (1.00000, 2.41421, 3.82843) -- (1.00000, 3.82843, 2.41421) -- (2.41421, 3.82843, 1.00000) -- (3.82843, 2.41421, 1.00000) -- (3.82843, -0.41421, 3.82843) -- cycle {};
\fill[facet] (-1.00000, 2.41421, 3.82843) -- (1.00000, 2.41421, 3.82843) -- (1.00000, 3.82843, 2.41421) -- (-1.00000, 3.82843, 2.41421) -- cycle {};
\fill[facet] (-7.24264, 3.82843, -3.82843) -- (-0.41421, 3.82843, -3.82843) -- (2.41421, 3.82843, -1.00000) -- (2.41421, 3.82843, 1.00000) -- (1.00000, 3.82843, 2.41421) -- (-1.00000, 3.82843, 2.41421) -- cycle {};
\fill[facet] (-2.41421, 1.00000, 3.82843) -- (-1.00000, 2.41421, 3.82843) -- (1.00000, 2.41421, 3.82843) -- (3.82843, -0.41421, 3.82843) -- (3.82843, -7.24264, 3.82843) -- (-2.41421, -13.48528, 3.82843) -- cycle {};
\fill[facet] (-10.07107, 1.00000, -3.82843) -- (-7.24264, 3.82843, -3.82843) -- (-1.00000, 3.82843, 2.41421) -- (-1.00000, 2.41421, 3.82843) -- (-2.41421, 1.00000, 3.82843) -- cycle {};
\fill[facet] (-10.07107, 1.00000, -3.82843) -- (-10.07107, -13.48528, -3.82843) -- (-2.41421, -13.48528, 3.82843) -- (-2.41421, 1.00000, 3.82843) -- cycle {};
%%
%%
%% Drawing permutahedron edges in the front
%%
\draw[edge2] (2.41421, 1.00000, 3.82843) -- (1.00000, 2.41421, 3.82843);
\draw[edge2] (2.41421, 1.00000, 3.82843) -- (2.41421, -1.00000, 3.82843);
\draw[edge2] (2.41421, 1.00000, 3.82843) -- (3.82843, 1.00000, 2.41421);
\draw[edge2] (1.00000, 2.41421, 3.82843) -- (1.00000, 3.82843, 2.41421);
\draw[edge2] (1.00000, 2.41421, 3.82843) -- (-1.00000, 2.41421, 3.82843);
\draw[edge2] (1.00000, 3.82843, 2.41421) -- (2.41421, 3.82843, 1.00000);
\draw[edge2] (1.00000, 3.82843, 2.41421) -- (-1.00000, 3.82843, 2.41421);
\draw[edge2] (1.00000, 2.41421, -3.82843) -- (1.00000, 3.82843, -2.41421);
\draw[edge2] (1.00000, 2.41421, -3.82843) -- (-1.00000, 2.41421, -3.82843);
\draw[edge2] (1.00000, 3.82843, -2.41421) -- (2.41421, 3.82843, -1.00000);
\draw[edge2] (1.00000, 3.82843, -2.41421) -- (-1.00000, 3.82843, -2.41421);
\draw[edge2] (2.41421, 3.82843, 1.00000) -- (2.41421, 3.82843, -1.00000);
\draw[edge2] (2.41421, 3.82843, 1.00000) -- (3.82843, 2.41421, 1.00000);
\draw[edge2] (2.41421, 3.82843, -1.00000) -- (3.82843, 2.41421, -1.00000);
\draw[edge2] (1.00000, -2.41421, 3.82843) -- (2.41421, -1.00000, 3.82843);
\draw[edge2] (1.00000, -2.41421, 3.82843) -- (-1.00000, -2.41421, 3.82843);
\draw[edge2] (3.82843, 1.00000, 2.41421) -- (3.82843, 2.41421, 1.00000);
\draw[edge2] (3.82843, 2.41421, 1.00000) -- (3.82843, 2.41421, -1.00000);
\draw[edge2] (-2.41421, 1.00000, 3.82843) -- (-1.00000, 2.41421, 3.82843);
\draw[edge2] (-2.41421, 1.00000, 3.82843) -- (-2.41421, -1.00000, 3.82843);
\draw[edge2] (-2.41421, 1.00000, 3.82843) -- (-3.82843, 1.00000, 2.41421);
\draw[edge2] (-1.00000, 2.41421, 3.82843) -- (-1.00000, 3.82843, 2.41421);
\draw[edge2] (-1.00000, 3.82843, 2.41421) -- (-2.41421, 3.82843, 1.00000);
\draw[edge2] (-2.41421, 1.00000, -3.82843) -- (-1.00000, 2.41421, -3.82843);
\draw[edge2] (-2.41421, 1.00000, -3.82843) -- (-2.41421, -1.00000, -3.82843);
\draw[edge2] (-2.41421, 1.00000, -3.82843) -- (-3.82843, 1.00000, -2.41421);
\draw[edge2] (-1.00000, 2.41421, -3.82843) -- (-1.00000, 3.82843, -2.41421);
\draw[edge2] (-1.00000, 3.82843, -2.41421) -- (-2.41421, 3.82843, -1.00000);
\draw[edge2] (-2.41421, 3.82843, 1.00000) -- (-2.41421, 3.82843, -1.00000);
\draw[edge2] (-2.41421, 3.82843, 1.00000) -- (-3.82843, 2.41421, 1.00000);
\draw[edge2] (-2.41421, 3.82843, -1.00000) -- (-3.82843, 2.41421, -1.00000);
\draw[edge2] (-1.00000, -2.41421, 3.82843) -- (-1.00000, -3.82843, 2.41421);
\draw[edge2] (-1.00000, -2.41421, 3.82843) -- (-2.41421, -1.00000, 3.82843);
\draw[edge2] (-1.00000, -3.82843, 2.41421) -- (-2.41421, -3.82843, 1.00000);
\draw[edge2] (-2.41421, -1.00000, 3.82843) -- (-3.82843, -1.00000, 2.41421);
\draw[edge2] (-2.41421, -1.00000, -3.82843) -- (-3.82843, -1.00000, -2.41421);
\draw[edge2] (-3.82843, 1.00000, 2.41421) -- (-3.82843, 2.41421, 1.00000);
\draw[edge2] (-3.82843, 1.00000, 2.41421) -- (-3.82843, -1.00000, 2.41421);
\draw[edge2] (-3.82843, 1.00000, -2.41421) -- (-3.82843, 2.41421, -1.00000);
\draw[edge2] (-3.82843, 1.00000, -2.41421) -- (-3.82843, -1.00000, -2.41421);
\draw[edge2] (-3.82843, 2.41421, 1.00000) -- (-3.82843, 2.41421, -1.00000);
\draw[edge2] (-3.82843, -1.00000, 2.41421) -- (-3.82843, -2.41421, 1.00000);
\draw[edge2] (-3.82843, -1.00000, -2.41421) -- (-3.82843, -2.41421, -1.00000);
\draw[edge2] (-2.41421, -3.82843, 1.00000) -- (-3.82843, -2.41421, 1.00000);
\draw[edge2] (-3.82843, -2.41421, 1.00000) -- (-3.82843, -2.41421, -1.00000);
%%
%%
%% Drawing edges in the front
%%
\draw[edge] (3.82843, 2.41421, -1.00000) -- (3.82843, 2.41421, 1.00000);
\draw[edge] (3.82843, 2.41421, -1.00000) -- (2.41421, 3.82843, -1.00000);
\draw[edge] (3.82843, 2.41421, 1.00000) -- (2.41421, 3.82843, 1.00000);
\draw[edge] (3.82843, 2.41421, 1.00000) -- (3.82843, -0.41421, 3.82843);
\draw[edge] (2.41421, 3.82843, -1.00000) -- (2.41421, 3.82843, 1.00000);
\draw[edge] (2.41421, 3.82843, -1.00000) -- (-0.41421, 3.82843, -3.82843);
\draw[edge] (2.41421, 3.82843, 1.00000) -- (1.00000, 3.82843, 2.41421);
\draw[edge] (1.00000, 3.82843, 2.41421) -- (1.00000, 2.41421, 3.82843);
\draw[edge] (1.00000, 3.82843, 2.41421) -- (-1.00000, 3.82843, 2.41421);
\draw[edge] (3.82843, -0.41421, 3.82843) -- (1.00000, 2.41421, 3.82843);
\draw[edge] (3.82843, -0.41421, 3.82843) -- (3.82843, -7.24264, 3.82843);
\draw[edge] (1.00000, 2.41421, 3.82843) -- (-1.00000, 2.41421, 3.82843);
\draw[edge] (3.82843, -7.24264, 3.82843) -- (-2.41421, -13.48528, 3.82843);
\draw[edge] (-0.41421, 3.82843, -3.82843) -- (-7.24264, 3.82843, -3.82843);
\draw[edge] (-1.00000, 3.82843, 2.41421) -- (-1.00000, 2.41421, 3.82843);
\draw[edge] (-1.00000, 3.82843, 2.41421) -- (-7.24264, 3.82843, -3.82843);
\draw[edge] (-1.00000, 2.41421, 3.82843) -- (-2.41421, 1.00000, 3.82843);
\draw[edge] (-7.24264, 3.82843, -3.82843) -- (-10.07107, 1.00000, -3.82843);
\draw[edge] (-2.41421, -13.48528, 3.82843) -- (-10.07107, -13.48528, -3.82843);
\draw[edge] (-2.41421, -13.48528, 3.82843) -- (-2.41421, 1.00000, 3.82843);
\draw[edge] (-10.07107, -13.48528, -3.82843) -- (-10.07107, 1.00000, -3.82843);
\draw[edge] (-2.41421, 1.00000, 3.82843) -- (-10.07107, 1.00000, -3.82843);
%%
%%
%% Drawing the associahedron vertices in the front
%%
\node[vertex] at (3.82843, 2.41421, -1.00000)     {};
\node[vertex] at (3.82843, 2.41421, 1.00000)     {};
\node[vertex] at (2.41421, 3.82843, -1.00000)     {};
\node[vertex] at (2.41421, 3.82843, 1.00000)     {};
\node[vertex] at (1.00000, 3.82843, 2.41421)     {};
\node[vertex] at (3.82843, -0.41421, 3.82843)     {};
\node[vertex] at (1.00000, 2.41421, 3.82843)     {};
\node[vertex] at (3.82843, -7.24264, 3.82843)     {};
\node[vertex] at (-0.41421, 3.82843, -3.82843)     {};
\node[vertex] at (-1.00000, 3.82843, 2.41421)     {};
\node[vertex] at (-1.00000, 2.41421, 3.82843)     {};
\node[vertex] at (-7.24264, 3.82843, -3.82843)     {};
\node[vertex] at (-2.41421, -13.48528, 3.82843)     {};
\node[vertex] at (-10.07107, -13.48528, -3.82843)     {};
\node[vertex] at (-2.41421, 1.00000, 3.82843)     {};
\node[vertex] at (-10.07107, 1.00000, -3.82843)     {};
%%
%%
%% Drawing the permutahedron vertices in the front
%%
\node[vertex] at (2.41421, 1.00000, 3.82843)     {};
\node[vertex] at (1.00000, 2.41421, 3.82843)     {};
\node[vertex] at (1.00000, 3.82843, 2.41421)     {};
\node[vertex] at (1.00000, 2.41421, -3.82843)     {};
\node[vertex] at (1.00000, 3.82843, -2.41421)     {};
\node[vertex] at (2.41421, 3.82843, 1.00000)     {};
\node[vertex] at (2.41421, 3.82843, -1.00000)     {};
\node[vertex] at (1.00000, -2.41421, 3.82843)     {};
\node[vertex] at (2.41421, -1.00000, 3.82843)     {};
\node[vertex] at (3.82843, 1.00000, 2.41421)     {};
\node[vertex] at (3.82843, 2.41421, 1.00000)     {};
\node[vertex] at (3.82843, 2.41421, -1.00000)     {};
\node[vertex] at (-2.41421, 1.00000, 3.82843)     {};
\node[vertex] at (-1.00000, 2.41421, 3.82843)     {};
\node[vertex] at (-1.00000, 3.82843, 2.41421)     {};
\node[vertex] at (-2.41421, 1.00000, -3.82843)     {};
\node[vertex] at (-1.00000, 2.41421, -3.82843)     {};
\node[vertex] at (-1.00000, 3.82843, -2.41421)     {};
\node[vertex] at (-2.41421, 3.82843, 1.00000)     {};
\node[vertex] at (-2.41421, 3.82843, -1.00000)     {};
\node[vertex] at (-1.00000, -2.41421, 3.82843)     {};
\node[vertex] at (-1.00000, -3.82843, 2.41421)     {};
\node[vertex] at (-2.41421, -1.00000, 3.82843)     {};
\node[vertex] at (-2.41421, -1.00000, -3.82843)     {};
\node[vertex] at (-3.82843, 1.00000, 2.41421)     {};
\node[vertex] at (-3.82843, 1.00000, -2.41421)     {};
\node[vertex] at (-3.82843, 2.41421, 1.00000)     {};
\node[vertex] at (-3.82843, 2.41421, -1.00000)     {};
\node[vertex] at (-3.82843, -1.00000, 2.41421)     {};
\node[vertex] at (-3.82843, -1.00000, -2.41421)     {};
\node[vertex] at (-2.41421, -3.82843, 1.00000)     {};
\node[vertex] at (-3.82843, -2.41421, 1.00000)     {};
\node[vertex] at (-3.82843, -2.41421, -1.00000)     {};
\end{tikzpicture}

%% file: graphics/permBcycle.tex
\begin{tikzpicture}%
[x={(-0.801330cm, 0.247442cm)},
y={(-0.598223cm, -0.331531cm)},
z={(-0.000051cm, 0.910417cm)},
scale=.5,
backr/.style={color=red!30!white, thick},
backb/.style={color=blue!30!white, thick},
edger/.style={color=red!95!black, thick},
edgeb/.style={color=blue!95!black, thick},
facet/.style={fill=red!95!black,fill opacity=0.1},
vertex/.style={inner sep=1pt,rectangle,fill=white,thick,scale=.7}]
%
%
%% This TikZ-picture was produced with Sagemath version 10.3
%% with the command: ._tikz_3d_in_3d and parameters:
%% view = [0.177000000000000, -0.532900000000000, -0.827500000000000]
%% angle = 210.730000000000
%%
%% Coordinate of the vertices:
%%
\coordinate (213) at (2.41421, 1.00000, 3.82843);
\coordinate (123) at (1.00000, 2.41421, 3.82843);
\coordinate (132) at (1.00000, 3.82843, 2.41421);
\coordinate (21b3) at (2.41421, 1.00000, -3.82843);
\coordinate (12b3) at (1.00000, 2.41421, -3.82843);
\coordinate (1b32) at (1.00000, 3.82843, -2.41421);
\coordinate (312) at (2.41421, 3.82843, 1.00000);
\coordinate (b312) at (2.41421, 3.82843, -1.00000);
\coordinate (1b23) at (1.00000, -2.41421, 3.82843);
\coordinate (13b2) at (1.00000, -3.82843, 2.41421);
\coordinate (b213) at (2.41421, -1.00000, 3.82843);
\coordinate (1b2b3) at (1.00000, -2.41421, -3.82843);
\coordinate (1b3b2) at (1.00000, -3.82843, -2.41421);
\coordinate (b21b3) at (2.41421, -1.00000, -3.82843);
\coordinate (231) at (3.82843, 1.00000, 2.41421);
\coordinate (2b31) at (3.82843, 1.00000, -2.41421);
\coordinate (321) at (3.82843, 2.41421, 1.00000);
\coordinate (b321) at (3.82843, 2.41421, -1.00000);
\coordinate (b231) at (3.82843, -1.00000, 2.41421);
\coordinate (b2b31) at (3.82843, -1.00000, -2.41421);
\coordinate (31b2) at (2.41421, -3.82843, 1.00000);
\coordinate (b31b2) at (2.41421, -3.82843, -1.00000);
\coordinate (3b21) at (3.82843, -2.41421, 1.00000);
\coordinate (b3b21) at (3.82843, -2.41421, -1.00000);
\coordinate (2b13) at (-2.41421, 1.00000, 3.82843);
\coordinate (b123) at (-1.00000, 2.41421, 3.82843);
\coordinate (b132) at (-1.00000, 3.82843, 2.41421);
\coordinate (2b1b3) at (-2.41421, 1.00000, -3.82843);
\coordinate (b12b3) at (-1.00000, 2.41421, -3.82843);
\coordinate (b1b32) at (-1.00000, 3.82843, -2.41421);
\coordinate (3b12) at (-2.41421, 3.82843, 1.00000);
\coordinate (b3b12) at (-2.41421, 3.82843, -1.00000);
\coordinate (b1b23) at (-1.00000, -2.41421, 3.82843);
\coordinate (b13b2) at (-1.00000, -3.82843, 2.41421);
\coordinate (b2b13) at (-2.41421, -1.00000, 3.82843);
\coordinate (b1b2b3) at (-1.00000, -2.41421, -3.82843);
\coordinate (b1b3b2) at (-1.00000, -3.82843, -2.41421);
\coordinate (b2b1b3) at (-2.41421, -1.00000, -3.82843);
\coordinate (23b1) at (-3.82843, 1.00000, 2.41421);
\coordinate (2b3b1) at (-3.82843, 1.00000, -2.41421);
\coordinate (32b1) at (-3.82843, 2.41421, 1.00000);
\coordinate (b32b1) at (-3.82843, 2.41421, -1.00000);
\coordinate (b23b1) at (-3.82843, -1.00000, 2.41421);
\coordinate (b2b3b1) at (-3.82843, -1.00000, -2.41421);
\coordinate (3b1b2) at (-2.41421, -3.82843, 1.00000);
\coordinate (b3b1b2) at (-2.41421, -3.82843, -1.00000);
\coordinate (3b2b1) at (-3.82843, -2.41421, 1.00000);
\coordinate (b3b2b1) at (-3.82843, -2.41421, -1.00000);
%%
%%
%% Drawing edges in the back
%%
\draw[backb] (21b3) -- (12b3);
\draw[backr] (21b3) -- (2b31);
\draw[backr] (1b23) -- (13b2);
\draw[backr] (13b2) -- (31b2);
\draw[backr] (b213) -- (b231);
\draw[backr] (1b2b3) -- (1b3b2);
\draw[backb] (1b2b3) -- (b21b3);
\draw[backr] (1b3b2) -- (b31b2);
\draw[backr] (b21b3) -- (b2b31);
\draw[backr] (2b31) -- (b321);
\draw[backr] (b231) -- (3b21);
\draw[backr] (b2b31) -- (b3b21);
\draw[backr] (31b2) -- (b31b2);
\draw[backr] (3b21) -- (b3b21);
\draw[backr] (b1b2b3) -- (b1b3b2);
\draw[backb] (b1b2b3) -- (b2b1b3);
\draw[backr] (b1b3b2) -- (b3b1b2);
\draw[backr] (3b1b2) -- (b3b1b2);
%%
%%
%% Drawing vertices in the back
%%
\node[vertex] at (13b2)     {$13\bar{2}$};
\node[vertex] at (1b2b3)     {$1\bar{2}\bar{3}$};
\node[vertex] at ([xshift=-5pt,yshift=12pt]1b3b2)     {$1\bar{3}\bar{2}$};
\node[vertex] at ([xshift=10pt,yshift=10pt]b1b2b3)     {$\bar{1}\bar{2}\bar{3}$};
\node[vertex] at (b1b3b2)     {$\bar{1}\bar{3}\bar{2}$};
\node[vertex] at ([xshift=10pt,yshift=4pt]21b3)     {$21\bar{3}$};
\node[vertex] at (b21b3)     {$\bar{2}1\bar{3}$};
\node[vertex] at (31b2)     {$31\bar{2}$};
\node[vertex] at (b31b2)     {$\bar{3}1\bar{2}$};
\node[vertex] at (b3b1b2)     {$\bar{3}\bar{1}\bar{2}$};
\node[vertex] at ([xshift=10pt]b231)     {$\bar{2}31$};
\node[vertex] at ([xshift=10pt]2b31)     {$2\bar{3}1$};
\node[vertex] at (b2b31)     {$\bar{2}\bar{3}1$};
\node[vertex] at ([yshift=-9pt]3b21)     {$3\bar{2}1$};
\node[vertex] at ([yshift=10pt]b3b21)     {$\bar{3}\bar{2}1$};
%%
%% Drawing the facets
%%
\fill[facet] (b132) -- (132) -- (123) -- (b123) -- cycle {};
\fill[facet] (b1b32) -- (1b32) -- (12b3) -- (b12b3) -- cycle {};
\fill[facet, fill=blue] (b2b13) -- (2b13) -- (b123) -- (123) -- (213) -- (b213) -- (1b23) -- (b1b23) -- cycle {};
\fill[facet] (b3b12) -- (b1b32) -- (1b32) -- (b312) -- (312) -- (132) -- (b132) -- (3b12) -- cycle {};
\fill[facet] (b23b1) -- (b2b13) -- (2b13) -- (23b1) -- cycle {};
\fill[facet] (b2b3b1) -- (b2b1b3) -- (2b1b3) -- (2b3b1) -- cycle {};
\fill[facet] (321) -- (312) -- (132) -- (123) -- (213) -- (231) -- cycle {};
\fill[facet] (b321) -- (b312) -- (312) -- (321) -- cycle {};
\fill[facet] (32b1) -- (3b12) -- (b132) -- (b123) -- (2b13) -- (23b1) -- cycle {};
\fill[facet] (b32b1) -- (b3b12) -- (b1b32) -- (b12b3) -- (2b1b3) -- (2b3b1) -- cycle {};
\fill[facet] (b32b1) -- (b3b12) -- (3b12) -- (32b1) -- cycle {};
\fill[facet] (3b2b1) -- (b23b1) -- (b2b13) -- (b1b23) -- (b13b2) -- (3b1b2) -- cycle {};
\fill[facet] (b3b2b1) -- (b2b3b1) -- (2b3b1) -- (b32b1) -- (32b1) -- (23b1) -- (b23b1) -- (3b2b1) -- cycle {};
\fill[facet, fill=blue] (b2b1b3) -- (2b1b3) -- (b12b3) -- (12b3) -- (21b3) -- (b21b3) -- (1b2b3) -- (b1b2b3) -- cycle {};
%%
%%
%% Drawing edges in the front
%%
\draw[edgeb] (213) -- (b213);
\draw[edger] (213) -- (231);
\draw[edger] (123) -- (132);
\draw[edgeb] (123) -- (b123);
\draw[edger] (132) -- (312);
\draw[edger] (12b3) -- (1b32);
\draw[edger] (1b32) -- (b312);
\draw[edger] (312) -- (b312);
\draw[edgeb] (1b23) -- (b1b23);
\draw[edger] (231) -- (321);
\draw[edger] (321) -- (b321);
\draw[edgeb] (2b13) -- (b2b13);
\draw[edger] (2b13) -- (23b1);
\draw[edger] (b123) -- (b132);
\draw[edger] (b132) -- (3b12);
\draw[edgeb] (2b1b3) -- (b12b3);
\draw[edger] (2b1b3) -- (2b3b1);
\draw[edger] (b12b3) -- (b1b32);
\draw[edger] (b1b32) -- (b3b12);
\draw[edger] (3b12) -- (b3b12);
\draw[edger] (b1b23) -- (b13b2);
\draw[edger] (b13b2) -- (3b1b2);
\draw[edger] (b2b13) -- (b23b1);
\draw[edger] (b2b1b3) -- (b2b3b1);
\draw[edger] (23b1) -- (32b1);
\draw[edger] (2b3b1) -- (b32b1);
\draw[edger] (32b1) -- (b32b1);
\draw[edger] (b23b1) -- (3b2b1);
\draw[edger] (b2b3b1) -- (b3b2b1);
\draw[edger] (3b2b1) -- (b3b2b1);
%%
%%
%% Drawing the vertices in the front
%%
\node[vertex] at ([xshift=-8pt,yshift=7pt]213)     {$213$};
\node[vertex] at (123)     {$123$};
\node[vertex] at (132)     {$132$}; 
\node[vertex] at ([xshift=-3pt,yshift=-3pt]12b3)     {$12\bar{3}$};
\node[vertex] at (1b32)     {$1\bar{3}2$};
\node[vertex] at (312)     {$312$};
\node[vertex] at (b312)     {$\bar{3}12$};
\node[vertex] at (1b23)     {$1\bar{2}3$};
\node[vertex] at (b213)     {$\bar{2}13$};
\node[vertex] at (231)     {$231$};
\node[vertex] at (321)     {$321$};
\node[vertex] at (b321)     {$\bar{3}21$};
\node[vertex] at (2b13)     {$2\bar{1}3$};
\node[vertex] at (b123)     {$\bar{1}23$}; 
\node[vertex] at (b132)     {$\bar{1}32$};
\node[vertex] at (2b1b3)     {$2\bar{1}\bar{3}$};
\node[vertex] at (b12b3)     {$\bar{1}2\bar{3}$};
\node[vertex] at (b1b32)     {$\bar{1}\bar{3}2$};
\node[vertex] at (3b12)     {$3\bar{1}2$};
\node[vertex] at (b3b12)     {$\bar{3}\bar{1}2$};
\node[vertex] at (b1b23)     {$\bar{1}\bar{2}3$};
\node[vertex] at ([xshift=9pt,yshift=9pt]b13b2)     {$\bar{1}3\bar{2}$};
\node[vertex] at (b2b13)     {$\bar{2}\bar{1}3$};
\node[vertex] at ([xshift=9pt,yshift=-9pt]b2b1b3)     {$\bar{2}\bar{1}\bar{3}$};
\node[vertex] at (23b1)     {$23\bar{1}$};
\node[vertex] at (2b3b1)     {$2\bar{3}\bar{1}$};
\node[vertex] at (32b1)     {$32\bar{1}$};
\node[vertex] at (b32b1)     {$\bar{3}2\bar{1}$};
\node[vertex] at (b23b1)     {$\bar{2}3\bar{1}$};
\node[vertex] at (b2b3b1)     {$\bar{2}\bar{3}\bar{1}$};
\node[vertex] at ([xshift=8pt,yshift=8pt]3b1b2)     {$3\bar{1}\bar{2}$};
\node[vertex] at (3b2b1)     {$3\bar{2}\bar{1}$};
\node[vertex] at (b3b2b1)     {$\bar{3}\bar{2}\bar{1}$};
\end{tikzpicture}